\documentclass{amsart}
\usepackage{amssymb,amsmath,amsthm}
\usepackage{mathrsfs}
\usepackage{hyperref}
\usepackage{mathtools}
\usepackage{color}
\usepackage{verbatim}
\renewcommand{\[}{\begin{equation}\begin{aligned}}
\renewcommand{\]}{\end{aligned} \end{equation}}

\newtheorem{thm}{Theorem}
\newtheorem{prop}[thm]{Proposition}
\newtheorem{lemma}[thm]{Lemma}
\newtheorem{claim}[thm]{Claim}

\theoremstyle{remark}
\newtheorem{remark}[thm]{Remark}

\theoremstyle{definition}
\newtheorem{definition}[thm]{Definition}

\newcommand{\CC}{\mathbb{C}}
\newcommand{\RR}{\mathbb{R}}

\newcommand{\eps}{\varepsilon}

\newcommand{\cM}{\mathcal{M}}
\newcommand{\cD}{\mathcal{D}}
\newcommand{\bH}{\mathbf{H}}
\newcommand{\bx}{\mathbf{x}}

\newcommand{\rd}{\mathrm{d}}

\newcommand{\bOh}{\mathbf{0}}

\numberwithin{equation}{section}
\numberwithin{thm}{section}

\def\rd{\mathrm{d}}

\begin{document}

\title[Neck pinches along the Lagrangian mean curvature flow]{Neck pinches along the Lagrangian mean curvature flow of surfaces}

\author{Jason D. Lotay}
\address{Mathematical Institute, University of Oxford, Oxford OX2 6GG, United Kingdom.}  \email{jason.lotay@maths.ox.ac.uk}

\author{Felix Schulze}
\address{Mathematics Institute,
  University of Warwick,
	Coventry CV4 7AL,
	United Kingdom }
\email{felix.schulze@warwick.ac.uk}
      
\author{G\'abor Sz\'ekelyhidi}
\address{Department of Mathematics, Northwestern University, Evanston IL 60208, USA}
\email{gaborsz@northwestern.edu}

\date{\today}

\begin{abstract}
  Let $L_t$ be a zero Maslov, rational Lagrangian mean curvature flow
  in a compact
  Calabi--Yau surface, and suppose that at the first
  singular time a tangent flow is given by the static union of two transverse
  planes. We show that in this case the tangent flow is unique, and
  that the flow can be continued past the singularity as an immersed,  smooth, zero Maslov, rational Lagrangian mean curvature
  flow. Furthermore, if $L_0$ is a sphere that is stable in the sense
  of Thomas--Yau, then such a singularity cannot form.
\end{abstract}

\maketitle

\section{Introduction}
The question of the existence of special Lagrangian submanifolds is an
important problem in complex and symplectic geometry. Special Lagrangians play a central role in the
Strominger--Yau--Zaslow conjecture~\cite{SYZ96} on mirror symmetry, and
are of interest in the variational problem of finding area-minimizing Lagrangians, studied extensively by
Schoen--Wolfson~\cite{SW01}. Smoczyk~\cite{Smo96} showed that the mean
curvature flow preserves the class of Lagrangian submanifolds in
Calabi--Yau manifolds, and so a natural expectation is that a suitable Lagrangian
 can be deformed into a special Lagrangian 
using the flow. The Thomas--Yau conjecture~\cite{TY02},
motivated by mirror symmetry~\cite{Thomas01}, predicts that this is
indeed the case, assuming that the initial Lagrangian 
satisfies a certain stability condition. More recently
Joyce~\cite{Joyce15} formulated a detailed conjectural picture 
relating singularity formation along the Lagrangian mean curvature
flow to Bridgeland stability conditions on Fukaya categories.

It was shown by Neves~\cite{Neves:singularities} that singularities
are, in a sense, unavoidable along the Lagrangian mean curvature flow, even if the
initial Lagrangian is a small Hamiltonian perturbation of a special
Lagrangian. At the same time, Neves~\cite{Neves:zero-maslov} shows
that for the flow of zero Maslov Lagrangians any tangent
flow at a singular point is a union of special Lagrangian cones. This
means that Type I singularities -- which are typically easier to
analyse -- do not exist. In this paper we study the simplest kind of
singularities, called neck pinches in \cite[Conjecture
3.16]{Joyce15}, in the two-dimensional case. Our main
result is the following, which we state in the setting of a compact ambient Calabi--Yau surface, though it also works in $\CC^2$.  Note that here, and throughout, we allow our Lagrangians to be immersed, which is important in the context of Lagrangian mean curvature flow.

\begin{thm}\label{thm:intro1}
  Let $X$ be a  compact
   Calabi--Yau surface, and $L\subset X$ a zero
  Maslov, rational Lagrangian. Let $L_t$ be the mean curvature
  flow starting from $L$ for $t\in [0,T)$, where $T$ is the first finite
  singular time. Let $(\bx_T, T)$ be a singular point, and suppose that
  a tangent flow at $(\bx_T, T)$ is given by the transverse union of two
  multiplicity one planes. The tangent flow at $(\bx_T, T)$ is then
  unique.
\end{thm}

  In Theorem~\ref{thm:intro1} the assumption is that one
tangent flow is given by a union of multiplicity one transverse planes $P_1\cup P_2$,
with corresponding Lagrangian angles $\theta_1, \theta_2$. We note
here that by Neves~\cite[Corollary 4.3]{Neves:survey} the flow cannot
form a singularity unless
$\theta_1=\theta_2$. Therefore throughout the article we
will only be concerned with the case when $P_1$ and $P_2$ have the
same Lagrangian angle.

The uniqueness of tangent flows is a fundamental problem for analysing
the singularities of mean
curvature flow, and there have been several important results in this
direction recently~\cite{Schulze14, CM:uniqueness, CS21}. A major new
difficulty in Theorem~\ref{thm:intro1} is that it is the first example
of uniqueness for a tangent flow that is singular. The proof crucially
exploits several aspects of the Lagrangian setting, and does not apply
in the general setting of mean curvature flow.

Theorem~\ref{thm:intro1} allows us to analyse the behavior of the flow 
at the singularity. First, we have the following, showing that the
flow can be continued past the singular time if all singularities at time $T$ are modelled on two transverse, multiplicity one planes. Recall that the grading of a zero Maslov Lagrangian corresponds to a global choice of function representing the Lagrangian angle, see Definition \ref{dfn:ZM}.

\begin{thm}\label{thm:intro2}
  Suppose that $X, L$ are as in Theorem~\ref{thm:intro1} and assume that at each singular point $(\mathbf{x}, T)$ a tangent flow is a static union of two multiplicity one, transverse planes. Then $L_t$ converges to an immersed Lagrangian $C^1$-submanifold $L_T$ in the sense of
    currents as $t\to T$, and the flow can be restarted as a smooth, zero Maslov, rational Lagrangian mean curvature flow with initial
    condition $L_T$. Furthermore, the extended flow is smooth (together with its grading) through the singular time, away from the singular points. 
  \end{thm}

A slight extension of the ideas involved in proving uniqueness of the
tangent flow also allows us to show that if the tangent flow is given
by the union of two transverse planes, then close to the singularity the flow looks like the two transverse planes, desingularized by a Lawlor neck which ``pinches off'': see Theorem~\ref{thm:Lawlor}. This has
the following consequence. 

\begin{thm}\label{thm:intro3}
  Suppose that $X,L$ are as in Theorem~\ref{thm:intro1}. 
  %and in addition $L$ is a sphere. Assume that the flow $L_t$ is  embedded for $t < T$. Then 
  For $t < T$ sufficiently close to the
  singular time we can write $L_t$ as a graded self-connected sum of an immersed Lagrangian $M$ at a self-intersection point.  
  
If $M$ is not connected, then we can write it is a graded connected sum $M = M_1
  \# M_2$ and the following holds: 
  \begin{equation}\label{eq:stab1}
    \mathrm{vol}(L) > \left| \int_{M_1} \Omega \right| +
    \left|\int_{M_2} \Omega\right|, 
\end{equation}
where $\Omega$ is the holomorphic volume form on $X$. 
If in addition $L$ is almost calibrated, then we also have
\begin{equation}\label{eq:stab2}
  \phi(M_1), \phi(M_2) \subset (\inf_L \theta, \sup_L
    \theta), 
 \end{equation}
where $\phi(M_i)$ is the ``cohomological'' Lagrangian angle of
$M_i$ defined in \eqref{eq:phiMdefn}.
(See Section~\ref{sec:compactCY} for detailed definitions.) 
\end{thm}

This result provides some evidence for Thomas--Yau's Conjecture 7.3 in 
\cite{TY02}. Indeed, their conjecture states that if the flow has a
finite time singularity, then $L$ can be
decomposed into a graded connected sum $M_1\# M_2$ satisfying the
conditions in \eqref{eq:stab1}.  Our result shows that this is one of the possible scenarios when the tangent flow at the first singular time is given by two transverse planes.  In particular, the decomposition as a graded connect sum is guaranteed if $L$ is a sphere.  
%and the flow remains embedded until the first singular time. %Our result shows that this holds as long as the flow remains embedded until the first singular time, and the singularity has a tangent flow given by two transverse planes. 
Note that Joyce's conjectural picture~\cite{Joyce15} predicts that other
singularities could still form, notably those with tangent flows
given by two static planes meeting along a line. It is an
important problem to understand what we can say about the flow in the
presence of such singularities and some progress towards this was made
in the authors' previous work~\cite{LSS22}. An optimistic expectation is
that for a generic initial surface, the only tangent flows that appear
at singularities are of these two types, i.e.~two multiplicity one
planes meeting either at a point or along a line.

\subsection{Outline}
To conclude this introduction we give a brief outline of the contents
of the paper, along with some of the main ingredients of the
proofs. In most of the paper we will consider the flow in $\CC^2$
and we will only discuss the necessary changes in the case of a compact
ambient space in Section~\ref{sec:compactCY}.

The technical heart of our results is an analysis of rescaled Lagrangian mean
curvature flows $M_\tau \subset\CC^2$ which are close to the union
$V = P_1\cup P_2$ of two transverse
planes on a time interval $t\in [0,1]$, say. We introduce a distance
function $D_V(M_{\tau})$ from $V$ to $M_{\tau}$, which incorporates both an
$L^2$-type distance, as well as the difference in Lagrangian angles. (See
Definition~\ref{dfn:IVDV} for the precise definition.) Although the
non-compactness introduces some technical difficulties, the main
difficulty when compared to earlier studies is the presence of the
singularity at the origin. The first step for dealing with this is to
observe that the function $|zw|^2$ satisfies a useful differential
inequality along the flow: here $z, w$ are complex coordinates, for a
suitable hyperk\"ahler-rotated complex structure, such that $V =
\{zw=0\}$. The differential inequality is exploited in
Lemma~\ref{lem:IVest}, allowing us to convert bounds on
$D_V(M_\tau)$ to pointwise distance bounds at a later
time. 

In general the knowledge that $M_\tau$ is close to $V$ in the Hausdorff
sense does not imply graphicality of $M_\tau$ over $V$, even on regions
away from the origin, because of possible multiplicity. 
In Proposition~\ref{prop:graphical} we show that good graphicality of
$M_\tau$ over $V$ on a fixed annulus $B_2\setminus B_1$ can be propagated
out to larger annular regions $B_R\setminus B_r$ at later times, in the presence of
our pointwise distance bounds. This graphicality estimate is then used in
Proposition~\ref{prop:Aest} to derive a crucial estimate
$$ |\mathcal{A}(M_\tau)| \leq D_V(M_{\tau-1})^{1+\alpha_1} $$
for the excess $\mathcal{A}$ (defined in \eqref{eq:AMdefn})  in terms
of the distance, where $\alpha_1>0$.

Our next task, in Sections~\ref{sec:heateqnlimits} and
\ref{sec:threeann}, is to derive a three-annulus type lemma for the
distance function $D_V$. The main result is
Proposition~\ref{prop:3ann2}, and the proof relies on an analysis of
solutions of the drift heat equation on a plane with some
mild singularities at the origin, together with the non-concentration estimates in
Lemma~\ref{lem:IVest}.

The technical heart of the paper is the proof of the decay estimate,
Proposition~\ref{prop:decay10}, for the distance function. Given a three-annulus lemma as in
Proposition~\ref{prop:3ann2}, the usual strategy for controlling
the flow $M_\tau$ is to show that at each scale the flow must decay
towards the ``best fit'' cone of the form $V' = P_1' \cup P_2'$. The
difficulty is that our estimates only apply when $V'$ is special
Lagrangian, i.e.~the planes $P_1'$ and $P_2'$ have the same Lagrangian
angle. Since in general these angles might be different, the situation
is somewhat reminiscent of the case of non-integrable tangent flows,
which is typically dealt with using the Lojasiewicz--Simon
inequality~\cite{Simon.asympt}. Our approach is quite different from this,
and can be thought of as a quantitative version of Neves's
result~\cite[Corollary 4.3]{Neves:survey} stating that at a
singularity the tangent flow cannot be the union of two planes with
different Lagrangian angles. These considerations lead to the alternative (ii) in
Proposition~\ref{prop:decay10}. 

We give the proofs of the main applications in
Section \ref{sec:neckpinch}.  Given the decay estimate in Proposition~\ref{prop:decay10}, the
uniqueness result, Theorem~\ref{thm:intro1}, follows standard
arguments. The proof of the existence of a $C^1$ limiting surface
$L_T$ in Theorem~\ref{thm:intro2} is similar to
\cite[Corollary 1.2]{CS21}. We can then restart the flow
using the approach of Wang~\cite{Wang04}. 

The main ingredient for proving Theorem~\ref{thm:intro3} is to show that one can find
small Lawlor necks, i.e.~surfaces of the form $\{zw=\pm\epsilon\}$,
in the presence of a tangent flow $V$ given by two transverse planes. This
relies on showing that if between different scales the flow stays
close to possibly moving pairs of planes, then these planes have to
stay very close to each other, similarly to how the uniqueness of the
tangent flow is proved. A related result was shown by
Edelen~\cite[Theorem 13.1]{Edelen21} in the context of minimal
hypersurfaces.

Finally in Section~\ref{sec:compactCY} we will discuss the changes needed
when working in a compact ambient space $X$ instead of in $\CC^2$.
We follow the approach found
for instance in White~\cite[Section 4]{White.regularity},
isometrically embedding $X\subset \mathbb{R}^N$, and writing the
(rescaled) mean curvature flow in $X$ as a (rescaled) mean curvature
flow in $\mathbb{R}^N$ with an additional forcing term. The
quantities, such as $|zw|$, that we used in $\CC^2$ can be defined by
projecting to the tangent space $T_pX$ at the point $p$ where the
singularity forms. Along the
rescaled flow this projection as well as the forcing term introduces
additional errors when compared to the calculations in $\CC^2$,
however these errors decay exponentially fast and so the geometric
conclusions still hold. 

\subsection{Acknowledgements} We thank Dominic Joyce and Yang Li for
their interest in this work and helpful comments. 
JDL and FS were partially supported by a Leverhulme Trust Research Project Grant RPG-2016-174.  GSz was supported in part by NSF grant DMS-2203218.

\section{Preliminaries}\label{s:prelim}
In this section we introduce various key definitions and notation that we shall require throughout the article. 

\subsection{Lagrangians in \texorpdfstring{$\CC^2$}{C2}}  We first
recall some basic definitions concerning Lagrangian submanifolds in
$\CC^2$. 

\begin{definition} %[Lagrangian angle]
\label{dfn:ZM}
An oriented Lagrangian submanifold $L$  in $\CC^2$ is \emph{zero Maslov} if there exists a function $\theta$ on $L$ (called the \emph{Lagrangian angle}) so that
\begin{equation*}
\bH=J\nabla\theta,
\end{equation*}
where $\bH$ is the mean curvature vector of $L$ and $J$ is the complex structure on $\CC^2$. The choice of function $\theta$ is called a \emph{grading} of $L$. We further say that $L$ is \emph{almost calibrated} if $\theta$ can be chosen so that, for some $\epsilon>0$, 
\begin{equation*}
\sup\theta-\inf\theta\leq \pi-\epsilon.
\end{equation*}
%for some $\epsilon>0$.
\end{definition}

\begin{definition} %[Exact and rational Lagrangians]
\label{dfn:exact}
An oriented Lagrangian $L$ in $\CC^2$ is \emph{exact} if there exists a function $\beta$ on $L$ so that
\begin{equation*}
J\bx^{\perp}=\nabla\beta,
\end{equation*}
where $\bx^{\perp}$ is the normal projection of the position vector $\bx\in\CC^2$.  Equivalently,
\[
\rd\beta=\lambda|_L,
\]
where $\lambda$ is the Liouville form on $\CC^2$, which is a 1-form on
$\CC^2$ so that $\frac{1}{2}\lambda$ is a primitive for the K\"ahler
form $\omega$ on $\CC^2$.
The Lagrangian $L$ is \emph{rational}  if the set $\lambda(H_1(L,
\mathbb{Z}))$ is discrete in $\mathbb{R}$.  Note that exact Lagrangians are rational. 
\end{definition}

\subsection{Lagrangian mean curvature flow}
In most of this article we will consider a smooth, zero Maslov solution to
Lagrangian mean curvature flow (LMCF) 
$$[0, T)\ni t\mapsto L_t\subset\CC^2$$
which evolves with normal speed given by $\bH$. See
Section~\ref{sec:compactCY} for the setting of a compact ambient
space.

Throughout we will assume that the Lagrangian angle $\theta$ of $L_t$ along the
flow is uniformly bounded: $|\theta| < C_0$. In addition we assume that
 $L_t$ has uniformly
 bounded area ratios, i.e.~there exists $C_1>0$ such that
 $$\sup_{\bx,t}
   \mathcal{H}^2(L_t\cap B(\bx,r)) \leq C_1 r^2 \text{ for all } r>0,$$
 where
 $B(\bx,r)$ is the Euclidean ball of radius $r$ about $\bx\in\CC^2$. 
 We call
 $$\cM:=\{L_t\times \{t\}\, |\, t \in [0,T)\} \subset \CC^2\times \RR$$
 the spacetime track of the flow, and write $\cM(t)= L_t\, .$

It will  be useful to perform parabolic rescalings of our flows, so we shall introduce the following notation.
\begin{definition} %[Parabolic rescaling]
\label{dfn:parabolic}
For $\lambda>0$ we shall denote the \emph{parabolic rescaling}
$$\cD_{\lambda}:\CC^2\times \RR \to \CC^2\times \RR, (\bx, t) \mapsto (\lambda \bx, \lambda^2t)\, .$$
Note that for a (Lagrangian) mean curvature flow $\cM$, it holds that
$\cD_{\lambda}\cM$
is again a (Lagrangian) mean curvature flow.
\end{definition}

We recall Huisken's monotonicity formula~\cite{Huisken}:
\begin{equation}\label{eq:Huisken}
\begin{split}
 \frac{d}{dt} \int_{L_t} f \rho_{\bx_0, t_0}\, d\mathcal{H}^2 &= 
    \int_{L_t} (\partial_t f - \Delta f) \rho_{\bx_0, t_0}\,
    d\mathcal{H}^2\\
    &\qquad  - \int_{L_t} f\,  \bigg|\, \bH - \frac{(\bx-\bx_0)^\perp}{2 (t_0-t)}\bigg|^2 \rho_{\bx_0, t_0}\,
    d\mathcal{H}^2\, ,  
    \end{split}
\end{equation}
for $t<t_0$, where $f$ is a function on $L_t$ with polynomial growth (locally uniform in $t$), and
$$ \rho_{\bx_0, t_0}(\bx,t) = (4\pi(t_0-t))^{-1} \exp\left( -\frac{
      |\bx-\bx_0|^2}{4(t_0-t)}\right) $$
is the backwards heat kernel. The  \emph{entropy} $\lambda(L)$ defined by Colding-Minicozzi
\cite{ColdingMinicozzi:generic} is given by
 $$
  \lambda(L)=\sup_{\mathbf{x}_0\in\CC^2,\,r>0}\frac{1}{4\pi
    r}\int_L
  e^{-\frac{|\mathbf{x}-\mathbf{x}_0|^2}{4r}}\rd\mathcal{H}^2. 
 $$
By virtue of Huisken's monotonicity formula, $t\mapsto \lambda(L_t)$ is non-increasing along any $2$-dimensional mean curvature flow in $\CC^2$.

We will be studying the behaviour of the flow close to a singularity at $(\mathbf{x}_0, T)$. It is convenient to shift the flow in space-time such that $(\mathbf{x}_0, T)$ is the origin $(\mathbf{0},0)$, i.e.~we consider instead the flow
$$ \tilde \cM : =\cM - (\mathbf{x}_0, T)\, ,$$
defined for $t\in [-T,0)$. For ease of notation we will drop the
tilde in the following.

A tangent flow at $(\mathbf{0}, 0)$ is defined to be any weak limit of
a sequence of rescalings $\mathcal{D}_{\lambda_k}\mathcal{M}$, with
$\lambda_k \to \infty$. According to the structure theorem due to
Neves~\cite[Theorem A]{Neves:zero-maslov}, in our setting of a zero
Maslov flow with bounded Lagrangian angle in two dimensions, the tangent
flows are all given by unions of Lagrangian planes with
multiplicities. In addition the Lagrangian angle $\theta$ along the
sequence of rescalings converges in a suitable sense to the angles of
the planes. We recall that the assumptions of Theorem~\ref{thm:intro1}  mean that one
tangent flow at $(\mathbf{0}, 0)$ is given by a union of multiplicity one transverse planes $P_1\cup P_2$ with the same Lagrangian angle.

It turns out to be helpful to consider a further rescaling, which turns self-similarly shrinking solutions into static points of the flow.
\begin{definition} %[Rescaled flow]
\label{dfn:rescaled.flow}
The \emph{rescaled flow} is 
$$ [- \log(T), +\infty) \ni \tau \mapsto M_\tau := e^\frac{\tau}{2} \cM(-e^{-\tau})= e^\frac{\tau}{2} L_{-e^{-\tau}}, $$
which evolves with normal speed
\begin{equation}\label{eq:rescaled.flow}
 \bH + \frac{\bx^\perp}{2}\, .
\end{equation}
\end{definition}
In terms of the rescaled flow the tangent flows of $\mathcal{M}$ can be studied by
taking limits of sequences $M_{\tau_k}$ as $\tau_k\to \infty$. 

\subsection{Set-up}\label{sec:setup} 
For the majority of this article we will consider a rescaled Lagrangian mean curvature flow
$M_\tau$ in $\mathbb{C}^2$, for $\tau \in [T_0, T_1]$, close in a
suitable sense to the static flow of the transverse union of two
planes. The flow is assumed to have uniformly bounded Lagrangian angle
and area ratios as above. Two additional conditions will
play an important role, the first of which involves the following key quantity.

\begin{definition}\label{dfn:excess}
Let $M$ be a graded Lagrangian in $\mathbb{C}^2$.  We define the \emph{excess} of $M$ to be 
    \begin{equation}\label{eq:AMdefn}
  \mathcal{A}(M) = \int_M e^{-|\mathbf{x}|^2/4} - 2\int_{\mathbb{R}^2} e^{-|\mathbf{x}|^2/4} +
  \inf_{\theta_0} \int_M |\theta - \theta_0|^2 e^{-|\mathbf{x}|^2/4}.
\end{equation}
Note that we allow
  $\mathcal{A}$ to be negative.  We also observe that $\mathcal{A}$ is monotonically decreasing along the rescaled   flow by Huisken's montonicity formula, since it is an infimum of a family of decreasing functions.
\end{definition}

In the conditions below and throughout we let $B_R(\bx)$ denote the ball of radius $R$ about $\bx\in\mathbb{C}^2$ and let $B_R=B_R(\bOh)$.

\smallskip

\paragraph{\bf Condition (\ddag)}
 We assume that there is a small $c_0>0$ (to be chosen later) so that 
 \begin{equation}
 \tag{\ddag} 
 \label{cond:ddag}
 \mathcal{A}(M_{T_0}) - \mathcal{A}(M_{T_1}) < c_0.
 \end{equation}

\smallskip 

\paragraph{\bf Condition (\textasteriskcentered{})}
We assume that 
$M$ is exact in $B_1$, i.e.~the Liouville form
$\lambda$ satisfies
\begin{equation} \tag{$\ast$} \label{cond:ast}
 \int_\gamma \lambda = 0
\end{equation}
for every closed  loop
$\gamma\subset M\cap B_1$. In addition we assume $M$ is connected in $B_1$,
i.e.~$M\cap B_1$ cannot be written as the union of two nonempty submanifolds.

\smallskip

\noindent We shall see later that  Conditions \eqref{cond:ddag} and \eqref{cond:ast} hold along the rescaled flow when our original Lagrangian mean curvature flow develops a finite time singularity with tangent flow given by a special Lagrangian union of two transverse planes: see Lemma~\ref{lem:Conditions}.

\section{Distance to planes}\label{sec:disttoplanes}

We will now suppose that the flow $M_\tau$, defined for  $\tau\in[T_0,T_1]$, is close to a special
Lagrangian transverse union 
$V_0 = P_{0,1}\cup P_{0,2}$ of
planes through $\bOh$ in
$\mathbb{C}^2$, in a way that will be specified later. 

\subsection{Nearby pairs of planes and complex lines}
 It will be useful to restrict the set of pairs of planes we are considering to be those sufficiently close to $V_0$ in the following sense.

\begin{definition}\label{dfn:V}
Fix a small number $c_1 > 0$ and let
$\mathcal{V}$ be the space of all special Lagrangian unions $V = P_1
\cup P_2$ of two
planes through the origin, such that the angles between $P_i$ and
$P_{0,i}$ are smaller than $c_1$, i.e.~the cones in $\mathcal{V}$ are
small deformations of $V_0$. We further assume that $c_1$ is chosen sufficiently small  so that elements $P_1\cup P_2$ in $\mathcal{V}$ are
bounded away from the multiplicity two planes, i.e.~there exists $c_2>0$ so the angle between
$P_1$ and $P_2$ is at least $c_2$. Below we may further shrink $c_1$,
however it will always be a constant that depends on the choice of
$V_0$ only. 

For any $V\in \mathcal{V}$ we denote the Lagrangian angle of $V$ by
$\theta_V$.  We also let $\mathcal{V}'\subset\mathcal{V}$ be defined in  the
same way as $\mathcal{V}$ but using the constant $c_1/2$ instead of $c_1$.
\end{definition}

 We can choose a
hyperk\"ahler rotation of the complex structure of $\mathbb{C}^2$
depending on $\theta_V$,  and
complex coordinates $z,w$, such that $V$ is defined by $zw=0$. We have
the following basic observation, where the notion of Lagrangian and
Lagrangian angle refers to the original Calabi--Yau structure on
$\mathbb{C}^2$. 

\begin{lemma}\label{lem:nablazw}
  There is a constant $C > 0$, independent of $V\in \mathcal{V}$, such that on any oriented
  Lagrangian $2$-plane $P$ in $\mathbb{C}^2$ with Lagrangian angle $\theta$ we have
  $$ |\nabla z\cdot \nabla w| \leq C|\theta - \theta_V|. $$
\end{lemma}
\begin{proof}
  Let $\mathcal{L}$ denote the space of oriented
  Lagrangian 2-planes in $\mathbb{C}^2$, and let $\theta :
  \mathcal{L}\to S^1$ denote the Lagrangian angle function. Then $\rd\theta$ is
  nowhere vanishing. 

Let $P\in\mathcal{L}$ with $\theta=\theta_V$. After the
  hyperk\"ahler rotation which makes $V$ a pair of complex lines, $P$ also becomes a complex line on which $z, w$ and $zw$
  are all holomorphic, and in particular harmonic. We therefore have
$$ 2\nabla
z\cdot \nabla w = \Delta(zw) - z \Delta (w) - w \Delta(z) = 0.$$
Thus, $\nabla z\cdot\nabla w$ vanishes on the
  zero set of $\theta$. 
  
  Since $\rd\theta$ is nowhere vanishing, i.e.~$\theta-\theta_V$
  vanishes to order 1 along its zero set, the result follows.  
\end{proof}

We also record the following evolution equations.

\begin{lemma}\label{lem:evolzw} We have
\begin{equation}\label{eq:zwheateqn}
    (\partial_t - \Delta) |zw|^2 \leq 4 |zw| |\nabla z \cdot \nabla w|\, .
  \end{equation}
  Furthermore, for $\delta>0$ consider $u_\delta: = \sqrt{ \delta^2 + |\theta -\theta_0|^2}$ and $v_\delta := \sqrt{ \delta^2 + |zw|^2}$. Then
  \begin{equation}\label{eq:uvdelatheateqn}
(\partial_t - \Delta) u_\delta \leq 0\quad \text{and}\quad    (\partial_t - \Delta) v_\delta \leq 2 |\nabla z \cdot \nabla w| \, .
  \end{equation}
\end{lemma}
\begin{proof} We have that $z,w$ satisfy the heat equation, so 
  $$(\partial_t - \Delta) zw = - 2 \nabla z\cdot \nabla w \, .$$
 Thus we can compute
 \begin{equation}\label{eq:evolzw}
\begin{split}
 (\partial_t - \Delta)|zw|^2 &=  (\partial_t - \Delta) (\overline{zw}\, zw)  \\
 &= -2 (\overline{\nabla z\cdot \nabla w}\, zw + \overline{zw}\, \nabla z\cdot \nabla w) - 2 \overline{\nabla (zw)} \cdot \nabla (zw)\\ 
 & \leq 4 |zw| |\nabla z\cdot \nabla w|\, .   
 \end{split}
\end{equation}
For $v_\delta$ we compute using \eqref{eq:evolzw} and Kato's inequality
\begin{equation*}
\begin{split}
 (\partial_t - \Delta)v_\delta &=  \frac{1}{2v_\delta}(\partial_t - \Delta)|zw|^2 + \frac{1}{4v_\delta^3} |\nabla |zw|^2|^2  \\
  &=  -\frac{1}{v_\delta}\left( (\overline{\nabla z\cdot \nabla w}\, zw + \overline{zw}\, \nabla z\cdot \nabla w) + \overline{\nabla (zw)} \cdot \nabla (zw)\right)\\
  &\quad + \frac{|zw|^2}{v^3_\delta}|\nabla |zw||^2\\
  &\leq -\frac{1}{v_\delta}\left( (\overline{\nabla z\cdot \nabla w}\, zw + \overline{zw}\, \nabla z\cdot \nabla w)\right) \leq  2 |\nabla z \cdot \nabla w|\, .
 \end{split}
\end{equation*}
The computation for $u_\delta$ is analogous.
\end{proof}
\subsection{Distance to planes and monotonicity}

We now introduce two notions of distance to a pair of planes $V\in\mathcal{V}$.

\begin{definition}\label{dfn:IVDV}
Given any $V\in\mathcal{V}$, let $d_V$ denote the distance function in $\mathbb{C}^2$ to $V$ and define the 
$L^2$-distance $I_V(M_\tau)$ of $M_\tau$ to $V$ by
$$ I_V(M_\tau)^2 = \int_{M_\tau}
  \Big( |\mathbf{x}|^2 d_V^2  + |\theta - \theta_V|^2 \Big)\, e^{-|\mathbf{x}|^2/4}. $$
 It is convenient to introduce a variant $D_V$ of this, which
also encodes graphicality. Suppose that on the
annulus $B_2\setminus B_1$ the surface $M_\tau$ is
the graph of a (vector-valued) function $u$ over $V$ with $|u|,
|\nabla u|\leq
c_1$, for a small $c_1$ as in Definition \ref{dfn:V}. Then we
set $D_V(M_\tau) = I_V(M_\tau)$. 
If $M_\tau$ is not such a small graph over $V$ on the annulus,
then we let $D_V(M_\tau) = \infty$.
\end{definition}

Lemma~\ref{lem:evolzw} leads to the following properties of $I_V$.

\begin{lemma}\label{lem:IVest}
  The $L^2$-distance to $V\in\mathcal{V}$ has the following properties.
  \begin{enumerate}
  \item There is $C>0$, depending only on $\mathcal{V}$, so that 
  $$
  I_V(M_{\tau+s}) \leq CI_V(M_\tau)\quad\text{for any $s\in [0,1]$.}
  $$
  \item 
  For any $0<\delta \leq s\leq1$ there are $p > 1$ and $C > 0$ (both depending on $\delta$), such that on $M_{\tau+s}$
  we have 
  $$ |\mathbf{x}|^2 d_V^2 + |\theta - \theta_V|^2 \leq C e^{\frac{|\mathbf{x}|^2}{4p}}
    I_V(M_\tau)^2. $$
  \item There is a $C > 0$ satisfying the following. For any $\gamma >
    0$ there is a compact subset $K_\gamma\subset \mathbb{C}^2\setminus\{\mathbf{0}\}$ 
    such that if
    $$ \int_{K_\gamma\cap M_{\tau+1}} (|\mathbf{x}|^2 d_V^2 + |\theta - \theta_V|^2)
      e^{-|\mathbf{x}|^2/4} \leq \epsilon^2, $$
    then
    $$ I_V(M_{\tau+1}) \leq C\big(\epsilon + \gamma I_V(M_\tau)\big). $$
  \end{enumerate}
\end{lemma}
\begin{proof} 
  As in Lemma~\ref{lem:nablazw} we choose complex coordinates $z,w$
  for a hyperk\"ahler-rotated complex structure on $\mathbb{C}^2$ such
  that $V$ is given by $zw=0$.
    It follows from Lemma \ref{lem:nablazw} and \eqref{eq:evolzw} that
\begin{equation*}
\begin{split}
 (\partial_t - \Delta)|zw|^2 & \leq 4 |zw| |\nabla z\cdot \nabla w|\\   
 &\leq C|zw||\theta -\theta_V| \\
 &\leq C( |zw|^2 +
    |\theta - \theta_V|^2),
\end{split}
\end{equation*}
    and so
  \begin{equation}\label{eq:zwthetaheateqn}
    (\partial_t - \Delta) (|zw|^2 + |\theta - \theta_V|^2) \leq C(
    |zw|^2 + |\theta -\theta_V|^2).
  \end{equation}
The monotonicity formula along the unrescaled flow then implies that for $-T \leq t_1 <t_2 <0$
$$ \int_{L_{t_2}} (|zw|^2 + |\theta - \theta_V|^2 ) \rho_{\mathbf{0},0} \leq e^{C(t_2-t_1)}  \int_{L_{t_1}} (|zw|^2 + |\theta - \theta_V|^2 ) \rho_{\mathbf{0},0}\, .$$
Rescaling the flow parabolically such that $\tau$ corresponds to $t_1=-1$ (i.e.~translating $\tau$ to $0$ for the rescaled flow) this implies that for the rescaled flow
$$ \int_{M_{\tau+s}} (e^{-2s} |zw|^2 + |\theta - \theta_V|^2 ) e^{-|\mathbf{x}|^2/4} \leq e^{C}  \int_{M_\tau} (|zw|^2 + |\theta - \theta_V|^2 ) e^{-|\mathbf{x}|^2/4}\, .$$

   To deduce property (1) we claim that $|{\bf x}|d_V$ is uniformly equivalent to $|zw|$, with a
constant depending on $c_2$ in the definition of
$\mathcal{V}$. Indeed, the distance $d_V$ is uniformly equivalent to
$\min\{|z|,|w|\}$, while $|{\bf x}|$ is uniformly equivalent to
$\max\{|z|, |w|\}$, and $\min\{|z|, |w|\}\max\{|z|,|w|\} = |zw|$.  
  
 Let $s>0$.  A sharper estimate as in property (1) is obtained by using
  Ecker's log-Sobolev inequality~\cite[Theorem 3.4]{Ecker.logSobolev} along the rescaled flow
  $M_\tau$. We first  note that, similar to above, we can use \eqref{eq:uvdelatheateqn} and Lemma \ref{lem:nablazw} to estimate
  $$  (\partial_t - \Delta) (u_\delta + v_\delta ) \leq C(
    u_\delta+ v_\delta)\, .$$  
  The results in~\cite{Ecker.logSobolev} then imply that for some $\hat p >2$ depending on
  $s$, we have
  \begin{equation*}
 \left(\int_{M_{\tau+s/2}} ( u_\delta+ v_\delta)^{\hat p}
      e^{-|\mathbf{x}|^2/4} \right)^{1/\hat p} \leq C \left(\int_{M_{\tau}} ( u_\delta+ v_\delta)^2  e^{-|\mathbf{x}|^2/4}\right)^\frac{1}{2}
    \end{equation*}
  for $C$ depending on $\mathcal{V}$ and $\delta_0$. Letting $\delta \searrow 0$ then yields for $f = |zw|^2 + |\theta - \theta_V|^2$ and $p = \hat p/2>1$ that
  \begin{equation} \label{eq:lS10}
\left(\int_{M_{\tau+s/2}}f^{p}
      e^{-|\mathbf{x}|^2/4} \right)^{1/p} \leq C \int_{M_{\tau}}f  
      e^{-|\mathbf{x}|^2/4} \leq C I_V(M_\tau)^2\, .
\end{equation}

  Let us now consider the mean curvature flow $L_t$, with initial
  condition $L_{-1} = M_{\tau+s/2}$. Consider any $(\mathbf{x}_0, t_0)$
  with $t_0\in (-1,0)$ and let $\rho_{\mathbf{x}_0, t_0}$ be the backwards heat kernel centered at
  $(\mathbf{x}_0, t_0)$. Using the monotonicity formula
  again for $f^p$ with $f = |zw|^2 + |\theta - \theta_V|^2$, together with \eqref{eq:zwthetaheateqn}, we have the pointwise
  estimate
\begin{equation}\label{eq:fp.est.1}  f(\mathbf{x}_0, t_0)^p \leq C \int_{L_{-1}} f^p \rho_{\mathbf{x}_0, t_0}. 
\end{equation}
At the same time from \eqref{eq:lS10} we have that
\begin{equation}\label{eq:fp.est.2}
\int_{L_{-1}} f^p \rho_{\bOh,0} \leq C I_V(M_\tau)^{2p}.
\end{equation}
  To estimate $f(\mathbf{x}_0, t_0)^p$ we therefore need to bound $\rho_{\mathbf{x}_0,
    t_0} / \rho_{\bOh,0}$ at $t=-1$. We have
  \begin{equation} 
  \label{eq:rhoest1}
 \frac{\rho_{\mathbf{x}_0, t_0}}{\rho_{\bOh,0}} (\mathbf{x},-1)= C_{t_0} \exp\Big(
    -\frac{|\mathbf{x}-\mathbf{x}_0|^2}{4(t_0+1)} + \frac{|\mathbf{x}|^2}{4}\Big), 
    \end{equation}
  for a $t_0$-dependent constant $C_{t_0}$ (which is
  uniformly bounded as long as $t_0$ is bounded away from $-1$). 
  Since
  \[ (t_0+1) |\mathbf{x}|^2 - |\mathbf{x}-\mathbf{x}_0|^2 = t_0 |\mathbf{x} + t_0^{-1}\mathbf{x}_0|^2 -
    t_0^{-1}(t_0+1) |\mathbf{x}_0|^2 \]
  and $t_0\in(-1,0)$, we have
  \begin{equation}\label{eq:rhoest2}
    -\frac{|\mathbf{x}-\mathbf{x}_0|^2}{4(t_0+1)} + \frac{|\mathbf{x}|^2}{4} \leq
    -\frac{|\mathbf{x}_0|^2}{4t_0}. 
    \end{equation}
  It follows from \eqref{eq:fp.est.1}, \eqref{eq:fp.est.2},
  \eqref{eq:rhoest1} and \eqref{eq:rhoest2} that  
  $$ f(\mathbf{x}_0, t_0)^p \leq C_{t_0}\exp\left(-\frac{|\mathbf{x}_0|^2}{4t_0}\right) \int_{L_{-1}}
    f^p\rho_{\bOh,0} \leq C_{t_0}
    \exp\left(-\frac{|\mathbf{x}_0|^2}{4t_0}\right)
    I_V(M_\tau)^{p}. $$
  Scaling this estimate back to the rescaled flow $M_\tau$ (i.e. scaling $\mathbf{\bar x} = (-t_0)^{-1/2} \mathbf x$ at $t=t_0$ and $e^{-s} = -t_0$) we then
  have the following estimate on $M_{\tau+s}$: 
  $$ ( e^{2s} |zw|^2 + |\theta - \theta_V|^2)^p \leq C e^{ |\mathbf{x}|^2/4}
    I_V(M_\tau)^{2p}, $$
  and so
  $$ |zw|^2 + |\theta - \theta_V|^2 \leq C e^{ |\mathbf{x}|^2 / 4p}
    I_V(M_\tau)^{2} $$
  as claimed in property (2).

  To see property (3) we use \eqref{eq:lS10} again. With $s=2$ it
  implies that for any compact set $K$, with suitable $p,p' > 1$ we have
  \begin{equation}\label{eq:int40}
    \begin{aligned}
    \int_{M_{\tau+1}\setminus K} f e^{-|\bx|^2/4} &\leq \left(
      \int_{M_{\tau+1}\setminus K} e^{-|\bx|^2/4}\right)^{1/{p'}}
    \left(\int_{M_{\tau+1}\setminus K} f^p e^{-|\bx|^2 / 4}\right)^{1/p} \\
    &\leq \mathrm{Vol}(M_{\tau+1}\setminus K, e^{-|\bx|^2/4})^{1/{p'}}
    CI_V(M_\tau)^2.
    \end{aligned}
  \end{equation}
    Together with the integral bound on $K$ at $\tau+1$ assumed in the statement of property (3), if we choose
    $K$ sufficiently large, depending on $\gamma$, we get
    $$ I_V(M_{\tau+1})^2 \leq \epsilon^2 + C\gamma^2 I_V(M_\tau)^2, $$
    as required. 
\end{proof}

Next we control the growth of the distance $D_V$.

\begin{prop}\label{prop:Dgrowth1}
  There is a constant $C > 0$ such that
  if the constant $c_0$ in Condition \eqref{cond:ddag} and $D_V(M_\tau)$ are
  sufficiently small, then 
  \begin{equation}\label{eq:DV40}
  D_V(M_{\tau+s}) \leq CD_V(M_\tau)\quad\text{for $s
  \in [0,1]$,}
\end{equation}
as long as $\tau \in [T_0+1,T_1-2]$. 
\end{prop}
\begin{proof}
   Given the growth bound for $I_V$ in Lemma~\ref{lem:IVest}, we only
  need to ensure that if $M_{\tau}$ is a $C^1$-small graph over $V$ on
  the annulus $B_2\setminus B_1$, then so is $M_{\tau+s}$ for
  $s\in[0,1]$. We will show that this follows from the estimate for the excess, defined in Definition \ref{dfn:excess}. 
  
  Recall that
$\theta$ satisfies the heat equation along the Lagrangian mean
curvature flow and $|\nabla \theta| = |\mathbf{H}|$. 
 It follows from Huisken's monotonicity formula that  we have
  \[ \label{eq:Achange} \mathcal{A}(M_{T_0}) - \mathcal{A}(M_{T_1})
    \geq \int_{T_0}^{T_1}  \int_{M_\tau} \left( 2|\mathbf{H}|^2 + \left|\mathbf{H} +
        \frac{\mathbf{x}^\perp}{2}\right|^2\right) e^{-|\mathbf{x}|^2/4}. \]
 
 We can then argue by contradiction, and suppose that we have a
  sequence of rescaled flows $M^i_\tau$, satisfying Condition \eqref{cond:ddag}  with corresponding constants $c_0 \to 0$, however the conclusion
  \eqref{eq:DV40} does not hold.  We can assume that we are working at
  $\tau=0$ and that also 
  $D_V(M^i_0)\to 0$. 
  
  As in
  Neves~\cite[Theorem A]{Neves:zero-maslov}, up to choosing a subsequence, the flows $M^i$
  converge to a static limit flow $M^\infty$ given by a union of planes $V'$ for $\tau \in (-1,2)$. 
  The assumption that $D_V(M^i_0)\to 0$ implies that $M^i_0$ is the
  graph of a (vector-valued) function $u_i$ over $V$ on the annulus $B_2\setminus B_1$,
  with $|u_i|, |\nabla u_i| < c_0$ and $|u_i| \to 0$. Note that this implies that $V'=V$. By White's regularity theorem~\cite{White.regularity} the convergence is smooth on $(B_3 \setminus B_{1/2})\times (-1/2,3/2)$. This implies that, for $i$ sufficiently large, $M^i_\tau$ is the graph of a function $|u_i|, |\nabla u_i| < c_0$ over $V$ on $B_2\setminus B_1$ for all $t\in [0,1]$.
  
  Thus Lemma~\ref{lem:IVest} implies $D_V(M^i_s)\leq CD_V(M^i_0)$ for all  $s\in [0,1]$ as required. 
\end{proof}

\begin{remark}
From now on we will assume that $c_0>0$ in Condition \eqref{cond:ddag} is small
enough for Proposition~\ref{prop:Dgrowth1} to hold. 
\end{remark}

 For later use, we will need to compare the distances $D_V$ as we vary the cone
  $V\in\mathcal{V}$. For $V, V'\in \mathcal{V}$ let $d(V,V')$ denote
  the Hausdorff distance between $V\cap B_1$ and $V'\cap B_1$.  
  \begin{lemma}\label{lem:DVV'comp}
    There is a constant $C$ such that if $V, V'\in \mathcal{V}$ and
    $D_V(M_\tau)$ is sufficiently small, then 
    $$ D_{V'}(M_\tau) \leq C( D_{V}(M_\tau) + d(V,V')). $$
  \end{lemma}
  \begin{proof}
    For any $\mathbf{x}\in B_1$ we have $d_{V'}(\mathbf{x}) \leq d_V(\mathbf{x}) + d(V, V')$,
    so by scaling we have $d_{V'}(\mathbf{x}) \leq d_V(\mathbf{x}) + |\mathbf{x}| d(V, V')$ for
    $\mathbf{x}\in \mathbb{C}^2$. We also have $|\theta_V - \theta_{V'}|\leq C
    d(V, V')$ for a constant $C$. Combining these observations we get
    $$ I_{V'}(M_\tau) \leq C (I_V(M_\tau) + d(V,V')^2). $$
    To get the same estimate for $D_V, D_{V'}$ we just need to ensure
    that $M_\tau$ is graphical over $V'$ on $B_2\setminus
    B_1$ for $u$ with $|u|,|\nabla u|$ sufficiently small as in Definition \ref{dfn:IVDV}. This follows using Condition \eqref{cond:ddag} as in the proof of
    Proposition~\ref{prop:Dgrowth1}. 
  \end{proof}
  
  \section{Graphicality}\label{sec:graphicality}

\subsection{Graphicality and distance to planes}
We want to see that an estimate for $D_V(M_{\tau-1})$ can be used to
show graphicality of $M_\tau$ on a much larger region than just the
fixed annulus $B_2\setminus B_1$. For this we first need the
following.
\begin{lemma}\label{lem:graphextension}
  Let $L_t$ be a mean curvature flow of surfaces in the ball
  $B_2\subset \mathbb{C}^2$ for $t\in [-4,0]$, with uniform
  area bound $\mathcal{H}^2(L_t) \leq C$. Suppose that $S$ is an embedded smooth surface
  passing through the origin, and with second fundamental form satisfying
  $|A|\leq \delta$ for some $\delta>0$. Assume that in addition we have the following.
  \begin{enumerate}
  \item $L_t$ is contained in the $\delta$-neighbourhood of $S$.
  \item On the parabolic ball $[-1/4,0] \times B_{1/2}$ the flow $L_t$
    is the graph of a (vector-valued) function $u$ over $S$ with $|u| <\delta$ and $|\nabla u| < 1$.
  \item We have the estimate
    \begin{equation}\label{eq:Hest10} \int_{-4}^0 \int_{L_t\cap B_2}
      |\mathbf{H}|^2 < \delta.
      \end{equation}
  \end{enumerate}
 Let $\varepsilon>0$ be given. Then if $\delta$ is chosen sufficiently small, the flow $L_t$ is the
  graph of $u$ over $S$ on the parabolic ball $[-1,0]\times B_1$, with
  $|\nabla u| \leq \varepsilon$.
\end{lemma}
\begin{proof}
  We argue by contradiction. Suppose that we have a sequence of flows
  $L_t^i$ satisfying the assumptions with corresponding constants
  $\delta_i\to 0$, and surfaces $S_i$, but the claimed graphicality
  fails for all $i$. Up to choosing a subsequence we
  can assume that the $S_i$ converge to a plane
  $S_\infty$. We can furthermore assume that the flows $(L^i_t)_{-4\leq t \leq 0}$ converge to a unit regular Brakke flow $(\mu_t)_{-4\leq t \leq 0}$ on $B_2$, supported on $S_{\infty}$. The constancy theorem implies that $\mu_t$ agrees with $S_{\infty}$ up to an integer multiplicity, which is monotonically decreasing in time.
  
Note that the graphicality assumption (2) together with interior parabolic estimates
  implies that $L^i_t\to S_\infty$ smoothly on $[-1/4,0]\times B_{1/2}$ , and thus $\mu_t$ agrees with $S_\infty$ with multiplicity 1 on $[-1/4,0]\times B_{2}$. 
 
At the same time, for a cutoff function $\chi$ supported in $B_2$ we
 have
  $$ \bigg|\partial_t \int_{L^i_t} \chi \bigg| = \bigg|\int_{L^i_t} - \chi |\mathbf{H}|^2 +
    \langle D\chi, \mathbf{H}\rangle\bigg|  \leq C\int_{L^i_t\cap B_2} 
    |\mathbf{H}|^2 +  |\mathbf{H}|, $$
  so for any $t_0 < t_1$ we have
 $$ \bigg|\int_{L^i_{t_0}} \chi - \int_{L^i_{t_1}} \chi\bigg|  \leq C \int_{-4}^0 \int_{L_t^i\cap B_2}
    |\mathbf{H}|^2 + |\mathbf{H}| \to 0. $$
  It follows that $(\mu_t)$ is static on $[-4,0]$ and thus agrees with $S_\infty$ with multiplicity one. From White's regularity theorem \cite{White.regularity} we deduce that
  for sufficiently large $i$ the flows $L^i_t$ converge smoothly on $[-1,0]\times B_1$, which implies the required
  graphicality.
  \end{proof}

Using Lemma \ref{lem:graphextension} repeatedly we can extend the graphicality of our
flow to larger and larger regions, as long as it stays in a small
neighbourhood of a smooth surface and we can control the integral of
$|\mathbf{H}|^2$. For the latter we have the following result.
\begin{lemma}\label{lem:intHbound}
  There are constants $C,  r_0, \alpha > 0$ and $p > 1$ satisfying the following. 
  Suppose that $L_t$ is a Lagrangian mean curvature flow for $t\in
  [-1,0)$, where $L_{-1} = M_\tau$ for some $\tau$. Then 
  whenever $r\leq r_0$ and $t_0 \in (-3/4,-1/4)$ with $[t_0-r^2, t_0+r^2]\subset [-3/4,-1/4]$ and 
  \[\label{eq:DV.x0.p.bound} D_V(M_{\tau-1})^2 \exp\Big(-\frac{|\mathbf{x}_0|^2}{4pt_0}\Big) \leq
  1,\]
  we have
 \[\label{eq:upper-bound-intH2} r^{-2} \int_{t_0-r^2}^{t_0} \int_{L_t\cap B_r(\mathbf{x}_0)} |\mathbf{H}|^2 \leq 
    CD_V(M_{\tau-1})^\alpha. \]
\end{lemma}
\begin{proof}
  We consider the monotonicity formula applied with the backwards heat
  kernel $\rho_{\mathbf{x}_0,t_0+r^2}$ centered at $(\mathbf{x}_0, t_0 + r^2)$. Using  $|\nabla \theta|^2
  = |\mathbf{H}|^2$ we have that
  \[\label{eq:H.est.1} \int_{t_0-r^2}^{t_0} \int_{L_t} 2|\mathbf{H}|^2 \rho_{\mathbf{x}_0, t_0+r^2} \leq
    \int_{L_{-1}} |\theta - \theta_V|^2 \rho_{\mathbf{x}_0, t_0 + r^2}. \]
  By the pointwise estimate in Lemma~\ref{lem:IVest} there is some
  $p_1>1$ close to $1$ and  $C > 0$ such that on $L_{-1}$ we have 
 \[\label{eq:H.est.2} |\theta - \theta_V|^2 \leq C e^{|\mathbf{x}|^2/4p_1} D_V(M_{\tau-1})^2. \]
 Therefore we need to estimate the integral
 $$ \int_{L_{-1}} \exp\Big(\frac{|\mathbf{x}|^2}{4p_1}
   -\frac{|\mathbf{x}-\mathbf{x}_0|^2}{4(t_0+r^2+1)}\Big). $$
Let $p_2 \in (1,p_1)$. Arguing as in \eqref{eq:rhoest2}
we find that
 \[\label{eq:H.est.3} \frac{|\mathbf{x}|^2}{4p_2} -\frac{|\mathbf{x}-\mathbf{x}_0|^2}{4(t_0+r^2+1)} \leq
   \frac{|\mathbf{x}_0|^2}{4(p_2-1-t_0-r^2)}.\]
The integral of $\exp\Big(\frac{|\mathbf{x}|^2}{4p_1}
 -\frac{|\mathbf{x}|^2}{4p_2}\Big)$ on $L_{-1}$ is uniformly bounded, so combining \eqref{eq:H.est.1}, \eqref{eq:H.est.2} and \eqref{eq:H.est.3} gives
 $$ \int_{t_0-r^2}^{t_0} \int_{L_t} |\mathbf{H}|^2 \rho_{\mathbf{x}_0, t_0+r^2} \leq C
   D_V(M_{\tau-1})^2 \exp\Big(\frac{
     |\mathbf{x}_0|^2}{4(p_2-1-t_0-r^2)}\Big). $$
 It remains to choose $p$ close enough to 1 in the bound \eqref{eq:DV.x0.p.bound} and $\alpha, r_0
 > 0$ small enough so that for $r\leq r_0$ we have
 $$D_V(M_{\tau-1})^{2-\alpha} \exp\Big(\frac{
     |\mathbf{x}_0|^2}{4(p_2-1-t_0-r^2)}\Big) \leq 1. $$
 Using \eqref{eq:DV.x0.p.bound}, this follows if
 \[\label{eq:alpha.ineq} \frac{2-\alpha}{2}\frac{1}{4pt_0} + \frac{1}{4(p_2-1-t_0-r^2)} <
   0. \]
 Note that in the limiting case, when $\alpha$ and $r_0$ are both $0$ and $p=1$, \eqref{eq:alpha.ineq} reduces to the inequality $p_2 >  1$, which is satisfied. We can
 therefore arrange that \eqref{eq:alpha.ineq} also holds for suitable
 $\alpha, r_0, p$. Combining this with the fact that on $[t_0-r^2,t_0]\times B_r(\bx_0)$ the function $\rho_{\bx_0,t_0+r^2}$ is bounded from below by a positive multiple of $r^{-2}$ yields \eqref{eq:upper-bound-intH2}.
\end{proof}

\subsection{Graphicality scale}
Note that under the correspondence between the mean curvature flow and
its rescaled version, if $L_{-1} =M_{\tau}$, then $M_{\tau+1}
=e^{1/2} L_{-e^{-1}}$. In particular, setting $t_0 = -e^{-1}$ and, for $\mathbf{x}_0\in L_{-e^{-1}}$, letting $\tilde{\mathbf{x}}_0 =
e^{1/2} \mathbf{x}_0$, we see that $|\mathbf{x}_0|^2 / 4pt_0 = |\tilde{\mathbf{x}}_0|^2 / 4p$. This
motivates the following. 

\begin{definition}\label{dfn:Rd}
We choose $p_0 > 1$ smaller than the $p>1$ in both
Lemma~\ref{lem:IVest} for $s=1$ and Lemma~\ref{lem:intHbound}. For
any (small) $d > 0$, we define $R_d > 0$ so that
\begin{equation}\label{eq:Rddefn}
  d^2 e^{R_d^2 / 4p_0} = 1.
\end{equation}
\end{definition}

The $R_d$ just defined will be the radius up to which we can obtain good graphicality of our flow. This motivates the following definition.

\begin{definition}%[Good graphicality]
\label{dfn:good-graphicality}
We  say that $M_s$ has \emph{good graphicality} on an annulus $A:= B_{r_2}\setminus B_{r_1}$ for $0<r_1<r_2<\infty$ over $V$ if $M_s \cap A$ is the graph of a (vector-valued) function $u$ over $V$ with $|u|, |\nabla u |\leq c_1$, where $c_1$ is as in Definition~\ref{dfn:V}.
\end{definition}

We have the following.

\begin{prop}\label{prop:graphical} Use the notation of Definition \ref{dfn:Rd}. 
  There are constants $\epsilon, A > 0$  and $p > p_0$  such that if $D_V(M_{\tau-1})
  =d < \epsilon$, 
  then $M_{\tau+1}$ is the graph of a (vector-valued) function $u$ over $V$ on the
  annulus $B_{R_d} \setminus B_{Ad^{1/2}}$, satisfying the following
  estimates:
  \begin{itemize}
    \item for $1 < |\mathbf{x}| < R_d$ we have $|u|, |\nabla u| \leq A
      e^{|\mathbf{x}|^2 / 8p}d$;
     \item for $Ad^{1/2} < |\mathbf{x}| < 2$ we have $|\mathbf{x}|^{-1}|u|, |\nabla u|
       \leq Ad |\mathbf{x}|^{-2}$.
   \end{itemize}
\end{prop}
\begin{proof}  Note that 
  by Proposition~\ref{prop:Dgrowth1},  $M_s$ has good graphicality
  over $V$ (in the sense of Definition~\ref{dfn:good-graphicality}) on  $B_2\setminus B_1$ for $s\in
  [\tau, \tau+1]$. In addition, property (2) in
  Lemma~\ref{lem:IVest} shows that there is $C>0$ and $p>p_0$ such that  
  \[\label{eq:dVbound} d_V^2 \leq C |\mathbf{x}|^{-2} e^{|\mathbf{x}|^2/4p} d^2 \]
  on $M_s$ for $s\in [\tau,\tau+1]$.
Therefore, by the definition of $R_d$ in \eqref{eq:Rddefn}, if $1 < |\mathbf{x}| < R_d$ we have
  \begin{equation} \label{eq:dVbound10}
    d_V^2 \leq C e^{R_d^2(p_0-p)/4pp_0}.
  \end{equation}
  In particular, recalling that $p>p_0$, if we  let $d$ be small, so that $R_d$ is large by \eqref{eq:Rddefn}, we can
  ensure that $d_V$ is as small as we like on the annulus
  $B_{R_d}\setminus B_1$ along the rescaled flow $M_{s}$ for
  $s\in[\tau,\tau+1]$.

  Applying Lemma~\ref{lem:intHbound} with $t_0 = -e^{-1}$ we see that for all $\bx_0$ with
  $ \exp(-\tfrac{|\bx_0|^2}{4pt_0}) \leq d^{-2}$ we have
  \[\label{eq:bound-int-h2}
   r_0^{-2} \int_{t_0-r_0^2}^{t_0} \int_{L_t\cap B_{r_0}(\bx_0)} |\mathbf{H}|^2 \leq C d^\alpha.\]
  Rescaling the flow parabolically such that $L_{-1} =
  M_{\tau+1}$ (i.e.~scaling parabolically by $e^{1/2}$) this  implies
 that for any $\mathbf{x}_0\in B_{R_d}$ we have a
  backwards parabolic ball $[-1-r_0^2, -1]\times B_{r_0}(\bx_0)$ on which the 
  spacetime integral of $|\mathbf{H}|^2$ is bounded by
  $Cr_0^2d^\alpha$. Note that the constant changes by a controlled factor due to the rescaling.
  
  Recall that
  the second fundamental form of $V$ vanishes and that $M_s$ has good graphicality over $V$ on $B_2\setminus B_1$ for $s \in [\tau,\tau+1]$. Since $L_{-1} = M_{\tau+1}$ this implies that $L_t$ is the graph of a (vector-valued) function $\hat u$ over $V$ on $\sqrt{-t} (B_2\setminus B_1)$ with $|\hat u| \leq \sqrt{-t}\cdot c_1, |\nabla \hat u| \leq c_1$ for $t\in [-e, -1]$.
  
  From \eqref{eq:bound-int-h2}, scaling the flow by $4r_0^{-1}$ and shifting the origin accordingly, by taking $d$ sufficiently small, we see for every $\bx_0 \in B_{4 r_0^{-1}R_d}\setminus B_{4r_0^{-1}}$ that \eqref{eq:bound-int-h2} implies that the hypotheses (1)--(3) of 
  Lemma~\ref{lem:graphextension} (with $\varepsilon = c_1$) are satisfied.  Thus, undoing the scaling and shifting of the origin,
  the good graphicality of $L_{t}$ for $t\in [-1-r_0^2, -1]$ over $V$
  can be propagated out from 
  the ball $\sqrt{-t}(B_2\setminus B_1)$ if $d$ is sufficiently small, to the
  annulus $\sqrt{-t}(B_{R_d}\setminus B_1)$ on which $L_{t}$ is still in a
  sufficiently small neighbourhood of $V$. Note that once we have good
  graphicality of the flow on a parabolic ball, then on a smaller
  parabolic ball $|\nabla \hat u|$ can
  be bounded in terms of $|\hat u|$ by standard parabolic theory. This introduces the constant $A$ in the claimed estimates.
  
  The argument for extending graphicality to the annulus $B_2\setminus
  B_{Ad^{1/2}}$ is similar. Here the distance bound
  \eqref{eq:dVbound} implies that $d_V \leq C |\mathbf{x}|^{-1} d$ on $L_{t}$ for $t\in [-1,-e]$ (recalling that $L_{-1}=M_{\tau+1}$). Suppose
  that $\mathbf{x}_0\in V$ with $Ad^{1/2}\leq |\mathbf{x}_0| < 2$, and let $r = A^{-1}|\mathbf{x}_0|$.  If $A$ is sufficiently large, then $V\cap B_r(\mathbf{x}_0)$ is a plane, so
  it has zero second fundamental form. In addition
  for $t\in[-1,e]$ the distance from $L_{t}$ to $V$ on $B_r(\mathbf{x}_0)$ is
  bounded by $C |\mathbf{x}_0|^{-1}d$, where $C$ can be chosen
  independent of sufficiently large $A$. After scaling up by $r^{-1}$ the
  distance from $r^{-1}L_{t}$ to $V$ on $r^{-1}B_r(\mathbf{x}_0)$ is bounded above by
  $CA |\mathbf{x}_0|^{-2} d \leq CA^{-1}$ since $|\mathbf{x}_0|\geq Ad^{1/2}$. By choosing
  $A$ large enough, we can again ensure that
  Lemma~\ref{lem:graphextension} can be applied repeatedly to extend
  the region of good graphicality. The
  required estimate for $|u|$ follows from the pointwise bound for
  $d_V$, while the estimate for $|\nabla u|$ in
  the annular region $B_2\setminus B_{Ad^{1/2}}$ follows by standard
  parabolic theory and scaling parabolic balls of the form $[t-r^2,
  t]\times B_r(\mathbf{x})$ to unit size, where $r = |\mathbf{x}|/2$. 
\end{proof}

\subsection{Excess and distance}
The graphicality bound from Proposition \ref{prop:graphical} implies the following estimate for the excess
$\mathcal{A}$ in \eqref{eq:AMdefn} in terms of $D_V(M)$.
\begin{prop}\label{prop:Aest}
  There is a small $\alpha_1 > 0$ such that if $D_V(M_{\tau-1})$ is
  sufficiently small, then  
  $$\label{eq:Aest}|\mathcal{A}(M_\tau)| \leq D_V(M_{\tau-1})^{1+\alpha_1}.$$
\end{prop}
\begin{proof}
By Definition~\ref{dfn:IVDV} and Proposition~\ref{prop:Dgrowth1} we have
  $$\int_{M_\tau} |\theta-\theta_V|^2 e^{-|\mathbf{x}|^2/4} \leq
    D_V(M_\tau)^2 \leq CD_V(M_{\tau-1})^2. $$
Therefore, by formula \eqref{eq:AMdefn} we only need to estimate the difference between the Gaussian
  areas of $M_\tau$ and   $V$ (recalling that $V$ is a pair of planes). 
  
  Let $d=D_V(M_{\tau-1})$ and recall $p_0>1$ and $R_d>0$ given in Definition~\ref{dfn:Rd}.  We also recall the constants $\epsilon,A>0$ from Proposition~\ref{prop:graphical} and we assume that $d<\epsilon$.  We further assume that $d$ is
    sufficiently small so that $d^{1/10} > Ad^{1/2}$.
  
  We study four regions
  separately.
  
\smallskip
\paragraph{\textbf{(a)} $|\mathbf{x}|>R_d$.}  By the area growth bounds for $M_\tau$ we have constants $C,k>0$ such that
    $$ \int_{M_\tau\setminus B_{R_d}} e^{-|\mathbf{x}|^2/4} \leq CR_d^k
      e^{-R_d^2/4}.$$
   Once $R_d$ is sufficiently large, 
   $$CR_d^k
    e^{-R_d^2/4} \leq e^{-R_d^2/4p_0} = d^2.$$
    The same integral bound
    also holds on $V\setminus B_{R_d}$ so the required estimate \eqref{eq:Aest}
    holds on this region with $\alpha_1=1$.

\smallskip
\paragraph{\textbf{(b)} $1 < |\mathbf{x}| < R_d$.} Here
    Proposition~\ref{prop:graphical} states that $M_\tau$ is the graph
    of $u$ over $V$, with $|u|, |\nabla u|\leq
    Ae^{|x|^2/8p}d$, for some $p > p_0>1$. By the definition of $R_d$ in \eqref{eq:Rddefn},
    we have $Ae^{R_d^2/8p}d \to 0$ as $d\to 0$. It follows that the
    area form $dA_{M_\tau}$ of $M_\tau$, pulled back to $V$, can be
    compared to the area form $dA_V$ of $V$ as follows:
    $$ \left|\frac{dA_{M_{\tau}}}{dA_V} -1\right|\leq  C e^{|\mathbf{x}|^2 / 4p} d^2, $$
    for some constant $C$. 
    Note that the difference in the area forms is at least quadratic
    in $u$ since $V$ has zero mean curvature. Integrating, we have
    $$ \left|\int_{M_\tau\cap (B_{R_d}\setminus B_1)} e^{-|\mathbf{x}|^2/4} -
      \int_{V\cap (B_{R_d}\setminus B_1)} e^{-|\mathbf{x}|^2/4}\right| \leq Cd^2
      \int_{V\cap  (B_{R_d}\setminus B_1)} e^{|\mathbf{x}|^2/4p}
      e^{-|\mathbf{x}|^2/4}. $$
    Since $p>1$ the last integral is bounded
    independently of $d$, and so the required estimate \eqref{eq:Aest} holds on this
    region too, with $\alpha_1=1$. 

\smallskip
\paragraph{\textbf{(c)} $d^{1/10} < |\mathbf{x}| < 1$.}  Since we have assumed that $d^{1/10} > Ad^{1/2}$, 
    Proposition~\ref{prop:graphical} implies that on this
    region $M_\tau$ is the graph of $u$ over $V$, with 
    $|\mathbf{x}|^{-1}|u|, |\nabla u| \leq Ad |\mathbf{x}|^{-2}$. Similarly to (b)
    we can again compare the area forms:
    $$ \left|\frac{dA_{M_{\tau}}}{dA_V} -1\right|\leq C d^2 |\mathbf{x}|^{-4}. $$
    Integrating, we have
    $$\begin{aligned} \Bigg| \int_{M_\tau\cap (B_1\setminus B_{d^{1/10}})} e^{-|\mathbf{x}|^2/4} &-
      \int_{V\cap (B_1\setminus B_{d^{1/10}})} e^{-|\mathbf{x}|^2/4}\Bigg|\\
       &\leq  Cd^2
      \int_{V\cap  (B_1\setminus B_{d^{1/10}})} |\mathbf{x}|^{-4}
      e^{-|\mathbf{x}|^2/4} \\
      &\leq C d^2 d^{-4/10} = C d^{8/5}.
    \end{aligned} $$
    The required estimate \eqref{eq:Aest} therefore holds with $\alpha_1 = 3/5$.  

    \smallskip
    \paragraph{\textbf{(d)} $|\mathbf{x}|
 < d^{1/10}$} 
 Let us write $r_0=d^{1/10}$ for the radius of this ball for simplicity.  
 As in (c), Proposition~\ref{prop:graphical} implies that the cross section $M_\tau \cap \partial B_{r_0}$  is an exponential normal
  graph in the sphere $\partial B_{r_0}$ of a function $\tilde u$  over $V\cap \partial B_{r_0}$, where $\tilde u$ satisfies
$r_0^{-1} |\tilde u|, |\nabla \tilde u| \leq Ad r_0^{-2} = Ad^{4/5}$. The
  cross section $V\cap \partial B_{r_0}$ of $V$ is
  minimal in the sphere (a union of geodesics), so the cross section of $M_\tau$ has length
  \[\label{eq:length.bound} \Big|\mathcal{H}^1(M_\tau \cap \partial B_{r_0}) - \mathcal{H}^1(V\cap \partial B_{r_0})\Big| \leq
    Cd^{8/5} r_0 = C d^{17/10}. \]
  Let $V_{\tau}$ be the cone over $M_\tau\cap \partial
  B_{r_0}$. By \eqref{eq:length.bound}, the area of $V_{\tau}$ then satisfies
  \[\label{eq:V'bound}
   \Big| \mathcal{H}^2(V_{\tau} \cap B_{r_0})  - \mathcal{H}^2(V\cap B_{r_0}) \Big| \leq 
    C r_0 d^{17/10} = C d^{9/5}. \]
 On $M_\tau$ we have $|\theta-\theta_V| \leq Cd$ by property (2) in Lemma~\ref{lem:IVest}, noting that the exponential factor in property (2) can be bounded independently of $d$ for $d$ small.  Hence, up to rotating
 the holomorphic volume form $\Omega$  so that we can assume $\theta_V=0$,  we have
 that
 $$ \mathrm{Re}\,\Omega|_{M_\tau} \leq dA_{M_\tau} \leq
   (1 + Cd^2) \mathrm{Re}\,\Omega|_{M_\tau}.$$
 Therefore
 \[\label{eq:A21}
   \Bigg|\int_{M_\tau \cap B_{r_0}} e^{-|\mathbf{x}|^2/4} dA_{M_\tau} &-    \int_{M_\tau\cap B_{r_0}} e^{-|\mathbf{x}|^2/4}    \mathrm{Re}\,\Omega|_{M_\tau}\Bigg|\\&\leq Cd^2  \left|\int_{M_\tau\cap B_{r_0}} e^{-|\mathbf{x}|^2/4}     \mathrm{Re}\,\Omega|_{M_\tau}\right|.    
    \] 
  At the same time, there is a hypersurface $U_{\tau}$ in $B_{r_0}$ bounded by $M_{\tau}$ and $V_{\tau}$ so that
  \begin{equation}\label{eq:A22} \int_{M_\tau \cap B_{r_0}} e^{-|\mathbf{x}|^2/4}\, \mathrm{Re}\,\Omega
    = \int_{V_{\tau} \cap B_{r_0}} e^{-|\mathbf{x}|^2 / 4}\,
    \mathrm{Re}\,\Omega + \int_{U_{\tau}} \rd(e^{-|\mathbf{x}|^2/4})\wedge
    \mathrm{Re}\,\Omega, \end{equation}
since $\rd\Omega=0$.  
  We have $\rd(e^{-|\mathbf{x}|^2/4}) = -\frac{|\mathbf{x}|}{2} e^{-|\mathbf{x}|^2/4}
  \rd|\mathbf{x}|$, so it follows that
  \begin{equation}\label{eq:A23} \left|\int_{U_{\tau}} \rd(e^{-|\mathbf{x}|^2/4})\wedge
    \mathrm{Re}\,\Omega \right|\leq C r_0\mathcal{H}^3(U_{\tau}) = C d^{1/10} \mathcal{H}^3(U_{\tau}).
    \end{equation}
  
  Let us write $U_{\tau} = U_{\tau,1}\cup U_{\tau,2}$, where 
  $$U_{\tau,1}=U_{\tau}\cap B_{Ad^{1/2}}\quad\text{and} \quad U_{\tau,2} = U_{\tau} \setminus B_{Ad^{1/2}}.$$
  Note that outside of
  $B_{Ad^{1/2}}$ the surface $M_\tau$ is still the graph of some $u$
  over $V$ satisfying $|\mathbf{x}|^{-1}|u| , |\nabla u| \leq Ad
  |\mathbf{x}|^{-2}$ by Proposition~\ref{prop:graphical}. The cone $V_{\tau}$ is also the graph of a function $v$ over
  $V$ with $|\mathbf{x}|^{-1} |v|, |\nabla v| \leq C d^{4/5}$ by construction. It follows that,
  once $d$ is small, $M_\tau$ is the graph of some $\tilde{u}$ over $V_{\tau}$
  with $|\mathbf{x}|^{-1}|\tilde{u}|, |\nabla \tilde{u}| \leq Ad |\mathbf{x}|^{-2}$ (recalling that $|\mathbf{x}|^2<d^{1/5}$ in the region under consideration). We
  can then choose $U_{\tau}$ so that $U_{\tau,2}$ is the hypersurface swept out by
  the graphs of $s \tilde{u}$ over $V_{\tau}$ for $s\in [0,1]$. We 
  estimate the volume of $U_{\tau,2}$ by the integral
  $$ \mathcal{H}^3(U_{\tau,2}) \leq C \int_{Ad^{1/2}}^{d^{1/10}} \frac{Ad}{r}\, r\, dr
    \leq C d^{11/10}. $$

  Finally we choose $U_{\tau,1}$ to minimize area, such that its boundary is
  given by the union of $U_{\tau,2}\cap \partial B_{Ad^{1/2}}$, $M_\tau\cap
  B_{Ad^{1/2}}$ and $V_{\tau}\cap B_{Ad^{1/2}}$. 
  The isoperimetric inequality then  implies that 
  $$\mathcal{H}^3(U_{\tau,1}) \leq C \mathcal{H}^2(\partial
  U_{\tau,1})^{3/2}.$$ 
  To estimate   $\mathcal{H}^2(\partial U_{\tau,1})$, consider the
  three pieces of the boundary. In the sphere $\partial
  B_{Ad^{1/2}}$ the cross sections of $M_{\tau}$ and $V_{\tau}$ are graphs
  of functions bounded by $d^{1/2}$, so $\mathcal{H}^2(U_{\tau,2}\cap
  \partial B_{Ad^{1/2}}) \leq Cd$. From \eqref{eq:V'bound} we know
  that the boundary piece $V_{\tau}\cap
  B_{Ad^{1/2}}$ also has area bounded by $Cd$. To control 
  $\mathcal{H}^2(M_\tau\cap B_{Ad^{1/2}})$ it is enough to use that $M_\tau\cap B_{Ad^{1/2}}$ is
  almost calibrated, which follows from the fact that $|\theta-\theta_V|\leq Cd$ and $d$ is small. As in  
    \cite[Lemma 7.1]{Neves:zero-maslov} we
  have
  $$ \mathcal{H}^2(M_\tau\cap B_{Ad^{1/2}}) \leq C \mathcal{H}^1(M_\tau \cap \partial
    B_{Ad^{1/2}})^2  \leq Cd. $$
  In sum, it follows that with this choice of $U_{\tau}$ we
  have
   $$\mathcal{H}^3(U)\leq
  C(d^{3/2} + d^{11/10}) \leq Cd^{11/10}.$$
Therefore, using \eqref{eq:A21}, \eqref{eq:A22} and \eqref{eq:A23}, we
have
$$ \int_{M_\tau \cap B_{r_0}} e^{-|\mathbf{x}|^2/4}\, dA_{M_{\tau}} \leq (1 +
  Cd^2)\left( \int_{V\cap B_{r_0}} e^{-|\mathbf{x}|^2/4} +
    Cd^{18/10}+ Cd^{12/10}\right). $$
Here the term involving $d^{18/10}$ is obtained from comparing the area forms of
$V_{\tau}$ and $V$. 

\smallskip

Combining our estimates on the different regions (a)-(d) we have
$$ \left|\int_{M_\tau}e^{-|\mathbf{x}|^2/4} - \int_{V}
  e^{-|\mathbf{x}|^2/4}\right| \leq Cd^{12/10}, $$
and so the required estimate \eqref{eq:Aest} for $\mathcal{A}(M_\tau)$ holds with $\alpha_1 = 1/10$, once $d$ is sufficiently small.  
\end{proof}

\section{Limiting solutions of drift heat equation}\label{sec:heateqnlimits}

In this section we show that from a sequence of rescaled flows whose initial conditions are getting closer and closer to the pair of planes $V$, we can extract in the limit a solution to the drift heat equation which, after removing a leading order singular term, is defined on each plane $P_1,P_2$ in $V$.  We also show that this solution will satisfy good estimates.
  
\subsection{Sequences of rescaled flows}  We will need to pass to
limits along sequences of rescaled 
flows.
It is here that the local exactness imposed in Condition \eqref{cond:ast} will
begin to
play a role.
First we have the following, showing that Condition \eqref{cond:ast} is preserved
along the flow as long as $D_V(M_\tau)$ stays sufficiently small.
\begin{lemma}\label{lem:condast}
  There is an $\epsilon > 0$ satisfying the following.
  Suppose   the flow $M_\tau$ satisfies Condition \eqref{cond:ddag} for
  $\tau\in [-1, T]$, with $T > 0$, and $D_V(M_\tau) < \epsilon$ for
  $\tau\in [-1,T]$. If $M_0$ satisfies Condition \eqref{cond:ast} then $M_\tau$
  satisfies Condition \eqref{cond:ast} for $\tau\in [0,T]$. 
\end{lemma}
\begin{proof}
  If $\epsilon$ is sufficiently small, then $M_{\tau}$ is a smooth graph
  over $V$ on the annulus $B_2\setminus B_{1/2}$ for $\tau\in
  [0,T]$. If follows that no additional component of the flow can
  appear in $B_1$ at any time $\tau > 0$ if $M_0\cap B_1$ is
  connected. As for the exactness, if $\gamma_t$ denotes the evolution
  of a closed loop $\gamma$ along the (unrescaled) mean curvature
  flow, then $\partial_t \int_{\gamma_t} \lambda = 0$. Moreover, by the
  graphicality statement above, any closed loop $\gamma$ in
  $M_\tau\cap B_2$ is homotopic to a closed loop in $M_\tau\cap B_1$,
  for $\tau\in [0,T]$. It follows from this that if $\int_\gamma
  \lambda=0$ for any loop $\gamma\in B_1\cap M_0$, then the same holds
  for any loop $\gamma\in B_1\cap M_\tau$ for $\tau\in [0,T]$. In
  particular $M_\tau$ satisfies Condition \eqref{cond:ast} for $\tau\in
  [0,T]$. 
\end{proof}

  \begin{prop}\label{prop:limitu10}
There is a constant $C > 0$, depending only on $\mathcal{V}$,
satisfying the following. Let $T>0$ and let $M^i_{\tau}$ be a sequence
of flows defined for $\tau\in [-1,T+2]$ which satisfy 
    Condition \eqref{cond:ddag}, and such that $M^i_0$ satisfy Condition
    \eqref{cond:ast}. Suppose that $D_V(M^i_0)  =: d_i \to 
    0$.
 
 \begin{enumerate}
 \item   There exist compact sets $K_i\subset\mathbb{C}^2\setminus\{\mathbf{0}\}$
    exhausting $\mathbb{C}^2\setminus\{\mathbf{0}\}$, satisfying the following.
    For $s\in [1,T]$ the surface
    $M^i_s$ is the graph of $u_i(s)$ over $V$ on $K_i$ such
    that, up to choosing a subsequence, the $d_i^{-1}u_i$ converge
    locally smoothly on $[1,T] \times V\setminus \{\mathbf{0}\}$ to a
    solution $u(s)$ of the drift heat equation
    \begin{equation}\label{eq:dh20}
      \partial_s u = \Delta u + \frac{1}{2}(u - \mathbf{x}\cdot \nabla u).
    \end{equation}
\item      
   The limit $u$ can be identified with an exact 
    $1$-form on $V\setminus\{\mathbf{0}\}$. We write $u = (u_1, u_2)$, where
    $u_j$ is the restriction of $u$ to the plane $P_j\setminus\{\mathbf{0}\}$
    in terms of $V = P_1\cup P_2$.

    We can further
    decompose 
    \[\label{eq:u.decomp}
    u = e^s u_0 + \tilde{u}=e^s(a_1 \rd\ln |\mathbf{x}|, a_2 \rd\ln
    |\mathbf{x}|)+(\tilde{u}_{1},
    \tilde{u}_{2}),
    \] where 
    $a_1,a_2$ are constants such that $|a_1|, |a_2| \leq
    C$, and the  
$\tilde{u}_j$    extend smoothly across the origin. 
    
 \item    We
    have the following estimates at $s=1$:    
    \begin{equation}\label{eq:limituest10}
      \begin{aligned}  
      \int_{V\setminus \{\mathbf{0}\}} |\mathbf{x}|^2 |u|^2 e^{-|\mathbf{x}|^2/4} &\leq C, \\
      \sup_{B_1\cap V\setminus\{\mathbf{0}\}} |\mathbf{x}| |u| + |\rd^*u| &\leq C.
     \end{aligned}
    \end{equation}
    \end{enumerate}
  \end{prop}
  \begin{proof}
     It follows from Proposition~\ref{prop:graphical} that
     $M^i_s$ is the graph of $u_i$ over $V$ for $s\in [1,T]$ on the annuli $B_{R_{d_i}}\setminus B_{Ad_i^{1/2}}$, where $R_{d_i}\to\infty$ as $d_i\to 0$ by Definition \ref{dfn:Rd}, and hence on  larger
     and larger compact annuli $K_i$.  Moreover, on any fixed compact set $K\subset
      V\setminus \{\mathbf{0}\}$ we have uniform bounds for
     $d_i^{-1}u_i$ and $d_i^{-1}\nabla u_i$ as $i\to\infty$. 
    Standard parabolic estimates imply that, 
     up to choosing a subsequence, the $u_i$ converge locally smoothly
     to a solution of  \eqref{eq:dh20} on
     $V\setminus\{\mathbf{0}\}$. 
     
     Using that the $M^i_\tau$ are Lagrangian it
     follows that $u$ can be identified with a closed 
     1-form. Condition \eqref{cond:ast} implies that the
     integral of $u$ along the two circles $V\cap \partial B_1$
     vanishes, and so $u$ is actually exact. Writing $u=(u_1, u_2)$ as in (b), we then have $u_j =
     \rd f_j$ for functions $f_j$ on $P_j\setminus\{\mathbf{0}\}$.

     The estimates \eqref{eq:limituest10} follow directly from the
     definition of $I_V(M^i_1)$ together with the bounds given
     in Lemma~\ref{lem:IVest}. For the bound on $\rd^*u$ note that when
     we locally view $M^i_1$ as the graph of the 1-form $u_i$
     over a plane, then the difference $\theta - \theta_V$ in the
     Lagrangian angle is given by $\rd^*u_i$ up to lower order terms in
     $u_i$.  

     It remains to show the claimed decomposition \eqref{eq:u.decomp}. For this we focus on one of the planes $P = P_j$,
     and the corresponding solution $u=\rd f$ of the drift heat
     equation. By rescaling 
     $$U(x,t) = \sqrt{-t} u( x/ \sqrt{-t},
     -\ln(-t))$$
     we obtain a solution $U$ of the heat equation on a
     time interval $[T_0', T_1']$  on $P\setminus\{\mathbf{0}\}$. We have $U=\rd F$
     for $F : P\setminus\{\mathbf{0}\} \to \mathbb{R}$  and we can arrange that
     $F$ also satisfies the heat equation. The bound $|\rd^*u|\leq C$ on
   $B_1\setminus\{\mathbf{0}\}$ implies that we also have a uniform bound $|\Delta
     F|\leq C$ on $[T_0', T_1'] \times B_r(\bOh) \setminus\{\mathbf{0}\}$ for some
     $r > 0$. Since $\Delta F$ also satisfies the heat equation, it
     follows that $\Delta F$ extends smoothly across the origin in
     $P$. Using that $\Delta F = \partial_t F$ on $P\setminus \{\mathbf{0}\}$,
     for any $t\in [T_0', T_1']$ we have
     \begin{equation}\label{eq:Fveq}
       F(t) - F(T_0') = v(t) \text{ on } P\setminus\{\mathbf{0}\},
       \end{equation}
     where $v(t) = \int_{T_0'}^t \partial_s F\, ds$ is smooth across the
     origin. Since $\Delta F(T_0')$ is smooth across
     the origin, there is a smooth function $g$ such that $F(T_0') -
     g$ is harmonic on $P\setminus \{\mathbf{0}\}$. Note that the bound $|\rd F| \leq
     C|\mathbf{x}|^{-1}$ near the origin implies that 
     $$|F(T_0') - g| \leq C|\ln|\mathbf{x}||$$ 
     near the origin. This implies that $F(T_0') - g = a \ln |x|$
     for a constant $a$ satisfying $|a|\leq C$,  up to modifying $g$ by a smooth 
     function. Using this in \eqref{eq:Fveq} we have
     that 
     $$F(t) = a\ln |\mathbf{x}| + \tilde{F}(t),$$ 
     where $\tilde{F}(t)$ extends
     smoothly across the origin, and therefore solves the heat
     equation on all of $P$.  Scaling $F$ and hence $U$ back to give $u=\rd f$, this shows the required decomposition \eqref{eq:u.decomp}. 
   \end{proof}

\subsection{Estimates for solutions of the drift heat equation}   We will use the following estimates for the  smooth part of the limiting solution 
   of the drift heat equation obtained in
   Proposition~\ref{prop:limitu10}.
   \begin{prop}\label{prop:dhest11}
     Suppose that $u$ is an exact $1$-form valued solution of the drift heat
     equation \eqref{eq:dh20} on $\mathbb{R}^2$ on the time interval
     $[0,\infty)$. 
     \begin{enumerate}
       \item Suppose that at $s=0$ we have the bounds
      \begin{align*} \int_{\mathbb{R}^2} |\mathbf{x}|^2 |u|^2 e^{-|\mathbf{x}|^2/4}
         &\leq 1, \\
         |\rd^*u| &\leq 1, \text{ on } B_1. \end{align*} 
     Then there is a uniform constant $C>0$ so that at $s=0$ we also have
    $$ \int_{\mathbb{R}^2} |u|^2 e^{-|\mathbf{x}|^2/4} \leq C. $$
   \item If at $s=0$ we have $\int_{\mathbb{R}^2} |u|^2 e^{-|\mathbf{x}|^2/4}
     \leq 1$, then at time $s=1$  we have
     $$ |u|^2, |\nabla u|^2, |\nabla^2 u|^2 \leq C e^{|\mathbf{x}|^2 / 4p} $$
     for some constants $C>0$ and  $p > 1$.
     \end{enumerate}
   \end{prop}
   \begin{proof}
     To prove (1), it is enough to show that under the assumptions we
     have a uniform bound $|u| \leq C$ on $B_1$ at $s=1$. We can write
     $u=\rd f$, with $|\Delta f| \leq 1$ and $f(p)=0$ for a basepoint
     $p\in \partial B_1$. Elliptic estimates for the system $\rd u=0,
     |\rd^*u|\leq 1$ together with the integral bound for $u$ imply that
     we have a uniform bound $|u|\leq C$  on the annulus $B_2\setminus
     B_{1/2}$, and thus $f$ satisfies a gradient bound there. It follows
     that we have $|f|\leq C$ on $\partial B_1$. Since $|\Delta f|\leq
     1$ on $B_1$, the maximum principle then implies a uniform bound $|f| \leq
     C$ on $B_1$ and so we also have a uniform gradient bound $|\rd f|
     \leq C$ on $B_{3/4}$. The required estimate for $|u|$ follows. 

     To prove (2), we first argue as in the proof of
     Lemma~\ref{lem:IVest} to obtain the pointwise estimate $|u|^2
     \leq C e^{|\mathbf{x}|^2/4p}$ for $s \in [1/4, 1]$ for some $C>0$ and $p > 1$. In order to
     estimate $\nabla u$, we can consider the evolution of $f = |u|^2 + s
     |\nabla u|^2$. In terms of the drift Laplacian \[
    \label{eq:L0}\mathcal{L}_0 =
     \Delta - \frac{1}{2} \mathbf{x}\cdot \nabla\] we have (recalling that $s\geq 0$)
$$ (\partial_s - \mathcal{L}_0) (|u|^2 + s |\nabla u|^2) = |u|^2 -
  2|\nabla u|^2  - 2s |\nabla^2 u|^2 + |\nabla u|^2 \leq |u|^2. $$
  It follows, using an estimate analogous to \eqref{eq:lS10} (see also
  \cite[Theorem 1.6.2]{Bogachev}), that at time $s=1/2$ we have
  a bound
$$ \int_{\mathbb{R}^2} |\nabla u|^{2p} e^{-|\mathbf{x}|^2/4} \leq C, $$
   for some $p > 1$.  Arguing again as in the proof of
   Lemma~\ref{lem:IVest} we obtain the required pointwise bounds for
   $|\nabla u|$ at $s=1$.  The bound for $|\nabla^2 u|$ is similar. 
   \end{proof}

\section{Three-annulus lemmas}\label{sec:threeann}

In this section we prove two versions of the 3-annulus lemma.  The
first is for solutions of the drift heat equation.  The second is for
our distance $D_V$ to the planes $V$.

\subsection{Drift heat equation}
We show the following 3-annulus lemma for solutions of the drift heat
equation given by Proposition~\ref{prop:limitu10}. Note that this is
slightly stronger than log-convexity of the norm proved by
Colding--Minicozzi~\cite{ColdingMinicozzi:frequency}.  The proof is
similar to Simon~\cite[Lemma 3.3]{Simon.lectures}. 
\begin{lemma} \label{lem:3ann1}
  There are small $0 < \lambda_1 <
  \lambda_2<1$ satisfying the following.
  Let $u = e^s a_0 \rd\ln |\mathbf{x}| + \tilde{u}$ be a solution of the drift heat
  equation \eqref{eq:dh20} on $\mathbb{R}^2\setminus\{\mathbf{0}\}$, where $\tilde{u}$ extends
  smoothly across the origin.  We define the norm
  \[\label{eq:u.norm} \Vert u(\tau)\Vert^2 = |a_0|^2e^{2\tau} + \int_{\mathbb{R}^2}|\tilde{u}(\tau)|^2
    e^{-|\mathbf{x}|^2/4} \]
  and observe that   we have a decomposition 
  \begin{equation}\label{eq:udec20}
    u(s) = a_0 e^{s} \rd\ln |\mathbf{x}| + \sum_{i>0} a_i e^{\mu_i s} \phi_i,
  \end{equation}
  where the $\phi_i$ are orthonormal eigenfunctions of the drift  Laplacian $\mathcal{L}_0$ in \eqref{eq:L0}.
  \begin{enumerate}
    \item If $\Vert u(1)\Vert \geq
  e^{\lambda_1} \Vert u(0)\Vert$ then 
  $\Vert u(2)\Vert \geq e^{\lambda_2} \Vert u(1)\Vert$.
    \item If $u\neq 0$ has no homogeneous degree zero component, i.e.~no
      term corresponding to $\mu_i=0$ in \eqref{eq:udec20}, then we
      must have 
      $$\text{either $\Vert u(2)\Vert \geq e^{\lambda_1} \Vert
      u(1)\Vert$}\quad \text{or $\Vert u(1)\Vert \leq e^{-\lambda_1} \Vert
      u(0)\Vert$.}$$
\end{enumerate}
\end{lemma}
\begin{proof}
By \eqref{eq:udec20} and the definition of the norm, if we set $\mu_0=1$ then we have
  $$ \Vert u(s)  \Vert^2 = \sum_{i=0}^\infty  a_i^2 e^{2\mu_i t}.$$
  Fix a small $a > 0$ so that if $\mu_i\in [-10a,10a]$ for some $i$,
  then $\mu_i=0$.  We have
  \begin{align*}
   \Vert u(0)  \Vert^2 &= \sum a_i^2, \\
     e^{-2a}\Vert u(1)  \Vert^2  &= \sum a_i^2 e^{2\mu_i - 2a}, \\
     e^{-4a} \Vert u(2)  \Vert^2  &= \sum a_i^2 e^{4\mu_i - 4a}. 
     \end{align*}
   It follows that
   \begin{align*} \frac{1}{2} (\Vert u(0)  \Vert^2 + e^{-4a}\Vert u(2)  \Vert^2 ) &= \sum a_i^2
     e^{2\mu_i -2a} \frac{1}{2} (e^{2a-2\mu_i} +
     e^{2\mu_i-2a}) \\
     &= \sum a_i^2 e^{2\mu_i - 2a} \cosh( 2\mu_i-2a). 
     \end{align*}
   By our choice of $a$ we have $|2\mu_i - 2a| \geq a$ for all $i$, so we get
   \begin{equation} \label{eq:u30}
     \frac{1}{2} (\Vert u(0)  \Vert^2  + e^{-4a} \Vert u(2)  \Vert^2 ) \geq (1 + c)
     e^{-2a}\Vert u(1)  \Vert^2 ,
   \end{equation}
   for some $c > 0$. We choose $\lambda_1 < a < \lambda_2$ with
   $\lambda_j$ very close 
   to $a$. If $\Vert u(1)  \Vert^2  \geq e^{2\lambda_1}\Vert u(0)  \Vert^2 $
   then from \eqref{eq:u30}  we have
   $$ (1 + c) e^{-2a} \Vert u(1)  \Vert^2  \leq \frac{1}{2}e^{-4a}
     \Vert u(2)  \Vert^2  + \frac{1}{2} e^{-2\lambda_1} \Vert u(1)  \Vert^2 . $$
   Rearranging this we have
   $$ \Vert u(2)\Vert^2 \geq e^{2a}\Big( 2(1+c ) -
     e^{2a-2\lambda_1}\Big) \Vert u(1)\Vert^2 \geq e^{2\lambda_2}
     \Vert u(1)\Vert^2, $$
   if the $\lambda_j$ are sufficiently close to $a$. This shows 
   (1). The proof of (2), choosing $\lambda_1,\lambda_2$ closer to
   zero if necessary, is similar. 
\end{proof}

\subsection{Distance}
Using Lemma~\ref{lem:3ann1} we can show a 3-annulus lemma for the distance
$D_V$, by a contradiction argument. 
\begin{prop}\label{prop:3ann2} Let $\lambda_1,\lambda_2$ be as in
  Lemma~\ref{lem:3ann1}.  Let $0<\lambda_1' < \lambda_2'$ be such that $\lambda_j' \in
  (\lambda_1, \lambda_2)$. There is a large $N_0 > 0$ satisfying the
  following. Given an integer $N > N_0$, suppose that the flow
  satisfies Condition \eqref{cond:ddag} for $\tau\in [-1, 2N+10]$, and $M_0$
  satisfies Condition \eqref{cond:ast}.  There is an $\epsilon > 0$ depending on
  $N$ such that if $D_V(M_0) < \epsilon$, then 
  $$ D_V(M_{N}) \geq e^{\lambda_1' N} D_V(M_0) \text{ implies
    } D_V(M_{2N}) \geq e^{\lambda_2' N} D_V(M_{N}). $$
\end{prop}
\begin{proof}
  The proof is by contradiction, similar to the proof of \cite[Lemma
  2, p.~549]{Simon.asympt}, using property (3) in
  Lemma~\ref{lem:IVest} to deal with the singularity of $V$ at the origin
  and its noncompactness. See \cite[Proposition 5.12]{Sz20} for a
  related argument.

Suppose that the result fails for a given large integer $N$, so that we
  have a sequence of flows $M^i_\tau$ with $D_V(M^i_0)\to 0$ such
  that the conclusion fails. By Proposition~\ref{prop:Dgrowth1} we
  have $d_i = D_V(M^i_N)\to 0$ and our
  hypothesis can be written:
  \begin{equation}\label{eq:DVa1}
    \begin{aligned} D_V(M^i_0) &\leq
    e^{-\lambda_1'N}d_i, \\
    D_V(M^i_{2N}) &<  e^{\lambda_2'N} d_i.
  \end{aligned} \end{equation}
  In particular $d_i > 0$. 
  Using  Proposition~\ref{prop:limitu10} we can write $M^i_{ s}$
  as the graphs of $u_i(s)$ for $s\in [1, 2N + 8]$ over $V$ on larger
  and larger compact sets $K_i\subset V\setminus\{\mathbf{0}\}$. The
  inequalities \eqref{eq:DVa1} and Proposition~\ref{prop:Dgrowth1}
  imply that $D_V(M^i_{s})\leq C_N d_i$ for $s\in [1,2N+8]$. 
  
  As
  in Proposition~\ref{prop:limitu10}, up to choosing a subsequence, we
  can assume that the $d_i^{-1}u_i$ converge locally smoothly to a
  limit solution $u$ of \eqref{eq:dh20} on $V\setminus\{\mathbf{0}\}$. We can
  write $u = e^sa_0\rd \ln |\mathbf{x}| + \tilde{u}$ by Proposition \ref{prop:limitu10}, where $\tilde{u}$
  is smooth across the origin and $a_0$ is constant on each plane in $V$. Using \eqref{eq:limituest10} and
  \eqref{eq:DVa1}  
   the limit satisfies the estimates
  \[\label{eq:1.2N+1.est} \int_V |\mathbf{x}|^2 |\tilde{u}(1)|^2 e^{-|\mathbf{x}|^2/4} \leq C e^{-2\lambda_1'N}, \quad
    \int_V |\mathbf{x}|^2 |\tilde{u}(2N+1)|^2 e^{-|\mathbf{x}|^2/4} \leq C e^{2\lambda_2'N}, \]
  for $C > 0$ (independent of $N$). In addition the $\rd\ln
  |\mathbf{x}|$ component of $u$ satisfies $|a_0|^2 e^{2(2N+1)}
  \leq C e^{2\lambda_2'N}$, and so
  $$ |a_0| \leq C e^{(\lambda_2'-2)N}. $$
  Recall the norm in \eqref{eq:u.norm}.
   Proposition~\ref{prop:limitu10}, Proposition~\ref{prop:dhest11} and \eqref{eq:1.2N+1.est} imply
  that
  \begin{equation}\label{eq:u210}
    \Vert \tilde{u}(2)\Vert \leq Ce^{-\lambda_1'N}, \quad \Vert
   \tilde{u}(2N+2)\Vert \leq C e^{\lambda_2'N}.
  \end{equation}
 
  Let $\kappa > 0$. We now have the following using \eqref{eq:u210} together with
  Lemma~\ref{lem:3ann1}. 
  \begin{claim}\label{claim:kappa}
 For $N$ sufficiently large (depending on $\kappa$) we have  $\Vert \tilde{u}(N-1)\Vert \leq
  \kappa$. 
 \end{claim}

\begin{proof} 
 To see this note that we have two possibilities.
  \begin{itemize}
  \item If $\Vert \tilde{u}(k+1)\Vert \leq e^{\lambda_1}\Vert \tilde{u}(k)\Vert$ for
    $k=2,\ldots, N-2$, then we have
     $$ \Vert \tilde{u}(N-1)\Vert \leq e^{(N-3)\lambda_1} \Vert \tilde{u}(2)\Vert \leq
         C e^{(N-3)\lambda_1 - \lambda_1' N}. $$
       Since $\lambda_1' > \lambda_1>0$ 
       this implies that $\Vert
       \tilde{u}(N-1)\Vert \leq \kappa$ if $N$ is chosen sufficiently large. 
  \item If $\Vert \tilde{u}(k+1)\Vert \geq e^{\lambda_1}\Vert \tilde{u}(k)\Vert$ for
    some $k\leq N-2$, then by Lemma~\ref{lem:3ann1} we have $\Vert
    \tilde{u}(k+1)\Vert \geq e^{\lambda_2}\Vert \tilde{u}(k)\Vert$ for $k=N-1, \ldots,
    2N+1$. This implies
    $$ \Vert \tilde{u}(N-1)\Vert \leq e^{-(N+3)\lambda_2} \Vert \tilde{u}(2N+2)\Vert
      \leq Ce^{-(N+3)\lambda_2 + \lambda_2'N}. $$
    Since $\lambda_2 > \lambda_2'>0$, we have $\Vert \tilde{u}(N-1)\Vert\leq
    \kappa$ if $N$ is sufficiently large.\qedhere
  \end{itemize}
  \end{proof} 
\noindent Given Claim~\ref{claim:kappa}, let  us assume therefore that $N$ is large enough that $\Vert
  \tilde{u}(N-1)\Vert \leq \kappa$. The estimates in
  Proposition~\ref{prop:dhest11} then imply that we have pointwise bounds
  \[\label{eq:est.uN} |\tilde{u}(N)|^2, |\nabla \tilde{u}(N)|^2 \leq C \kappa^2 e^{|\mathbf{x}|^2 / 4p}, \]
  for some $C>0$ and  $p > 1$. The logarithmic component of $u$
  also satisfies
  $$ \Big| e^Na_0 \rd\ln |\mathbf{x}|\Big| \leq C e^{(\lambda_2'-1)N}
    |\mathbf{x}|^{-1} \leq \kappa |\mathbf{x}|^{-1}, $$
  for large $N$ (we can assume that $\lambda_2' < 1$). 
 
  We now use the local smooth convergence of the
  $d_i^{-1}u_i$ to $u$. For any fixed compact set $K\subset
  \mathbb{C}^2\setminus\{\mathbf{0}\}$ this implies that, as $i\to\infty$, the functions
  $d_i^{-1}d_V$ and $d_i^{-1}(\theta - \theta_V)$ on $M^i_{N}$
  converge to $|u |$ and $\rd^*\tilde{u}$ on $V\cap K$. Using the
  estimates \eqref{eq:est.uN} and the fact that $p>1$ it follows that
  for a given $K$, if we choose $i$ sufficiently large (depending on
  $K, \kappa$), we have (note that we can assume that $d_i^{-2} d_V^2$ differs
      from $|u|^2$ by at most $\kappa$ on $K$ for large $i$)  
      \begin{equation}\label{eq:intK20}
    \begin{aligned}
      \int_{M^i_{N}\cap K} (|\mathbf{x}|^2 d_V^2 + |\theta - \theta_V|^2)
    e^{-|\mathbf{x}|^2/4} &\leq d_i^2\kappa^2\\
    &\quad  +  d_i^2 \int_{V\cap K}
    ( |\mathbf{x}|^2 |u|^2 + 
    |\rd^*\tilde{u}|^2) e^{-|\mathbf{x}|^2/4} \\
    &\leq d_i^2\kappa^2 + C d_i^2 \kappa^2 \int_V (|\mathbf{x}|^2 + 1) e^{\frac{|\mathbf{x}|^2}{4p}
      - \frac{|\mathbf{x}|^2}{4}} \\
      &\leq C d_i^2 \kappa^2. 
  \end{aligned}
\end{equation}

  We will now apply part (3) of Lemma~\ref{lem:IVest} to estimate
  $I_V(M^i_{N})$ in terms of $I_V(M^i_{N-1})$ together with
  the integral bound \eqref{eq:intK20} for suitable $K$. Note first
  that by \eqref{eq:DVa1} and part (1) of Lemma~\ref{lem:IVest} we have
    $$I_V(M^i_{N-1}) = D_V(M^i_{N-1}) \leq C_Nd_i\, ,$$ 
    for an $N$-dependent constant, while $I_V(M^i_{N}) =
  D_V(M^i_{N})=d_i$.
  Let $\gamma > 0$. Assuming $d_i >0$ is sufficiently small (depending on $\gamma$), from
  Lemma~\ref{lem:IVest} we have a compact set $K_\gamma\subset\mathbb{C}^2\setminus\{\mathbf{0}\}$ such that the
  integral estimate \eqref{eq:intK20} on $K = K_\gamma$ implies
  \begin{equation}\label{eq:diIV}
    d_i = I_V(M^i_{N}) \leq C( d_i\kappa + \gamma C_Nd_i).
  \end{equation}
  We first choose $\kappa$ such that $C\kappa <
  1/4$. The choice of $\kappa$ determines an $N_0$, such that for $N >
  N_0$ we have the estimate $\Vert \tilde{u}(N-1)\Vert \leq \kappa$ from Claim~\ref{claim:kappa}. Choosing $N > N_0$ then determines the constant $C_N$, and we
  choose $\gamma$ such that $CC_N\gamma < 1/4$. This choice determines
  the set $K_\gamma$, and then for sufficiently large $i$ we have the
  estimate \eqref{eq:intK20} on $K=K_\gamma$. For such large $i$ the
  inequality \eqref{eq:diIV} holds, and it implies $d_i \leq d_i/2$,
  which is a contradiction as $d_i>0$.
\end{proof}

\section{Decay estimates}\label{sec:decay}

We first define a variant of the excess from Definition~\ref{dfn:excess} and show that it satisfies a monotonicity formula. 

\begin{definition} Recall the excess $\mathcal{A}(M)$ from Definition \ref{dfn:excess}. For any $\alpha > 0$ we let 
$$ \mathcal{A}_\alpha(M) = |\mathcal{A}(M)|^{\alpha-1}
  \mathcal{A}(M), $$
i.e.~$|\mathcal{A}_\alpha| = |\mathcal{A}|^\alpha$, but
$\mathcal{A}_\alpha$ has the same sign as $\mathcal{A}$. 
\end{definition}

\begin{lemma}\label{lem:Aepschange}
  For any $\tau_1 < \tau_2$ and for $\alpha \in (0,1)$ we have
  $$ \mathcal{A}_\alpha(M_{\tau_1}) -
    \mathcal{A}_\alpha(M_{\tau_2}) \geq
    \alpha\int_{\tau_1}^{\tau_2}| \mathcal{A}(M_s)|^{\alpha-1}
    \int_{M_s} \left( 2|\mathbf{H}|^2 + \left| \mathbf{H} +
        \frac{\mathbf{x}^\perp}{2}\right|^2\right)
    e^{-|\mathbf{x}|^2/4}\, ds. $$
\end{lemma}
\begin{proof}
  From Huisken's monotonicity formula we know that
  $\mathcal{A}_\alpha(M_s)$ is monotonically decreasing with $s$,
  being the infimum of a family of decreasing functions as we
  vary $\theta_0$ in the definition of $\mathcal{A}$ in \eqref{eq:AMdefn}. In particular
  $\mathcal{A}_\alpha(M_s)$ is differentiable almost everywhere, and
  at these points
  the derivative satisfies
  $$ \frac{d}{ds} \mathcal{A}_\alpha(M_s) \leq - \alpha
    |\mathcal{A}(M_s)|^{\alpha-1} \int_{M_s} \left( 2|\mathbf{H}|^2 + \left| \mathbf{H} +
        \frac{\mathbf{x}^\perp}{2}\right|^2\right) e^{-|\mathbf{x}|^2/4}. $$
  The required inequality follows by integrating with respect to
  $s\in [\tau_1,\tau_2]$. 
\end{proof}

The main technical result of this section is the
following. Recall that we defined $\mathcal{V}'\subset\mathcal{V}$ in the
same way as $\mathcal{V}$ in Definition \ref{dfn:V}, just with the constant $c_1/2$ instead of $c_1$ measuring closeness to the fixed pair of planes $V_0$. 
\begin{prop}\label{prop:decay10}
  There are $\epsilon_0,   C, N_1 > 0$ and $\alpha\in(0,1)$  such that if  $N > N_1$ is an
  integer,  $D_V(M_0) < \epsilon_0$, $M_1$ satisfies Condition \eqref{cond:ast}, and the flow
  exists for $\tau\in [-1, 3N^2+2]$ satisfying Condition \eqref{cond:ddag},
  then we have the following. If
  $V\in \mathcal{V}'$, then there is a $V' \in
  \mathcal{V}$ satisfying $d(V, V') \leq CD_V(M_0)$ together with one
  of the following conditions: 
  \begin{itemize}
\item[(i)] $D_{V'}(M_{N}) \leq \frac{1}{2} D_V(M_0)$,

  \item[(ii)] $D_{V'}(M_{N}) \leq  \mathcal{A}_\alpha(M_{N-3}) -
    \mathcal{A}_\alpha(M_{N}) $,
    \item[(iii)] $D_{V}(M_{N^2}) \geq e^{\lambda_1'N^2} D_{V}(M_0)$,
  \end{itemize}
  where $\lambda_1'$ is given by Proposition~\ref{prop:3ann2}.
\end{prop}
\begin{proof}
  We prove the result by contradiction. Suppose that for some large
  integer $N$ we have a sequence of flows $M^i_\tau$ for $\tau\in [-1,3N^2+2]$, and 
  $D_{V_i}(M^i_0) = d_i \to 0$ for some $V_i\in \mathcal{V}'$. We will show that
  if none of the conditions (i)--(iii) hold, then we reach a
  contradiction if $N$ is sufficiently large. First note that up to choosing a
  subsequence we can replace the sequence $V_i$ by a single $V$. 

   Using Proposition~\ref{prop:Dgrowth1} and Lemma~\ref{lem:condast} we
  know that for sufficiently large $i$ the surfaces $M^i_\tau$ satisfy
  Condition \eqref{cond:ast} for $\tau\in [1, 3N^2+2]$. 
  Using
  Proposition~\ref{prop:limitu10} we can find compact sets $K_i$
  exhausting $\mathbb{C}^2\setminus\{\mathbf{0}\}$ so
  that $M^i_s$ is the graph of $u_i(s)$ over $V\cap K_i$ for $s\in [1,
  3N^2]$.
  In addition, up to choosing a subsequence, the rescaled
  functions $d_i^{-1}u_i$ converge locally smoothly to a solution $u$
  of the drift heat equation \eqref{eq:dh20} on $[1,3N^2]\times V\setminus\{\mathbf{0}\}$.  
  
 Let $u_0$ be the static component of $\tilde u$ in the decomposition 
 \eqref{eq:udec20}, corresponding to the
  kernel of the drift Laplacian 
  $$\mathcal{L}_0+\frac{1}{2}=\Delta + \frac{1}{2}( 1 - \mathbf{x}\cdot
  \nabla).$$
   We can thus write $u_0 = (\rd f_1, \rd f_2)$ for
  homogeneous degree 2 functions $f_1, f_2$ on $\mathbb{R}^2$ with respect to the splitting $V=P_1\cup P_2$ into a pair of planes. We 
then let
  $a_j = \rd^*\rd f_j$ for $j=1,2$ (which are constants), set $a =
  \frac{1}{8}(a_1-a_2)$ and define
  \begin{align*}
      u_{00} &= \Big( \rd (f_1+a|\mathbf{x}|^2), \rd (f_2 - a|\mathbf{x}|^2) \Big), \\
      u_{01} &=  \Big( \rd(-a|\mathbf{x}|^2), \rd(a|\mathbf{x}|^2) \Big),
    \end{align*}
    again using the splitting $V=P_1\cup P_2$.  Note that 
 \begin{align*}\rd^*u_{00}&=(a_1-4a,a_2+4a)=(\textstyle\frac{1}{2}(a_1+a_2),\textstyle\frac{1}{2}(a_1+a_2)),\\
 \rd^*u_{01}&=(4a,-4a)=(\textstyle\frac{1}{2}(a_1-a_2),\textstyle\frac{1}{2}(a_2-a_1)).
  \end{align*}
We have therefore decomposed $u$ as 
  \begin{equation}\label{eq:udec30}
    u =u_0+u^{\perp}= u_{00} + u_{01} + u^\perp. 
  \end{equation}

  The purpose of \eqref{eq:udec30} is that we can deform the cone $V$
  in the direction of $u_{00}$ while keeping it special Lagrangian,
  since $\rd^*u_{00} = \frac{1}{2}(a_1+a_2)$ on both planes. 
  The directions $u_{01}$  however correspond to deformations of $V$ into 
  non-special Lagrangian directions, whenever $a_1\neq a_2$. 
  More precisely, let $V'_i$ denote the graph of $d_i u_{00}$ over
  $V$. Using Lemma~\ref{lem:DVV'comp} we have
  $D_{V'_i}(M^i_0)\leq Cd_i$, and we can argue as above to write
  $M^i_s$ as the graph of $u'_i(s)$ over larger and larger subsets of
  $V_i'$. The rescaled functions $d_i^{-1} u_i'$ then converge to $u'
  = u_{01} + u^\perp$ in the decomposition \eqref{eq:udec30}. 

  Fix a $\kappa > 0$. We first use (2) in Lemma~\ref{lem:3ann1} to
  show the following.
  
  \begin{claim}\label{claim:uperp} If $N$ is sufficiently large (depending on $\kappa$), then either $\Vert u^\perp(N-4)\Vert \leq \kappa$  or (iii)
  holds.
  \end{claim} 

\begin{proof}  Recall that we consider $u^\perp$ for $\tau \in [1,3N^2]$ and that $u^\perp$ has no homogeneous degree zero component. Suppose that $\Vert u^\perp(N-4)\Vert > \kappa$. If $N$
  is sufficiently large, there must be some $k < N-4$ for which $\Vert
  u^\perp(k+1)\Vert > e^{-\lambda_1}\Vert u^\perp(k)\Vert$, since
  otherwise we would have $\Vert u^\perp(N-4)\Vert \leq
  Ce^{-(N-5)\lambda_1}$, which for large $N$ is less than $\kappa$.
  Lemma~\ref{lem:3ann1} (2) now implies that $\Vert
  u^\perp(k+2)\Vert \geq e^{\lambda_1}\Vert u^\perp(k+1)\Vert$. However, Lemma~\ref{lem:3ann1} (1) then implies that 
    $$
   \Vert
  u^\perp(l+1)\Vert \geq e^{\lambda_2}\Vert u^\perp(l)\Vert \quad \text{for } l \geq k+2\, .$$
It follows that $\Vert u^\perp(N-3)\Vert \geq
  e^{\lambda_1}\kappa$ and $\Vert u^\perp(l+1)\Vert \geq
  e^{\lambda_2} \Vert u^\perp(l)\Vert$ for all $k\geq N-3$. Iterating
  this, we find that
  $$ \Vert u^\perp(N^2)\Vert \geq e^{\lambda_1 +
      (N^2+3-N)\lambda_2}\kappa. $$
  We can split $u^\perp$ further, writing
  \[\label{eq:split}
   u^\perp =  e^s a_0 \rd\ln |\bx| + \tilde u^\perp\]
  such that  $\tilde u^\perp$ is a smooth solution to the drift heat equation on $V$. Ecker's log-Sobolev inequality~\cite{Ecker.logSobolev} then  implies that there is $p>2$ such that
  $$\left(\int_V(\tilde u^\perp(N^2))^p e^{-|\bx|^2/4}\right)^{1/p} \leq C \|\tilde u^\perp (N^2-1)\| \leq C \|\tilde u^\perp (N^2)\|\, .$$
Combined with the estimates from Proposition \ref{prop:dhest11}  (2) for $\tilde u^\perp$ and \eqref{eq:split}, we deduce  that there is $r_0>0$ such that
$$ \int_{V\cap (B_{1/r_0}\setminus B_{r_0})}(u^\perp(N^2))^2 e^{-|\bx|^2/4} \geq \frac{1}{2} \|u^\perp(N^2)\|^2\, .$$
The definition of $D_V$ then implies that for sufficiently large
  $i$ we have
  $$ D_V(M^i_{N^2}) \geq C^{-1} d_i e^{\lambda_1 +
      (N^2+3-N)\lambda_2}\kappa \geq e^{\lambda_1' N^2} d_i, $$
  for sufficiently large $N$, since $\lambda_2 > \lambda_1'$.  Hence (iii) holds.  \end{proof}

  Let us suppose from now on that (iii) does not hold.  Suppose that $\Vert u_{01}\Vert = \kappa_1$ for some
  small $\kappa_1 \geq 0$ (note that $u_{01}$ is $s$-independent). We
  have $\kappa_1 < C$ for a uniform constant $C$. We
  also assume that for a given small $\kappa > 0$,
  $N$ is chosen large enough  so that $\Vert
  u^\perp(N-4)\Vert \leq \kappa$ by Claim~\ref{claim:uperp}. We now show the following, from which we will deduce that (i) holds if $\kappa_1$ is sufficiently small.
\begin{claim}\label{claim:kappa1}
If $i$ is
  sufficiently large,
  \begin{equation}\label{eq:DV'bound10}
    D_{V'_i}(M^i_{N-3}) \leq C(\kappa_1+\kappa)d_i.
  \end{equation}
\end{claim}  
\begin{proof}
By Proposition~\ref{prop:Dgrowth1}
  together with Lemma~\ref{lem:DVV'comp} we
  have $D_{V'_i}(M^i_{N-4}) \leq C_Nd_i$ for $C_N$ depending on $N$. As
  in the proof of Proposition~\ref{prop:3ann2} we can write $u = e^sa_0\rd \ln |\mathbf{x}| + \tilde{u}$ such that we have pointwise
  bounds of the form
  \begin{equation}\label{eq:uL31}
    |\tilde u(s)|^2, |\nabla^k \tilde u(s)|^2 \leq C (\kappa_1 + \kappa)^2
    e^{|\mathbf{x}|^2/4p}, \text{ for } k=1,2,3 \text{ and }s\in [N-3,N],
    \end{equation} for some $C>0$ and $p>1$.
  We can then use the local
  smooth convergence of $d_i^{-1} u_i\to u$ together with property (3) of
  Lemma~\ref{lem:IVest} to ensure that \eqref{eq:DV'bound10}
  holds.
\end{proof}
  
   It follows from Claim~\ref{claim:kappa1}, using Proposition~\ref{prop:Dgrowth1}, that for a
  larger $C$ we have
  $D_{V'_i}(M^i_s) \leq C(\kappa_1 + \kappa)d_i$ for $s\in [N-3, N]$.
  If now $C\kappa_1 < 1/4$, then by choosing $\kappa = C^{-1}/4$ we
  will have $D_{V'_i}(M^i_N)\leq \frac{1}{2} d_i$, i.e.~(i) holds for
  large enough $i$. 
  
  We therefore assume further that (i) does not hold, $\kappa_1=\|u_{01}\|\geq C^{-1}/4$,
  and we choose $\kappa < \kappa_1$. In particular, the value of
  $\rd^*u_{01}$ on the two planes differs by at least $C^{-1}$ for some
  $C > 0$. 

  In the rest of the proof our goal
  is to show that if $\kappa$ is sufficiently small (i.e. $N$ is
  large), and (ii) also does not hold, then we get a contradiction. The
  basic idea is that in this case the flow $M^i_s$ for $s\in [N-3,N]$ would have distance of
  order $\kappa d_i$ from a pair of planes whose Lagrangian angles
  differ by $\kappa_1 d_i$. If $M^i_N$ is connected and $\kappa \ll
  \kappa_1$, then as in \cite[Theorem B]{Neves:zero-maslov}, one might
  expect that this leads to a contradiction. The difficulty is that we
  need a quantitative version of this idea, which works uniformly as
  $d_i\to 0$. 

  We are assuming that (ii) fails, therefore
  $$ \mathcal{A}_\alpha(M^i_{N-3}) -
    \mathcal{A}_\alpha(M^i_N) < D_{V'_i}(M^i_N) \leq
    C_Nd_i. $$
  Using Lemma~\ref{lem:Aepschange} this implies
  \begin{equation}\label{eq:Aeps10}
    \int_{N-3}^N |\mathcal{A}(M^i_s)|^{\alpha-1} \int_{M^i_s}
    (|\mathbf{H}|^2 + |\mathbf{x}^\perp|^2) e^{-|\mathbf{x}|^2/4}\, ds \leq \alpha^{-1} Cd_i.
    \end{equation}
  From Proposition~\ref{prop:Aest} there is a small $\alpha_1>0$ so that 
  $$ |\mathcal{A}(M^i_s)| \leq D_{V'_i}(M^i_s)^{1+\alpha_1}\text{ for } s\in [N-3, N]. $$
Using this in \eqref{eq:Aeps10} for $\alpha\in(0,1)$ we get 
  $$  \int_{N-3}^N \int_{M^i_s} (|\mathbf{H}|^2 + |\mathbf{x}^\perp|^2) e^{-|\mathbf{x}|^2/4}\, ds
    \leq C_\alpha d_i^{1+ (1-\alpha)(1+\alpha_1)}, $$
  for an $\alpha$-dependent constant $C_\alpha>0$. For $\alpha$
  sufficiently small and $0<2\alpha_2<\alpha_1/2$, we see that 
  $$1+(1-\alpha)(1+\alpha_1) > 2 +
  \frac{\alpha_1}{2}>2+2\alpha_2.$$
Therefore, for sufficiently large $i$ we have 
$$  \int_{N-3}^N \int_{M^i_s\cap B_2} |\mathbf{H}|^2 + |\mathbf{x}^\perp|^2 \, ds \leq
    d_i^{2+2\alpha_2}, $$
where we have removed the Gaussian weight by restricting to $M^i_s\cap B_2$.

  Let $\sigma > 0$ be small, to be chosen later, independent of $i$.
  We can find  $s^i_1, s^i_2\in [N-2, N-1]$ such that 
\begin{equation} \label{eq:sidiff} \frac{\sigma}{2} < s^i_2 -
  s^i_1 < \sigma, 
\end{equation} 
 and, in addition, at $s^i_j$ for $j=1,2$ we have
  \begin{equation}\label{eq:Hint30}
    \int_{M_{s^i_j}\cap B_2} |\mathbf{H}|^2 + |\mathbf{x}^\perp|^2 \leq d_i^{2+\alpha_2}
  \end{equation}
  for $i$ sufficiently large (depending on $\sigma$ as well). 

  Note that because of the bound $D_{V_i'}(M^i_{N-3}) \leq C_Nd_i$ and Proposition \ref{prop:graphical} we have that $M^i_s$
  has good graphicality
  over $V_i'$
  on the annulus 
  $B_{R_{C_Nd_i}} \setminus B_{C_0d_i^{1/2}}$ for $s\in [N-2,
  0]$.
  This graphicality and Condition \eqref{cond:ast} implies that the
  integral of the Liouville form $\lambda$ vanishes on any loop in
  $M^i_s\cap B_{R_{C_Nd_i}}$, so we can define  primitives
  $\beta$ satisfying $\rd\beta = \lambda$ on $M^i_s\cap B_{R_{C_Nd_i}}$. Restating
 \cite[Proposition 6.1]{Neves:zero-maslov} for the rescaled flow, we can
  choose the primitives $\beta$ along the flows $M^i_s$ such that $e^{-s}(\beta +
  2\theta)$ satisfies the drift heat equation \eqref{eq:dh20}.

  Our next goal is to estimate $\beta$ at the times $s^i_1,
  s^i_2$.   First we consider what happens on the ball $B_{C_0d_i^{1/2}}$. For simplicity we write $M$ for $M^i_{s^i_1}$ or
  $M^i_{s^i_2}$. 

\begin{claim}\label{claim:beta.1}
There is a constant $\beta_0$ (depending on $s,i$) such that $|\beta-\beta_0|<\kappa d_i$ on $M\cap B_{C_0d_i^{1/2}}$.
\end{claim}  
  
 \begin{proof} 
  Let  $\tilde{M} = d_i^{-1/2} M$  and let $\tilde{\beta}$ on $\tilde{M}$ be given by
  $\tilde{\beta}(p) = d_i^{-1}\beta(d_i^{1/2}p)$. Then $\tilde{M}$ is
  a connected, almost calibrated Lagrangian, with uniform area bounds,
  satisfying
  $$ \int_{\tilde{M}\cap B_{C_0}} |\nabla \tilde{\beta}|^2 = d_i^{-2} \int_{M\cap
      B_{C_0d_i^{1/2}}} |\nabla \beta|^2 \leq d_i^{\alpha_2}, $$
  using \eqref{eq:Hint30} and $|\nabla\beta| = |\mathbf{x}^\perp|$. (Recall that $\alpha_2>0$.) We also have a uniform bound 
  $$|\nabla\tilde{\beta}(p)|_{\tilde{M}} =
  d_i^{-1/2} |\nabla\beta(d_i^{1/2}p)|_M \leq C$$ 
  for $p\in B_{C_0}$,  
  so  \cite[Lemma 3.7]{Neves:survey}, together with the connectivity
  assumption in Condition \eqref{cond:ast}, implies that $\mathrm{osc}\,
  \tilde{\beta} \to 0$ as $i\to\infty$. It follows that on
  $M\cap B_{C_0d_i^{1/2}}$ we have $\mathrm{osc}\, (\beta) < \kappa d_i$
  for sufficiently large $i$. Setting $\beta_0 = \beta(q)$ for some
  $q\in B_{C_0d_i^{1/2}}$ yields the claim. 
\end{proof}  

  Next we extend this pointwise bound on $B_{C_0d_i^{1/2}}$ to an integral bound on the (larger) ball $B_1$, by using the
  integral estimate for $|\nabla\beta|^2$ from \eqref{eq:Hint30} again.
  
  \begin{claim}\label{claim:beta.2}
  For sufficiently large $i$,
  \begin{equation}
  \label{eq:betaL2}
   \int_{M\cap B_1} |\beta-\beta_0|^2 \leq 9
    \kappa^2
    d_i^2\, .
    \end{equation}
  \end{claim} 
  
  \begin{proof} Since on the annulus
  $B_1\setminus B_{C_0d_i^{1/2}}$ the surface $M$ has good graphicality 
  over $V'_i$, we can view $\beta$ as a function $b$ on two copies of
  the annulus $B_1 \setminus B_{C_0d_i^{1/2}}\subset \mathbb{R}^2$, where we use polar coordinates $r,\phi$.   Using
  \eqref{eq:Hint30} we then have
  $$ \int_{B_1 \setminus B_{C_0d_i^{1/2}}} |\nabla b|^2 \leq
    2d_i^{2+\alpha_2} $$
  for sufficiently large $i$. For each $r\in [C_0d_i^{1/2}, 1]$ and
  $\phi\in [0,2\pi]$ we have
  $$ |b(r, \phi)-\beta_0| \leq \kappa d_i + \int_{C_0d_i^{1/2}}^r |\nabla
    b|(s,\phi)\, ds, $$
    where $\beta_0$ is the constant given by Claim~\ref{claim:beta.1}.  Therefore,
  \begin{equation}\label{eq:intb1}
    \begin{aligned}  \int_0^{2\pi} \int_{C_0d_i^{1/2}}^1 &|b(r,
      \phi) - \beta_0|^2\, r\,dr\,d\phi\\
      & \leq 
      2\pi\kappa^2 d_i^2+ 2 \int_0^{2\pi}
      \int_{C_0d_i^{1/2}}^1 \left( \int_{C_0d_i^{1/2}}^r |\nabla
        b|(s,\phi)\,ds\right)^2\, r\, dr\, d\phi.
    \end{aligned} \end{equation}
  Note that using H\"older's inequality we have
 \begin{align*} \int_0^{2\pi}
      \int_{C_0d_i^{1/2}}^1 \Bigg( \int_{C_0d_i^{1/2}}^r |\nabla
         b| &(s,\phi)\,ds\Bigg)^2\, r\, dr\, d\phi\\
         &\leq C
      \int_0^{2\pi}\int_{C_0d_i^{1/2}}^1 \int_{C_0d_i^{1/2}}^r |\nabla
      b|^2(s,\phi)\,s\, ds\, r^2dr\, d\phi \\
      &\leq C \int_0^{2\pi} \int_{C_0d_i^{1/2}}^1 |\nabla b|^2
      (s, \phi)\, ds\, d\phi \\
      &\leq C d_i^{2+\alpha_2} \leq \kappa^2 d_i^2, 
    \end{align*}
  once $i$ is large enough (so that $d_i$ is small). 
The result follows for $i$ sufficiently large, combined with Claim \ref{claim:beta.1} and the uniform area ratio bounds.
\end{proof}
 
We now show that we can get a similar pointwise estimates to Claim~\ref{claim:beta.1} on $K_i\cap M$, where we recall the compact sets $K_i$ given at the start of the proof.   

\begin{claim}\label{claim:beta.3}
Up to replacing the compact sets $K_i$ by smaller sets (still exhausting
    $\mathbb{C}^2\setminus \{\mathbf{0}\}$ in the limit as $i\to\infty$), there is a constant $C>0$ and $p>1$ such that 
        \begin{equation}\label{eq:betaKbound}
      |\beta - \beta_0| \leq C\kappa d_i e^{|\mathbf{x}|^2 / 8p} \text{ on } K_i\cap M.
    \end{equation}
\end{claim}

\begin{proof}  Recall the decomposition of $u$ in \eqref{eq:udec30}.   
  We now  show  that  $\nabla\beta$ is of order $\kappa d_i$ on
    compact sets away from $0$,
    because the term $u_{01}$ in \eqref{eq:udec30}  does not contribute to
    $\mathbf{x}^\perp$, being homogeneous of degree 1. More precisely, recall
    that $M^i_s$ is the graph of $u_i(s)$ over $K_i\cap V$   and note that  $d_i^{-1} \mathbf{x}^\perp$ on $M^i_s$ converges locally smoothly as $i\to\infty$ 
    to $u - \mathbf{x}\cdot \nabla u$ on $V\setminus\{\mathbf{0}\}$. Since $u_{00},
    u_{01}$ have degree 1, it follows that they have no contribution to $u-\mathbf{x}\cdot\nabla u$.  Therefore,  as $i\to\infty$ we have
    $d_i^{-1} |\nabla \beta| \to |u^\perp - \mathbf{x}\cdot \nabla
      u^\perp|$ locally smoothly. 
    At the same time, by Claim~\ref{claim:uperp} we have $\Vert u^\perp(N-4)\Vert
    \leq C\kappa$ and so, by Proposition~\ref{prop:dhest11}, for
    $s\in [N-3,N]$ we also have pointwise bounds
    \begin{equation}\label{eq:up1} |u^\perp|, |\nabla u^\perp| \leq C
      \kappa e^{|\mathbf{x}|^2/8p} 
      \end{equation}
      for some $C>0$ and $p>1$.
    Returning to the setting where $M = M^i_{s^i_j}$ for $j=1,2$ and 
    using \eqref{eq:betaL2} we
    can integrate the estimate we have for $|\nabla \beta|$ to find
    that up to replacing $K_i$ by smaller sets and decreasing $p$ we get \eqref{eq:betaKbound}.
    \end{proof}

    We also need a more global estimate for $\beta$ and $\theta$, up to the good
    graphicality radius $R_{Cd_i}$ (from Definition~\ref{dfn:Rd}) on $M^i_s$ for $s\in [N-3,N]$.
  
\begin{claim}\label{claim:beta.4}  Recall Definition~\ref{dfn:Rd}.
There is $C>0$ and $p>p_0>1$
such that for $s\in [N-3,N]$
we have 
    \[\label{eq:betabound2}
      |\beta - \beta_0| &\leq C d_i e^{|\mathbf{x}|^2 / 8p} \text{ on }
      M^i_s\cap B_{R_{C_Nd_i}}\setminus B_{1/2},\\
      |\theta - \theta_{V_i'}| &\leq Cd_i e^{|\mathbf{x}|^2/8p} \text{ on }
      M^i_s\cap B_{R_{C_Nd_i}}\setminus B_{1/2},
    \]
    once $i$ is sufficiently large.
\end{claim}    

\begin{proof}    
 By the smooth convergence of $d_i^{-1}u_i \to u$ on
  the annulus $B_2\setminus B_{1/2}$, together with the bounds
  \eqref{eq:uL31},  we have estimates
$$ |\nabla\beta|, |\Delta \beta|, |\theta-\theta_{V'}|, |\nabla\theta|, |\Delta\theta| \leq
  Cd_i $$
on $M^i_s\cap B_2\setminus B_{1/2}$ for $s\in [N-3,N]$.
Using the evolution equation for $\beta$ and
\eqref{eq:betaKbound}
it then follows that 
\[\label{eq:beta.B2.bound}
|\beta - \beta_0|
\leq Cd_i\quad\text{ on }M^i_s\cap B_2\setminus B_{1/2}
\] for $s\in [N-3, N]$.
To extend
this estimate out to distance $R_{C_Nd_i}$ we first observe that, by
    Proposition~\ref{prop:graphical}, on  
    $B_{R_{C_Nd_i}}\setminus B_1$ the surface $M^i_s$ is the graph of $u_i$
    over $V_i'$ satisfying $|u_i|, |\nabla u_i| \leq C d_i
    e^{|\mathbf{x}|^2/8p}$ for some $p > p_0>1$. Moreover, by the definition of
    $R_{C_Nd_i}$ in \eqref{eq:Rddefn} and the fact that $p>p_0$ we know that $d_i e^{
      R_{C_Nd_i}^2/8p} \to 0$ as $i\to \infty$. Since the leading order
    term in $\mathbf{x}^\perp$ is $u_i - \mathbf{x}\cdot \nabla u_i$ for $i$ large, we see that on
    $M^i_s\cap B_{R_{C_Nd_i}}\setminus B_1$ for $s\in [N-3,N]$ we have (decreasing $p>p_0$ if necessary)
    $$ |\nabla \beta| = |\mathbf{x}^\perp| \leq C d_i e^{|\mathbf{x}|^2/8p} $$
    for sufficiently large $i$. Integrating this and using
    our bound \eqref{eq:beta.B2.bound} on $B_2\setminus B_{1/2}$ implies that for an even smaller $p>p_0 > 1$ we have the estimate \eqref{eq:betabound2} for $i$ sufficiently large and $s\in [N-3,N]$.
    The bound for $\theta$ in \eqref{eq:betabound2} follows similarly,
    since to leading order $\theta -\theta_{V_i'}$ is given by 
    $\rd^*u_i$ for $i$ large. 
    \end{proof}

Recall the times $s^i_2>s^i_1$ which satisfy \eqref{eq:sidiff} depending on some   small $\sigma>0$, which we are free to choose.    We will now use that $e^{-s}(\beta + 2\theta)$ and $\theta$
    satisfy the drift heat equation \eqref{eq:dh20} to derive pointwise estimates in
    $B_2$ for $\theta$ at time $s^i_2$ in terms of an integral
    estimate for $\beta$ at time $s^i_1$. 

\begin{claim}    
    Let
    $$ h = e^{s^i_1 -s}(\beta - \beta_0 + 2(\theta -\theta_{V_i'})) - 2(\theta -\theta_{V_i'})$$ on $M^i_s$
    so that $h(s^i_1) = \beta - \beta_0$. There is some $C>0$ independent of $i$ such that
    \[\label{eq:h.B2.bound} \sup_{M^i_s\cap B_2} h^2 \leq C\kappa^2 d_i^2 \quad \text{ at } s=s^i_2  \]
    and hence
    \[\label{eq:theta.B2.bound}
    \mathrm{osc}\,\theta\leq C\kappa d_i\quad\text{ on }M^i_{s^i_2}\cap B_2\setminus B_{1/2}.
    \]
    \end{claim}
  
  \begin{proof}  
     Since we have only defined $\beta$ on
    the ball $B_{R_{Cd_i}}$ we need to incorporate a cutoff
    function. In fact even if $\beta$ were defined globally we would
    need such a cutoff if we do not assume that $\beta$ has uniform polynomial   growth bounds.
    
    To that end, let $\chi_0:[0,\infty) \to\mathbb{R}$ denote a cutoff function with $\chi_0(t)=1$
    for $t < (R_{C_Nd_i}-1)^2$ and $\chi_0(t)=0$ for $t > R_{C_Nd_i}^2$.
    We
    can arrange that $\chi, \chi', \chi''$ are uniformly bounded
    independently of $i$. Let $\chi(x) = \chi_0(|x|^2)$.      Note that $\chi^2 h^2$ is then defined
    globally.

 Recall $\mathcal{L}_0$ given in \eqref{eq:L0}.   Along the rescaled flow (for surfaces) we have
    \begin{equation}\label{eq:x21} (\partial_s - \mathcal{L}_0) |\mathbf{x}|^2
      = |\mathbf{x}^{\rm T}|^2 - 4
      \end{equation}
    and so 
    \begin{equation}\label{eq:x22}
      (\partial_s - \mathcal{L}_0) e^{s^i_1-s}|\mathbf{x}|^2 \leq -2\quad \text{
        for }s\in [s^i_1, s^i_2].
    \end{equation}
  We also have
    $$ (\partial_s - \mathcal{L}_0) \chi = \chi_0'(|\mathbf{x}|^2) (\partial_s -
      \mathcal{L}_0) |\mathbf{x}|^2 - \chi_0''(|\mathbf{x}|^2) |\nabla |\mathbf{x}|^2|^2. $$
    Hence,
    \begin{equation}\label{eq:chiest10}
      |(\partial_s - \mathcal{L}_0) \chi| \leq  C|\mathbf{x}|^2
      \end{equation}
    and $(\partial_s - \mathcal{L}_0)\chi$ is supported on the set where
    $R_{C_Nd_i} -1 < |\mathbf{x}| < R_{C_Nd_i}$.
    
Since $h$ satisfies the drift heat equation \eqref{eq:dh20}, we may compute
    $$(\partial_s - \mathcal{L}_0) (\chi^2 h^2) = 2\chi h^2
        (\partial_s - \mathcal{L}_0)\chi - 2|\nabla \chi|^2 h^2 -
        2\chi^2 |\nabla h|^2 - 8\chi h \nabla \chi\cdot\nabla h.$$
    We use the inequality $|8\chi h\nabla\chi\cdot\nabla h|\leq
    2\chi^2 |\nabla h|^2 + 8 h^2|\nabla \chi|^2$ and the estimate
    \eqref{eq:chiest10} to get
    \begin{equation}\label{eq:chi2h2} (\partial_s - \mathcal{L}_0)
      (\chi^2 h^2) \leq \begin{cases} C 
        |\mathbf{x}|^2 h^2  &\text{ for } R_{C_Nd_i} - 1 < |\mathbf{x}| < R_{C_Nd_i}, \\
        0 &\text{ otherwise.}\end{cases} \end{equation}
By Claim~\ref{claim:beta.4}, there is $p>p_0 >1$
so that we also have the bound
\[\label{eq:h.bound} h^2\leq C d_i^2 e^{R_{Cd_i}^2 / 4p}\quad\text{ for
}    1 <|\mathbf{x}| < R_{Cd_i}.\] 
We now define the function
    $$ \Theta =\Big( e^{s_1^i-s} |\mathbf{x}|^2 -
      e^{-\sigma} (R_{Cd_i}-1)^2\Big)_+, $$
     where $(\ldots)_+$ means the positive part of the function and
      $\sigma$ is the constant in \eqref{eq:sidiff}.
      Using \eqref{eq:x22} and $s\in[s_1^i,s_2^i]$ we have
    that
    \[\label{eq:Theta.evol} (\partial_s - \mathcal{L}_0) \Theta \leq \begin{cases} -2  &\text{
          when }e^{s_1^i-s} |\mathbf{x}|^2 > e^{-\sigma} (R_{Cd_i} -1)^2, \\
        0  &\text{ otherwise, } \end{cases}\]
    in the distributional sense.  Note that by \eqref{eq:sidiff} we have $e^{s-s^i_1-\sigma}<1$ for $s\in [s_1^i,s_2^i]$. 
  We deduce from \eqref{eq:chi2h2}, \eqref{eq:h.bound} and \eqref{eq:Theta.evol} that there is some $C_1>0$ so that 
    \begin{equation}\label{eq:ch10}
      (\partial_s - \mathcal{L}_0) \Big(\chi^2h^2 + C_1d_i^2 R_{Cd_i}^2
      e^{R_{Cd_i}^2/4p} \Theta\Big) \leq 0\quad  \text{ for } s\in [s^i_1,
      s^i_2].
    \end{equation}

    At $s=s^i_1$ we have $\chi^2h^2 = \chi^2(\beta-\beta_0)^2$, and so using
    \eqref{eq:betaL2}, \eqref{eq:betabound2},
    together with the uniform area ratio bounds for $M^i_s$, we can
    ensure that
    \begin{equation}\label{eq:ch11}
      \int_{M^i_{s^i_1}} \chi^2 h^2 e^{-|\mathbf{x}|^2/4}  \leq C\kappa^2
      d_i^2,
      \end{equation}
    once $i$ is large enough. To
    estimate the integral of $\Theta$ at $s=s^i_1$, note that
    $\Theta(\mathbf{x}) = 0$ if $s=s^i_1$ and $|\mathbf{x}| < e^{-\sigma/2}
    (R_{Cd_i}-1)$, which holds if $|\mathbf{x}| < e^{-\sigma}R_{Cd_i}$ once $i$
    is large. For all $\mathbf{x}$ we have $\Theta(\mathbf{x}) \leq
    |\mathbf{x}|^2$ and so by our previous  observation we have
    $$ \int_{M^i_{s^i_1}} \Theta e^{-|\mathbf{x}|^2/4}\leq \int_{M^i_{s^i_1}
          \setminus B_{e^{-\sigma}R_{Cd_i}}} |\mathbf{x}|^2 e^{-|\mathbf{x}|^2/4} \leq C
        e^{-2\sigma} R_{Cd_i}^2 e^{- e^{-2\sigma} R_{Cd_i}^2  / 4}. $$
    If $\sigma$ is chosen sufficiently small, so that $e^{2\sigma} <
    p$ for the $p>1$ in Claim~\ref{claim:beta.4}, see 
    \eqref{eq:h.bound}, then we will have
    \[ \label{eq:final_est}\int_{M^i_{s^i_1}} d_i^2 R_{Cd_i}^2 e^{R_{Cd_i}^2/4p} \Theta\, 
      e^{-|\mathbf{x}|^2/4} \leq C d_i^2 R_{Cd_i}^4  \exp\left(
        \frac{R_{Cd_i}^2}{4p} -
        \frac{R_{Cd_i}^2}{4e^{2\sigma}}\right) \leq \kappa^2 d_i^2   \]
    for sufficiently large $i$. Combining this with \eqref{eq:ch11}
    and using \eqref{eq:ch10} we can apply the monotonicity formula to
    obtain pointwise estimates for $\chi^2h^2$ at $s=s^i_2$. Since
    $s^i_2 > \sigma/2$, and $\sigma$ only depends on $p$ in
    \eqref{eq:final_est},
    we obtain the estimate \eqref{eq:h.B2.bound} for $h^2$.
    
    At $s=s^i_2$ we have
   \begin{equation*}
 \begin{split}
 h &=  e^{s^i_1 -s_2^i}(\beta - \beta_0 + 2(\theta -\theta_{V_i'})) - 2(\theta -\theta_{V_i'})\\
 &= e^{s^i_1 -s_2^i} \beta + 2(e^{s^i_1 -s_2^i} -1) \theta - e^{s^i_1 -s_2^i}(\beta_0 + 2\theta_{V_i'}) + 2 \theta_{V_i'}\, .
\end{split}
\end{equation*}
    and at this time the oscillation of $\beta$ on $B_2\setminus
    B_{1/2}$ is  bounded by $C\kappa d_i$ from
    \eqref{eq:betaKbound} for $i $ large. Noting \eqref{eq:sidiff}, it follows that the oscillation of $\theta$
    is also bounded by $C\kappa d_i$ on this annulus for $i$ sufficiently large as claimed.
    \end{proof}

To complete the proof, recall the decomposition \eqref{eq:udec30}   and that on $B_2\setminus B_1$ we know that
    $M^i_{s^i_2}$ is the graph of $u_i$ over $V'_i$, where $d_i^{-1} u_i
    \to u_{01} + u^\perp$, and so $d_i^{-1}(\theta - \theta_{V'_i})
    \to \rd^*(u_{01} + u^\perp)$. From \eqref{eq:up1} we have
    $|\rd^*u^\perp|\leq C\kappa$ on $B_2\setminus B_1$, while the value of $\rd^*u_{01}$ on the
    two planes differs by at least $C^{-1}$. Therefore the oscillation
    of $\rd^*(u_{01} + u^\perp)$ on $B_2\setminus B_1$ is at least
    $C^{-1} - C\kappa > C^{-1}/2$ if we choose $\kappa$ sufficiently
    small. The oscillation of $\theta$ on $B_2\setminus B_{1/2}$ is
    therefore at least $C^{-1}d_i/4$ for large $i$. This contradicts the bound \eqref{eq:theta.B2.bound} if $\kappa$ is sufficiently small. 
  \end{proof}
  
  \subsection{Closeness to planes}

We next show that if condition (iii) in Proposition~\ref{prop:decay10} 
holds, then we can still arrange that the flow remains close to the original pair of planes
$V_0$ as long as the change in the excess $\mathcal{A}$ is controlled. From now on 
we fix $ N_1>0$ and $\alpha\in(0,1)$ such that
Proposition~\ref{prop:decay10} applies, and we assume that $N>N_1$ is a fixed integer
large enough so that
Proposition~\ref{prop:3ann2} applies to $N^2$ and such that
\[\label{eq:bound-N} e^{-\lambda_1'N^2}<1/2\, .\]  Again recall that $\lambda_1'>0$ is given by Proposition~\ref{prop:3ann2}.
\begin{prop}\label{prop:d20} 
  Let $\delta_1 > 0$. There is an $\epsilon_1 > 0$ depending on
  $\delta_1$ such that, if we have a flow $M_\tau$ satisfying
  Condition \eqref{cond:ddag} for $\tau\in [-1,T+10]$, $M_1$ satisfies Condition
  \eqref{cond:ast}, and
  \begin{enumerate}
  \item $\mathcal{A}_\alpha(M_0) - \mathcal{A}_\alpha(M_T) < \epsilon_1$,
  \item $D_V(M_0) < \epsilon_1$, for some $V\in \mathcal{V}'$, 
  \item  $D_V(M_{N^2}) \geq e^{\lambda_1'N^2} D_V(M_0)$, 
  \end{enumerate}
  then $D_V(M_T) < \delta_1$. 
\end{prop}
\begin{proof}
  Using assumptions (2), (3), if $\epsilon_1$ is sufficiently small
  then we can apply Proposition~\ref{prop:3ann2} to deduce
  $D_V(M_{2N^2}) \geq e^{\lambda_1' N^2} D_V(M_{N^2})$. We can keep
  iterating this estimate, to get
  $$ D_V(M_{iN^2}) \geq e^{\lambda_1' N^2} D_V(M_{(i-1)N^2}) $$
  for $i=3, 4, \ldots, k$, where $k$ is the largest integer
  which still satisfies
  $$ D_V(M_{(k-3)N^2})  < \epsilon_0, \text{ and } kN^2 < T, $$
  for a constant $\epsilon_0$ that is smaller than the $\epsilon$ in
  Proposition~\ref{prop:3ann2}. 

  If $(k+1)N^2 \geq T$, then 
  Proposition~\ref{prop:Dgrowth1} implies that $D_V(M_T) \leq C\epsilon_0$. If
  $\epsilon_0$ is chosen sufficiently small (and $\epsilon_1 <
  \epsilon_0$), then this implies $D_V(M_T) < \delta_1$ as
  required. 
  
  We therefore assume that $(k+1)N^2 < T$ and 
  $D_V(M_{(k-2)N^2}) \geq \epsilon_0$. We have $D_V(M_{(k-3)N^2})
  \leq e^{-\lambda_1'N^2} D_V(M_{(k-2)N^2})$ so using \eqref{eq:bound-N} we have
  \begin{equation}\label{eq:DVchange}
    D_V(M_{(k-2)N^2}) - D_V(M_{(k-3)N^2}) \geq (1-
    e^{-\lambda_1'N^2}) D_V(M_{(k-2)N^2}) \geq
    \frac{\epsilon_0}{2}.
  \end{equation}
 We claim that \eqref{eq:DVchange} together with condition
  (1) and   $D_V(M_{(k-3)N^2}) < \epsilon_0$ leads to a
  contradiction if $\epsilon_0$ is chosen small, and $\epsilon_1$ is
  sufficiently small depending on $\epsilon_0$. 

  Shifting $\tau = (k-3)N^2$ to $\tau = 0$, we can thus suppose that we have a sequence of flows $M_\tau^i$ satisfying Condition \eqref{cond:ddag} for
  $\tau\in [-N^2, N^2]$, satisfying
  \begin{align*}
      \mathcal{A}_\alpha(M_{-N^2}^i) - \mathcal{A}_{\alpha}(M_{N^2}^i) <
      \epsilon_{1,i},\\
     D_V(M^i_{N^2}) - D_V(M^i_{0}) \geq
    \frac{\epsilon_0}{2}, \\
    D_V(M^i_{-N^2}) < \epsilon_0, 
    \end{align*}
for a sequence
  $\epsilon_{1,i}\to 0$. From this, together with the same argument as in
 \cite[Theorem A]{Neves:zero-maslov}, we know that as $i\to \infty$,
  along a subsequence, these flows converge  to a static flow given by a union of
  planes. By the bound on $D_V(M^i_{-N^2})$ these planes must be given by
  some $V' \in \mathcal{V}$ if $\epsilon_0$ is small enough. In
  particular for any $\tau\in [0, N^2]$
  the $M^i_\tau$ converge
  locally smoothly to $V'$ on compact sets in $\mathbb{C}^2\setminus
  \{\mathbf{0}\}$. 

  Since $M^i_{N^2}$ and $M^i_0$ both converge to $V'$, for any compact
  set $K\subset\mathbb{C}^2\setminus \{\mathbf{0}\}$ we have as $i\to\infty$:
  $$ \int_{M^i_{N^2}\cap K} (|\mathbf{x}|^2 d_V^2 + |\theta -\theta_V|^2)
    e^{-|\mathbf{x}|^2/4} - \int_{M^i_{0}\cap K} (|\mathbf{x}|^2 d_V^2 + |\theta -\theta_V|^2)
    e^{-|x|^2/4} \to 0. $$
  We can use the uniform bound $D_V(M^i_{-N^2}) <\epsilon_0$ and  argue
  as in the proof of Proposition~\ref{prop:3ann2}, in
  particular in \eqref{eq:diIV}, to show $D_V(M^i_{N^2}) - D_V(M^i_0)
  \to 0$ which gives our required contradiction.
\end{proof}

We can now prove our main result controlling the distance of flows
close to the union of two transverse planes by combining  Propositions~\ref{prop:decay10} and \ref{prop:d20}.  
\begin{prop}\label{prop:d30}
  Let $\delta_2 > 0$. There is an $\epsilon_2 > 0$ depending on
  $\delta_2$, such that if we have a flow $M_\tau$ satisfying
  Condition \eqref{cond:ddag} for $\tau\in [-1,T+10]$, $M_1$ satisfies
  Condition \eqref{cond:ast}, and
  \begin{enumerate}
  \item $\mathcal{A}_\alpha(M_0) - \mathcal{A}_\alpha(M_T) < \epsilon_2$,
  \item $D_{V_0}(M_0) < \epsilon_2$,
  \end{enumerate}
  then $D_{V_0}(M_T) < \delta_2$. 
\end{prop}

\noindent Note that $T$ is independent of the constants $\delta_2, \epsilon_2$ and, in particular, can be large.

\begin{proof} 
  We iterate the cases (i), (ii) in Proposition~\ref{prop:decay10} as
  long as possible, starting with $V_0$. We obtain a sequence $V_1,
  \ldots, V_k \in \mathcal{V}$ together with numbers $e_i = D_{V_i}(M_{iN})$ such that
  $d(V_i, V_{i+1}) \leq Ce_i$. We can continue this iteration and
  define $V_{k+1}$ unless one of the following occurs: 
  \begin{itemize}
  \item[(a)] $kN + 3N^2 +2 > T$,
  \item[(b)] $V_k\not\in \mathcal{V}'$ or $D_{V_k}(M_{kN}) \geq
    \epsilon_0$, 
  \item[(c)] $D_{V_k}(M_{kN+N^2}) \geq e^{\lambda_1'N^2}
    D_{V_k}(M_{kN})$. 
  \end{itemize}
  We show that if $\epsilon_2$ is sufficiently small, then (b)
  cannot occur before (a) or (c) does. To see this we can argue as in
  \cite[Theorem 6.7]{Sz20} to control the sum of the $e_i$, to find
  \[\sum_{i=0}^ke_i \leq 2e_0 + 2C(\mathcal{A}_\alpha(M_0) -
    \mathcal{A}_\alpha(M_T)) \leq C\epsilon_2. \]
  In particular this implies that
  both $e_k$ and  $d(V_0, V_k)$ are bounded above by $C\epsilon_2$, so we can
  ensure that (b) does not occur for $\epsilon_2$ sufficiently
  small. In addition, using Lemma~\ref{lem:condast}, Condition
 \eqref{cond:ast} is preserved. 

  If (a) occurs first, then we have \[D_{V_k}(M_T)\leq CD_{V_k}(M_{kN})
  \leq C\epsilon_2\] by Proposition~\ref{prop:Dgrowth1}, and by Lemma~\ref{lem:DVV'comp} we get
  $D_{V_0}(M_T) \leq C\epsilon_2$.
  If (c) occurs first, then from Proposition~\ref{prop:d20} we
  conclude that $D_{V_k}(M_T) < \delta_1$ if we choose $\epsilon_2 <
  \epsilon_1$, and Lemma~\ref{lem:DVV'comp} implies $D_{V_0}(M_T) \leq
  C(\delta_1 +\epsilon_2)$.

  If we choose $\delta_1>0$ sufficiently small (determining a value for
  $\epsilon_1>0$), and then choose 
  $\epsilon_2>0$ small so that also $\epsilon_2 < \epsilon_1$, then in
  either case we will have $D_{V_0}(M_T) < \delta_2$ as required. 
\end{proof}

\section{Neck pinches}\label{sec:neckpinch}
In this section we give the main geometric applications of the
estimates we have obtained. We suppose that
$\mathcal{M}(t)$ is a rational, zero Maslov Lagrangian mean curvature
flow in $\mathbb{C}^2$, with uniformly bounded  area ratios and
uniformly bounded Lagrangian angle.

\subsection{Uniqueness of tangent flows} Our first result is the  uniqueness of tangent flows given by
 a union of transverse planes. Note that by  \cite[Corollary
 4.3]{Neves:survey} the two planes must have the same Lagrangian angle
 if a singularity forms. We first have the following, ensuring that
 Conditions \eqref{cond:ddag} and \eqref{cond:ast} hold along the corresponding rescaled
 flow.

 \begin{lemma}\label{lem:Conditions}
   Let $\mathcal{M}(t)$ be a mean
   curvature flow in $\mathbb{C}^2$, with initial condition given by
   a rational, zero Maslov Lagrangian with uniformly bounded area
   ratios and uniformly bounded Lagrangian angle. Suppose that the
   flow develops a singularity at $(\mathbf{0},0)$, 
   with a tangent flow given by the static flow $V_0$, where
   $V_0$ is a special Lagrangian union of two transverse
   planes. Let $M_\tau = e^{\tau/2}\mathcal{M}(-e^{-\tau})$ denote the
   corresponding rescaled flow. Then there is a sequence $\tau_i\to
   \infty$ satisfying
   \begin{enumerate}
   \item $D_{V_0}(M_{\tau_i})\to 0$,
   \item $M_\tau$ satisfies Condition \eqref{cond:ddag} for $\tau\in [\tau_0,
     \infty)$,
   \item $M_{\tau_i+1}$ satisfies Condition \eqref{cond:ast}.
   \end{enumerate}
 \end{lemma}
 \begin{proof}
     Note that
     the uniform bounds on area ratios and the bound for the
     Lagrangian angle is preserved along the flow. Condition \eqref{cond:ddag}
     holds on $[\tau_0, \infty)$ for sufficiently large $\tau_0$
     by the monotonicity of $\mathcal{A}(M_\tau)$. The fact that
     $D_{V_0}(M_{\tau_i})\to 0$ follows from the assumption that one
     tangent flow at $(\mathbf{0},0)$ is given by $V_0$. 

     It remains to show Condition \eqref{cond:ast} for $M_{\tau_i+1}$ for
     large enough $i$. Let us first consider the connectedness of
     $B_1\cap M_{\tau_i+1}$. Note that by the assumption
     $D_{V_0}(M_{\tau_i})\to 0$ we have that $M_{\tau_i+1}$ has good
     graphicality over $V_0$ on $B_2\setminus B_{1/2}$ for large
     $i$. For large $i$ the
      pointwise bounds in Lemma~\ref{lem:IVest} also imply that $B_2
      \cap M_{\tau_i+1}$ is almost calibrated. Since there are no
      compact almost calibrated Lagrangians in $\mathbb{C}^2$, 
     this implies that $B_1\cap M_{\tau_i+1}$ has either 1 or 2
     connected components. If it has 2 components, then for
     sufficiently large $i$ we can argue as in \cite[Corollary
     4.3]{Neves:survey} to show that in fact the original flow
     $\mathcal{M}(t)$ does not have a singularity at
     $(\mathbf{0},0)$. Therefore for sufficiently large $i$, $B_1\cap
     M_{\tau_i+1}$ is connected. 

     Finally consider the  exactness
     part of Condition \eqref{cond:ast}. The rationality assumption is preserved along the
     flow, see \cite[Section 6]{Neves:zero-maslov}. It follows that
     there is a constant $a > 0$ such that, after rescaling,
     for any loop $\gamma\subset M_{\tau_i+1}$ we have
     \begin{equation}\label{eq:intl10}
       \int_\gamma \lambda \in 2\pi a e^{\tau_i+1} \mathbb{Z}.
     \end{equation}
     Although we might only be able to define a multivalued function
     $\beta$ satisfying $\rd\beta = \lambda$ on $M_\tau$,
     \eqref{eq:intl10} implies that $f = \sin(e^{-\tau_i-1}a^{-1}\beta)$ is
     single valued. Without loss of generality we can assume that
     $f(\mathbf{x}_0)=0$ for a basepoint $\mathbf{x}_0\in B_1\cap M_{\tau_i+1}$. We have
      $$ \nabla f = e^{-\tau_i-1}a^{-1}\cos(e^{-\tau_i-1}a^{-1}\beta) \nabla
        \beta, $$
      so using $|\nabla \beta| = |\mathbf{x}^\perp|$ we have $|\nabla
      f| \leq e^{-\tau_i-1}a^{-1}$ on $B_1$. For large $i$, $B_2\cap
      M_{\tau_i+1}$ is almost calibrated, and so as
      in the proof of \cite[Lemma 7.2]{Neves:zero-maslov}, we have a 
      uniform lower bound $\mathcal{H}^2(\hat{B}_i(\mathbf{x}_1,1)) > K$
      for the intrinsic unit balls in $M_{\tau_i+1}$ centred at any
      $\mathbf{x}_1\in B_1\cap M_{\tau_i+1}$. We also have an upper bound for the
      area of $B_1\cap M_{\tau_i+1}$, using the bound for the area
      ratios. Together with the connectedness this implies that there
      is a uniform constant $C_1$ such that for any $\mathbf{x}_1\in B_1\cap
      M_{\tau_i+1}$, the points $\mathbf{x}_0, \mathbf{x}_1$ can be connected by a curve in
      $B_1\cap M_{\tau_i+1}$ of length at most $C_1$, for large
      $i$. The gradient bound for $f$, together with $f(\mathbf{x}_0)=0$, then
      implies that $|f|\leq C_1 e^{-\tau_i-1}a^{-1}$ on $B_1\cap
      M_{\tau_i+1}$. For sufficiently large $i$ we can then define a
      single-valued function $\beta$ on $B_1\cap M_{\tau_i+1}$
      satisfying $\rd\beta = \lambda$, so $\int_\gamma \lambda = 0$
      follows for any loop $\gamma\subset B_1\cap M_{\tau_i+1}$. 
\end{proof}

 \begin{thm}\label{thm:u1}
   Suppose that $\mathcal{M}(t)$ satisfies the same assumptions as in
   Lemma~\ref{lem:Conditions}. Then all tangent flows at $(\mathbf{0},0)$ are given by $V_0$.
\end{thm}
\begin{proof}
Let $M_\tau$ be the rescaled flow around $(\mathbf{0},0)$. Using
  Lemma~\ref{lem:Conditions} we have that Condition \eqref{cond:ddag} holds on a
  time interval of the form $[\tau_0, \infty)$, and in addition we
  have a sequence $\tau_i\to \infty$ such that $D_{V_0}(M_{\tau_i})\to
  0$ and $M_{\tau_i+1}$ satisfies Condition \eqref{cond:ast}.
  
  The uniqueness of the tangent flow then follows directly from
  Proposition~\ref{prop:d30}. Indeed,  from
  the monotonicity of $\mathcal{A}_\alpha$ from Lemma~\ref{lem:Aepschange} we also have
$$ \mathcal{A}_\alpha(M_{\tau_i}) - \lim_{\tau\to\infty}
  \mathcal{A}_\alpha(M_\tau) \to 0 \quad\text{ as } i\to \infty. $$
It follows from Proposition~\ref{prop:d30} that
$$ \lim_{i\to \infty} \sup_{\tau > \tau_i} D_{V_0}(M_\tau) = 0, $$
which implies that $M_\tau \to V_0$ locally smoothly on compact sets
in $\mathbb{C}^2\setminus \{\mathbf{0}\}$ as $\tau\to \infty$. 
\end{proof}

\subsection{Lawlor necks} We now show that tangent flows given by a
special Lagrangian union of two transverse planes can only form if for
sufficiently small times before the singularity the flow looks locally
like the two transverse planes, desingularised by a shrinking Lawlor
neck. Recall that, given a special Lagrangian pair of transverse planes $V_0$ in $\mathbb{C}^2$, up to scale there are two (exact) Lawlor necks  $N_\pm$ asymptotic to these planes (corresponding to $zw=\pm1$ in suitable complex coordinates $(z,w)$ under hyperk\"ahler rotation).  Using ideas of Seidel (cf.~\cite{Seidel99}), one can see that  $N_{\pm}$ are not Hamiltonian isotopic (using compactly supported isotopies), but we shall not use this fact in our result below.  %Since the Lagrangian mean curvature flow in the zero Maslov class case preserves the Hamiltonian isotopy class and the relative orientation of the planes in the tangent flow is fixed by Theorem~\ref{thm:u1}, the Hamiltonian isotopy class of Lawlor necks we see  as we approach the singular time will be the same at all scales.

\begin{thm} \label{thm:Lawlor}
  Under the hypotheses of Theorem~\ref{thm:u1}, for every $\varepsilon>0$ there is $r_0>0$ and a smooth function $r:[-r_0^2,0) \to (0, r_0)$ with $r(t) \to 0$ as $t\to 0$  and points $\mathbf{x}_0(t) \to 0$ such that $\cM(t)\cap \big(B_{r_0}(\mathbf{0})\setminus B_{r(t)}(\mathbf{x}_0(t))\big)$ is a $C^1$-graph over $V_0$ with $C^1$-norm bounded by $\varepsilon$. Furthermore,
  \[ r(t)^{-1}(\mathcal{M}(t) - \mathbf{x}_0(t))\]
  converge locally smoothly on $\mathbb{C}^2$ to a unique choice of
  Lawlor neck (either $N_+$ or $N_-$) asymptotic to $V_0$ %(depending on the relative orientation of $P_1\cup P_2 = V_0$) 
  at  infinity of maximal neck size such that, outside of $B_1(\mathbf{0})$, $N$ can be written as a $C^1$-graph over $V_0$ with $C^1$-norm bounded by $\varepsilon$.
\end{thm}

To prove this result, we suppose $\cM$ satisfies the hypotheses of
Theorem~\ref{thm:u1} and consider basepoints $X =
(\bx,t)\in\CC^2\times [-1,0]$ close to $(\mathbf{0}, 0)$ at which to
discuss closeness of $\cM$ to some $V\in \mathcal{V}$.  

\begin{definition} %[$\varepsilon$-transverse] 
Let $\varepsilon_0\in(0,1)$, $\varepsilon\in (0,\varepsilon_0)$ and
$V\in \mathcal{V}$. We say that $\cM$ is \emph{$\varepsilon$-close} to
$V$ at  $X=(\bx,t)$ if the flow $\cM'=\cM - X$ is a Lagrangian $C^1$-graph with $C^1$-norm bounded by $\varepsilon$ over $V$ on $(B_{\varepsilon^{-1}}(\mathbf{0})\setminus \bar{B}_\varepsilon(\mathbf{0}))\times [-\varepsilon^{-2}, - \varepsilon^{2}]$.  We assume  $\varepsilon_0$ is chosen sufficiently small (depending on $\mathcal{V}$) such that $\cM'\cap ((B_{\varepsilon^{-1}}(\bOh)\setminus \bar{B}_\varepsilon(\bOh))\times [-\varepsilon^{-2}, - \varepsilon^{2}])$ is the union of two disjoint embedded annuli.
\end{definition}

 \begin{remark} \label{rem:pseudolocality} We will assume further that $\eps_0>0$ is sufficiently small such that pseudolocality \cite{IlmanenNevesSchulze} implies that for every $\delta>0$ there is $C_{\delta}\gg 1$ and $0<\eps\leq \eps_0$ such that if $\cM(t_0)\cap B_{C_\delta}(\bx)$ is a $C^1$-graph over a 2-plane $P$ with $C^1$-norm  bounded by $\varepsilon$, then $\cM(t) \cap B_{1}(\bx)$ is a $C^1$-graph over $P$ with $C^1$-norm  bounded by $\delta$ for $t \in [t_0, t_0+1]\cap [-1,0)$.\end{remark}

Next we identify the range of scales at which the flow is close to
some $V\in \mathcal{V}$.  We fix a small $\varepsilon\in(0,\epsilon_0)$.

\begin{definition} %[Transversality scales]  
Suppose that $\cM$ is $\varepsilon$-close at $X$ to some $V\in \mathcal{V}$. 
We define $\lambda_\text{min}(X), \lambda_\text{max}(X)$ to be the endpoints of the maximal interval $$1 \in (\lambda_\text{min}(X), \lambda_\text{max}(X)) \subseteq (0, \infty)$$ such that for all $\lambda \in (\lambda_\text{min}(X), \lambda_\text{max}(X))$ we have that $\cD_{\lambda^{-1}}(\cM-X)$ is $\varepsilon$-close at $(\mathbf{0},0)$ to $V_\lambda$ for some $V_\lambda \in \mathcal{V}$.

Note that $\lambda_\text{min}(X), \lambda_\text{max}(X)$ are continuous in the base-point $X$. 
\end{definition}

\begin{remark} Since we can assume that all tangent flows of $\cM$ at $(\bOh,0)$ are $\cM_{V_0}$, for any sequence $\lambda_i \searrow 0$ we have 
\begin{equation}\label{eq:Mi.transverse}
\cM_i := \cD_{\lambda_i^{-1}} \cM \rightharpoonup \cM_{V_0}\, .
\end{equation}
Note that along the sequence $\cM_i$ this implies that for all points $X=(\bx,t)$ sufficiently close to $\{\mathbf{0}\} \times (-\infty, 0)$ we have that $\lambda_\text{min}(X) \to 0$ and $\lambda_\text{max}(X) \to \infty$. 
\end{remark}

We now wish to rule out the possibility that $\lambda_{\text{min}}(X)=0$.

\begin{lemma} %[Zero minimal transversality scale implies no singularity] 
\label{lem:zero-transversality}
Assume that for some $X = (\bx,t)$ with $t<0$ we have along $\cM_i$ in \eqref{eq:Mi.transverse} that $\lambda_\text{min}(X) = 0$. Then, locally around $X$, the flow $\cM$ is the smooth flow of two immersed planes.
\end{lemma}

\begin{proof}
We first note that Remark \ref{rem:pseudolocality} implies that we have smooth control on the flow forward in time in $B_{\lambda (\varepsilon^{-1}-C_\delta)}(\bx)\setminus B_{\lambda (C_\delta+\varepsilon)}(\bx)$ up to time $t$ as one goes down with the scale $\lambda$ from $1$ down to $\lambda_\text{min}$, i.e.~in  $B_{(\varepsilon^{-1}-C_\delta)}(\bx)\setminus B_{\lambda_\text{min} (C_\delta+\varepsilon)}(\bx)$. So if $\lambda_\text{min}(X) = 0$  we have that $\cM_i(t)$ is locally around $\bx$ the union of two smooth embedded flows, where each is a small $C^1$-graph over the Lagrangian planes $P_1$ and $P_2$ respectively (for $V_0 = P_1\cup P_2$). Since the flow $\cM_i$ is smooth, this has to be true forwards in time until $t=0$, and thus there is no singularity at $(\mathbf{0},0)$ for $\cM_i$  and thus for $\cM$.
\end{proof}

Given Lemma~\ref{lem:zero-transversality}, we can thus always assume that $\lambda_\text{min}(X) > 0$. We now argue that there is more or less a `unique' point  where $\lambda_\text{min}$ is minimised.

\begin{lemma} %[Uniqueness of minimum for the minimal transversality scale]
\label{lem:uniq.lambda.min}
For $\cM_i$ as in \eqref{eq:Mi.transverse}  consider points $X_i(t)= (\bx_i(t),t)$ which minimise $\lambda_\text{min}$ relative to other points $X=(\bx,t)$.  Then 
$$\lambda_\text{min}(\bx,t) > \lambda_\text{min}(X_i(t))>0$$ 
for $(\bx,t)  \in \big(B_{\lambda_\text{max}(X_i)(\varepsilon^{-1} - C)}(\bx_i(t)) \setminus B_{\lambda_\text{min}(X_i)(C+\varepsilon)}(\bx_i(t))\big) \times \{t\}$, where $C=C_\delta$ for a suitable $\delta>0$ in Remark \ref{rem:pseudolocality}.

\end{lemma}

\begin{proof}
 This follows by a similar argument to the proof of Lemma~\ref{lem:zero-transversality} since Remark \ref{rem:pseudolocality} implies that we have smooth control on the flow forward in time in $B_{\lambda (\varepsilon^{-1}-C)}(\bx)\setminus B_{\lambda (C+\varepsilon)}(\bx)$ up to time $t$ as one goes down with the scale $\lambda$ from $\lambda_\text{max}$ down to $\lambda_\text{min}$.
\end{proof}

\begin{lemma}\label{lem:down-the-scales}
 There is $0<\varepsilon_1\leq \varepsilon_0$ such that for $0<\varepsilon\leq \varepsilon_1$ the following holds. If there is $\lambda_0 \in (\lambda_\text{min}(X),\lambda_\text{max}(X))$ such that we can choose $V_{\lambda_0} = V_0$, then $V_{\lambda} \in \mathcal{V}'$ for all $\lambda \in  (\lambda_\text{min}(X), \lambda_0]$.
\end{lemma}
\begin{proof} We consider the rescaled flow $(\hat M_\tau)_{\tau\geq
    0}$ for $\hat\cM:=\cD_{\lambda_0^{-1}}(\cM - X)$ and choose
  $\delta_2>0$ and $\varepsilon_0 >0$ small such that $\hat
  M_{-2\log(\lambda)+2\log(\lambda_0(X))}$ for $\lambda \in
  (\lambda_\text{min}(X), \lambda_0]$ being $\varepsilon$-close to $V_\lambda$ (for $0<\varepsilon \leq \varepsilon_0$) and $D_{V_0}(\hat M_{-2\log(\lambda)+2\log(\lambda_0(X))}) \leq \delta_2$ implies that $V_\lambda \in \mathcal{V}'$. This fixes $\epsilon_2>0$ in Proposition~\ref{prop:d30}. We can then choose $\varepsilon_1>0$ sufficiently small such that 
$$\hat M_0, \hat M_{-2\log(\lambda_\text{min}(X)) -2\log(\lambda_0)}$$ 
being $\varepsilon$-close to $V_0$ and $V_{\lambda_\text{min}(X)}$  respectively implies that condition (1) in Proposition~\ref{prop:d30} is met. Applying Proposition~\ref{prop:d30} yields the statement.
\end{proof}

\subsection{Finding Lawlor necks}  We consider $0<\varepsilon\leq \varepsilon_1$, where $\varepsilon_1$ as in Lemma \ref{lem:down-the-scales}. Consider a sequence  $\lambda_i \searrow 0$ and let $\cM_i$ be as in \eqref{eq:Mi.transverse}.  We fix $t<0$ and let $X_i(t)$ be as in Lemma \ref{lem:uniq.lambda.min}.   
We consider the flows
\begin{equation}\label{eq:Mitprime.transverse} \cM'_{i,t} : = \cD_{\lambda^{-1}_\text{min}(X_i(t))}(\cM_i - X_i(t))\ . 
\end{equation}
We can assume that $ \cM'_{i,t} \rightharpoonup \widehat{\cM}$, where
$\widehat{\cM}$ is an ancient unit-regular Brakke flow such that
$\mathcal{D}_{\lambda^-1} \widehat{\cM}$ is $\varepsilon$-close to
some $V_\lambda\in \mathcal{V}$ for $\lambda \in  [1,\infty)$, but not for a sequence  $\lambda_i \nearrow 1$. Furthermore, $\widehat{\cM}$ is (locally) the limit of smooth, exact, almost calibrated Lagrangian mean curvature flows with uniformly bounded Lagrangian angle and uniformly bounded area ratios.

\begin{lemma}\label{lem:finding-Lawlor-necks}
The flow $\widehat{\cM}$ is a static Lawlor neck $N$ asymptotic to  $V' \in \mathcal{V}'$, where $V'$ is $\varepsilon$-close to $V_0$, where the centre and the scale of $N$ are such that there is no point $X=(\bx,t)$ with $\lambda_\text{min}(X) < 1$.  
\end{lemma}
\begin{proof}
We first note that $\widehat{\cM}$ has entropy bounded by
two. Furthermore, from the argument in the proof of Lemma
\ref{lem:zero-transversality} we see that outside of
$(B_{C_{\delta}+1}(\bOh) \times (-\infty, 0)) \cup \bigcup_{t<0}
B_{\sqrt{-t}(C_\delta+1)}(\bOh) \times \{t\}$ the flow is a smooth
Lagrangian which is a controlled $C^1$-graph over $V_0$. Let
$\widehat{\cM}'$ be a tangent flow of $\widehat{\cM}$ at
$-\infty$. The discussion before implies that
$\mathcal{D}_\lambda\widehat{\cM}'$ is $\varepsilon$-close to some
$V_\lambda \in \mathcal{V}'$ for all $\lambda>0$ as well as that
$\widehat{\cM}'$ is a smooth Lagrangian and controlled $C^1$-graph
over $V_0$ outside of $\bigcup_{t<0} B_{\varepsilon\sqrt{-t}}(\bOh)
\times \{t\}$. 
Furthermore, the proof of  \cite[Theorem 3.1]{NevesTian:translating},
see also \cite[Theorem 3.1]{LambertLotaySchulze}, directly extends to
Brakke flows which are limits of smooth Lagrangian mean curvature
flows with uniformly bounded Lagrangian angle and uniformly bounded
area ratios. Thus we can assume that $\widehat{\cM}'$ is a static pair
of planes $ V' \in \mathcal{V}'$.  
Note that if there is a point where $\widehat{\cM}$ has Gaussian density two, then the flow is static (since it is unit regular) and up to a translation equal to $\cM_{V'}$. However, this would imply that there are points $\tilde{X}_i=(\tilde{\bx}_i,t)$ close to $X_i(t)$ with $\lambda_\text{min}(\tilde{X}_i) < \lambda_\text{min}(X_i(t))$ where  $X_i(t)$ is as in  Lemma~\ref{lem:uniq.lambda.min}. This yields a contradiction.

We may therefore assume that all Gaussian densities are less than two, which implies the convergence $\cM'_{i,t} \rightharpoonup \widehat{\cM}$ is smooth and thus $\widehat{\cM}$ is a smooth, ancient, exact, almost calibrated Lagrangian mean curvature flow with uniformly bounded Lagrangian angle and uniformly bounded area ratios. Since $\widehat{\cM}$ cannot be a union of Lagrangian planes, by \cite[Theorem 1.1]{LambertLotaySchulze} it is up to rigid motions a static Lawlor neck. Arguing similarly as before we see that there cannot be a point $X=(\bx,t)$ with $\lambda_\text{min}(X) < 1$.
\end{proof}

\begin{proof}[Proof of Theorem \ref{thm:Lawlor}]
Consider $0<\varepsilon\leq \varepsilon_1$, where $\varepsilon_1$ as
in Lemma \ref{lem:down-the-scales}. Replacing $\cM$ by
$D_{\lambda^{-1}}\cM$ for some sufficiently small $\lambda>0$ we can
assume that for all $t\in [-1,0)$ there are points $X(t)$ as in Lemma
\ref{lem:uniq.lambda.min} such that $\lambda_\text{min}(X(t))<1\leq
\lambda_\text{max}(X(t))$, minimising $\lambda_{\text{min}}$ in their
time slice, and with $V_\lambda = V_0$ for $\lambda =1$.   Lemma \ref{lem:down-the-scales} then implies that for each $t\in (0,1]$ and $\lambda \in [\lambda_\text{min}(X(t)), 1]$ we have that $\cD_{\lambda^{-1}}(\cM - X(t))$ is $\varepsilon$-close to $V_{\lambda, t} \in \mathcal{V}'$. Applying Remark \ref{rem:pseudolocality} we see that $\cM(t) - \bx(t)$ is a small $C^1$-graph over $V_0$ on $B_1(\bOh) \setminus B_{C \lambda_\text{min}(X(t))}(\bOh)$ for a suitable fixed $C>0$ (depending only on $\varepsilon>0$). 

Restricting to $-\delta\leq t<0$ we can apply Lemma \ref{lem:finding-Lawlor-necks} to find that the rescaled flow $\cD_{(\lambda_\text{min}(X(t))^{-1}}(\cM - X(t))$ has to be a small $C^1$-graph over a Lawlor neck $N$ over a large compact set. Note that by continuity of the flow (and assuming sufficient $C^1$-closeness to $N$) this has to be either $N_+$ or $N_-$ for all $-\delta\leq t<0$.
%and, as noted before Theorem~\ref{thm:Lawlor}, the Hamiltonian isotopy class of Lawlor neck is fixed, so we obtain a unique choice of Lawlor neck $N$.  
The full convergence to $N$ as $t\nearrow 0$ then follows by considering a sequence $\varepsilon_i\to 0$.
\end{proof}
  
  \subsection{Continuing the flow past the singularity}
We can now argue that the uniqueness of the tangent flow implies that at the singular time $t=0$, in a neighbourhood of $\bOh \in \mathbb{C}^2$, the flow limits to the union of two Lagrangian graphs such that we can restart the flow as a Lagrangian mean curvature flow.

\begin{lemma} \label{lem:structure-sing-time} There is $r_0>0$ such that on $B_{r_0}(\bOh)\setminus \{\bOh\}$ the flow $\cM(t)$ converges as $t\to 0$ locally to two Lagrangian graphs $L_1, L_2$ over $P_1$ and $P_2$ respectively (where $V_0 = P_1\cup P_2$). Moreover, using Lagrangian neighbourhoods for $P_i$, we have that  $L_i = \text{graph}_{P_i}(\rd f_i)$ for $i=1,2$, where $f_i \in C^2\big(P_i\cap B_{r_0}(\bOh)\big)$ is smooth away from $\bOh$ with $f_i(\bOh)=\rd f_i(\bOh)= \nabla_{P_i}(\rd f_i)(\bOh) = 0$.
\end{lemma}

\begin{proof}
 Note that Theorem \ref{thm:u1} implies that there is $0<\lambda_0\leq 1$ and a continuous increasing function $\varepsilon: (0, \lambda_0] \to (0,1)$ with $\varepsilon(\lambda) \to 0$ as $\lambda \to 0$ such that  $\mathcal{D}_{\lambda^{-1}} \cM$ is $\varepsilon(\lambda)$-close to $V_0$. Combining this with Remark \ref{rem:pseudolocality} as in the proof of Theorem \ref{thm:Lawlor} then yields that there is a smooth Lagrangian limiting surface $L$ on $B_{r_0}(\bOh)\setminus \{\bOh\}$ such that $\lambda^{-1}L$ is $\delta(\lambda)$-close to $V_0$ for some continuous decreasing function $\delta: (0, \lambda_0] \to (0,1)$ with $\delta(\lambda) \to 0$ as $\lambda \to 0$. This implies the claimed convergence and that we can decompose $L = L_1\cup L_2$ such that each $L_i$ is a small $C^1$-graph over $P_i\cap B_{r_0}(\mathbf{0})$ of a vector-valued function $u_i$, which is smooth away from $\mathbf{0}$ and $C^1$ across $\bOh$ with $u_i(\bOh) = \nabla_{P_i}u_u(\bOh) = 0$.  Applying the Lagrangian neighbourhood theorem to each $P_i$, we see that $u_i$ may be identified with a closed 1-form on $P_i\cap B_{r_0}(\mathbf{0})$, which is then necessarily exact.  The Poincar\'e lemma then gives a $C^2$-function $f_i$ with $\rd f_i=u_i$ as claimed.  
\end{proof}

\begin{proof}[Proof of Theorem \ref{thm:intro2}]
By Lemma \ref{lem:structure-sing-time} there can be only finitely many singularities at time $T$ where a tangent flow is a static union of two multiplicity one transverse planes. For simplicity of notation we can thus assume as before that there is one singularity and by shifting space and time that it occurs at $(\bOh,0)$. 

Using Lemma \ref{lem:structure-sing-time} we see that $\cM(t)$ converges as $t\nearrow 0$ to a $C^1$-immersed Lagrangian $L$, where the convergence (and $L$) is smooth away from $\bOh$, which is zero Maslov and rational. Furthermore, $L\cap B_{r_0}(\bOh)$ is given as the union of two Lagrangian graphs as stated in Lemma \ref{lem:structure-sing-time}. 

We can use the decomposition given in Lemma \ref{lem:structure-sing-time} to approximate $L$ by smooth, zero Maslov, rational Lagrangians $L^i$ in $C^1$ (by approximating the $C^2$-functions giving $L$ as a graph by smooth functions). Furthermore, we can assume by the estimates of \cite{Wang04} that there is $T>0$ and smooth, zero Maslov, rational solutions to Lagrangian mean curvature flow $(L^i_t)_{t\in [0,T)}$ with $L^i_0 = L^i$. By the $C^1$-convergence of $L^i \to L$ and interior estimates for higher codimension mean curvature flow (see \cite{Wang04} or \cite[Appendix]{BegleyMoore}), we see that the flows are uniformly controlled in $C^\infty$ for $t>0$. Note that we can assume that the convergence $L^i \to L$ is smooth away from the singular points.
Taking the limit we thus see that there is a smooth, zero Maslov, rational Lagrangian mean curvature flow $(L_t)_{t\in (0,T)}$ with $L_t \to L$ in $C^1$ (and smoothly away from $\bOh$) as $t\searrow 0$. This implies that the extended flow is smooth through the singular time, away from the singular points. Note that for $\varepsilon>0$ there is $\delta>0$ such that
$$ \sup_{\substack{\bx \in L_t\cap B_\delta\\ -\delta^2<t<0}} |\theta(\bx,t) - \theta_0| \leq \varepsilon, $$
where $\theta_0$ is the Lagrangian angle of the special Lagrangian cone $V_0$. Thus the grading $\theta$ for the extended flow can be chosen that it is smooth as well through the singular time, away from the singular points.
\end{proof}

\section{The flow in a compact ambient space}\label{sec:compactCY}
In this section we consider the Lagrangian mean curvature flow in a
compact Calabi--Yau surface  and we briefly explain the
modifications needed to prove the results from
Section~\ref{sec:neckpinch} in this setting.

Let us suppose that $X$ is a compact complex surface with complex structure
$J$, admitting a non-vanishing holomorphic volume form $\Omega$ and a
K\"ahler metric $\omega$ with volume form $\frac{1}{2}\omega^2 = \Omega
\wedge\bar\Omega$. Let $L\subset X$ be a Lagrangian submanifold. We say that $L$
is zero Maslov if there is a Lagrangian angle function $\theta: L \to
\mathbb{R}$ satisfying
$$ \Omega|_L = e^{i\theta} dA_L $$
in terms of the Riemannian area form $dA_L$ of $L$. Following the
notion of rationality in Fukaya~\cite[Definition 2.2]{Fukaya03} we assume furthermore
that $[\omega]$ defines an integral cohomology class in $H^2(X;
\mathbb{R})$, and let $\xi$ denote a complex line bundle over $X$
together with a unitary connection $\nabla^\xi$ with curvature form
$F_{\nabla^\xi} = 2\pi i \omega$. The connection $\nabla^\xi$ is then
flat when restricted to $L$.
\begin{definition}\label{defn:rationalL}
  The Lagrangian $L$ in  $X$ is \emph{rational}
if the holonomy group of $\nabla^\xi$ on $L$ is a finite subgroup of
$U(1)$.
\end{definition}

We will follow the approach of White~\cite[Section 4]{White.regularity},
viewing the mean curvature flow in $X$ as a mean curvature flow in a
larger Euclidean space with an additional forcing term. More
precisely, let $X\subset \mathbb{R}^N$ be an isometric embedding. The
mean curvature flow $L_t$ in $X$ is then equivalent to the flow
$$ \frac{\partial}{\partial t} \mathbf{x} = \mathbf{H} + \nu, $$
where $\mathbf{H}$ denotes the mean curvature vector inside $\mathbb{R}^N$
and $\nu(\mathbf{x}, t) = -\mathrm{tr}\,
\Pi(\mathbf{x})|_{T_\mathbf{x}L_t}$ is defined in terms of the second
fundamental form $\Pi$ of $X$, restricted to the tangent space of $L_t$
at $\mathbf{x}$. In particular $|\nu|\leq C$ for a constant
independent of $\mathbf{x},t$. Note that $\mathbf{H} + \nu$ is
simply the mean curvature of $L_t$ in the ambient space $X$, and so we
have $|\nabla \theta| = |\mathbf{H} + \nu|$.

We need to recall the form of the monotonicity formula for the
mean curvature flow with forcing term $\nu$: instead of
\eqref{eq:Huisken} we have
\begin{align*} 
\frac{d}{dt} \int_{L_t} f \rho_{\mathbf{x}_0, t_0}\, d\mathcal{H}^2
  &= \int_{L_t} (\partial_t f - \Delta f) \rho_{\mathbf{x}_0, t_0}\,
  d\mathcal{H}^2 + \int_{L_t} f \left| \frac{\nu}{2}\right|^2\,
  \rho_{\mathbf{x}_0, t_0} \, d\mathcal{H}^2 \\
&\qquad - \int_{L_t} f\, \left| \mathbf{H} -
  \frac{(\mathbf{x}-\mathbf{x}_0)^\perp}{2(t_0-t)} + \frac{\nu}{2} \right|^2\,
\rho_{\mathbf{x}_0, t_0}\, d\mathcal{H}^2. 
\end{align*}
The estimate of Ecker~\cite[Theorem 3.4]{Ecker.logSobolev} also
applies to subsolutions of the (drift) heat equation along the
rescaled flow with a forcing term, with a slightly modified function
$p(t)$ as in \cite[Theorem 3.2]{Ecker.logSobolev}. 

For simplicity we assume   the first singular time is at $t=0$  and
we are studying the tangent flow at $\bOh\in \mathbb{R}^N$. The
corresponding rescaled flow $M_\tau$ is given by
$$ \frac{\partial}{\partial \tau} \mathbf{x} = \mathbf{H} +
  \frac{\mathbf{x}^\perp}{2} + \nu. $$
Note that $M_\tau\subset e^{\tau/2}X\subset \mathbb{R}^N$, and the forcing
term $\nu$ is now given by $\nu(\mathbf{x}, \tau) = -\mathrm{tr}\,
\Pi(\mathbf{x}, \tau)|_{T_\mathbf{x} M_\tau}$, where $\Pi(\mathbf{x},
\tau)$ is the second fundamental form of $e^{\tau/2}X$. In particular
we have the estimate $|\nu|\leq Ce^{-\tau/2}$. 

Let us write $J_0, \Omega_0, \omega_0$ for the complex structure,
holomorphic volume form and symplectic form on $T_{\bOh}X$. In particular
we identify $T_{\bOh}X = \mathbb{C}^2$, equipped with its standard
structures. Note that $\lim_{\tau\to\infty} e^{\tau/2} X = T_{\bOh}X$ in
the sense of $C^\infty$ convergence on compact sets, and so any
tangent flow at $(\bOh,0)$, defined as a sequential limit of $M_\tau$ as
$\tau\to\infty$, lives naturally in $T_{\bOh}X$. As in
Section~\ref{sec:setup} we consider special Lagrangian unions $V=
P_1\cup P_2$ contained in a neighbourhood $\mathcal{V}$
of a given tangent flow $V_0$. For any $V\in\mathcal{V}$ there is a
hyperk\"ahler rotation of the complex structure on $T_{\bOh}X$ such that
$V$ is given by $\{zw=0\}$ for complex coordinates $z,w$. Note that
$z,w$ can be viewed as linear functions on $\mathbb{R}^N$, and in
particular they define functions on $X\subset \mathbb{R}^N$. We have
the following analogous result to Lemma~\ref{lem:nablazw}.

\begin{lemma}\label{lem:nablazw2}
  There is a constant $C > 0$, depending on $X\subset\mathbb{R}^N$ and
  the choice of $\mathcal{V}$, such that on any Lagrangian subspace
  $P\subset T_{\mathbf{x}}X$ with Lagrangian angle $\theta$ we have
  $$ |\nabla z \cdot \nabla w| \leq C ( |\theta - \theta_V| +
    |\mathbf{x}|). $$
\end{lemma}
\begin{proof}
  In a small neighbourhood $U$ over ${\bOh}\in X$ we can use a Darboux chart
  to define a smooth projection map $\pi : U\to T_{\bOh}X$, such that
  $\pi^*\omega_0 = \omega$ and the derivative of $\pi$ is the identity
  map at ${\bOh}$. It follows that the complex structures and holomorphic
  volume forms satisfy $|\pi^*J_0 - J|, |\pi^*\Omega_0 - \Omega| \leq
  C |\mathbf{x}|$ for a constant $C > 0$. Similarly $|\pi^*z - z|,
  |\pi^*w - w| \leq C|\mathbf{x}|$, and the derivatives of $z, w$
  satisfy the same bounds. The required estimate on the neighbourhood
  $U$  then follows by
  applying Lemma~\ref{lem:nablazw} to the image $\pi(P) \subset
  T_{\bOh}X$. The estimate is clear outside of $U$ since $|\mathbf{x}|$ is
  bounded away from zero on $X\setminus U$.   
\end{proof}

We define the excess $\mathcal{A}(M_\tau)$ for a Lagrangian in
$e^{\tau/2}X\subset \mathbb{R}^N$ as in \eqref{eq:AMdefn}, and
Condition~\eqref{cond:ddag} is as before. Note that the condition of uniformly
bounded area ratios is automatic in the compact case using the monotonicity
formula, and the uniform bound for the Lagrangian angle is also
preserved by the maximum principle. Let us record here the following consequence of the
monotonicity formula, analogous to \eqref{eq:Achange}, recalling that
$|\nabla\theta| = |\mathbf{H} + \nu|$:
 \begin{align*} \mathcal{A}(M_{\tau_1}) - \mathcal{A}(M_{\tau_2}) &\geq
  \int_{\tau_1}^{\tau_2} \int_{M_\tau} \left( 2|\mathbf{H} + \nu|^2 +
    \left|\mathbf{H} + \frac{\mathbf{x}^\perp}{2} +
      \frac{\nu}{2}\right|^2\right)
  e^{-|\mathbf{x}|/4}\,d\mathcal{H}^2\,  d\tau\\
  &\quad - C \int_{\tau_1}^{\tau_2} \int_{M_\tau} |\nu|^2
  e^{-|\mathbf{x}|^2/4}\, d\mathcal{H}^2\, d\tau\\
  &\geq \frac{1}{16} \int_{\tau_1}^{\tau_2} \int_{M_\tau} (
  |\mathbf{H}|^2 + |\mathbf{x}^\perp|^2)\, d\mathcal{H}^2\, d\tau - C
  e^{-\tau_1}, 
\end{align*}
for $\tau_1 < \tau_2$, where we also used the uniform bound on the
Lagrangian angle. 

We define $I_V(M_\tau)$ and $D_V(M_\tau)$ according to
Definition~\ref{dfn:IVDV}. Note that now the function $|\mathbf{x}|
d_V$ on $M_\tau$ is no longer uniformly equivalent to $|zw|$. In fact
if we denote the orthogonal projection of $\mathbf{x}$ onto $T_{\bOh}X$ by
$\tilde{\mathbf{x}}$, then $|zw|(\mathbf{x})$ is uniformly equivalent to
$|\tilde{\mathbf{x}}| d_V(\tilde{\mathbf{x}})$. At the same time,
since on $X$ we have $|\mathbf{x} - \tilde{\mathbf{x}}| \leq 
C|\mathbf{x}|^2$, it follows that on $M_\tau\subset e^{\tau/2}X$ we
have
$$ |\mathbf{x} - \tilde{\mathbf{x}}| \leq C e^{-\tau/2} |\mathbf{x}|^2. $$
It follows from this, together with the bounds on the area ratios of $M_\tau$, that
$$ C^{-1} \int_{M_\tau} |zw|^2 e^{-|\mathbf{x}|^2/4} \leq
  \int_{M_\tau} |\mathbf{x}|^2 d_V^2\, e^{-|\mathbf{x}|^2/4} \leq
  C\left(\int_{M_\tau} |zw|^2 e^{-\mathbf{x}^2/4} +
    e^{-\tau/2}\right). $$
In particular as long as $I_V(M_\tau) \geq e^{-\tau/2}$, we have that
$I_V(M_\tau)$ is uniformly equivalent to the Gaussian $L^2$-norm of
$|zw| + |\theta - \theta_V|$. We have the following. 

\begin{lemma}
  Suppose that $V\in \mathcal{V}$ and $I_V(M_\tau) \geq
  e^{-\tau/2}$. Then the conclusions (1), (2), (3) of
  Lemma~\ref{lem:IVest} hold, with $K_\gamma\subset
  \mathbb{R}^N\setminus \{\bOh\}$ in (3).
\end{lemma}
\begin{proof}
  In the proof of Lemma~\ref{lem:IVest} we essentially applied the monotonicity
  formula to the function $f = |zw| + |\theta-\theta_V|$. Here we
  claim that the same argument works if instead we use the function
  $$ f = |zw| + |\theta - \theta_V| + e^{-\tau/2} |\mathbf{x}| +
    e^{-\tau}. $$
  Here we are thinking of $\tau$ as being fixed. Consider the solution
  of the mean curvature flow with forcing term as above with $L_{-1} =
  M_\tau$. Then using Lemma~\ref{lem:nablazw2}, and the bound
  $|\nu|\leq Ce^{-\tau/2}$ for the forcing term, implies that
  $$ (\partial_t - \Delta) |zw| \leq C(|\theta - \theta_V| + e^{-\tau/2}|\mathbf{x}|)$$
  in the distributional sense. We also have
   \begin{align*} (\partial_t - \Delta)
      |\theta - \theta_V| &\leq 0, \\
      (\partial_t - \Delta) |\mathbf{x}| &\leq |\nu| \leq C
      e^{-\tau/2},
    \end{align*}
  and combining these we find that $(\partial_t - \Delta) f \leq Cf$
  for a constant $C > 0$. At the same time if $I_V(M_\tau) \geq
  e^{-\tau/2}$, then $I_V(M_\tau)$ is uniformly equivalent to the
  Gaussian $L^2$-norm of $f$. Using this the arguments in the proof of
  Lemma~\ref{lem:IVest} can be followed verbatim. 
\end{proof}

Note that if $D_V(M_\tau) = d \geq e^{-\tau/2}$ and $\tau$ is
large, then the region $|\mathbf{x}| < R_d$ (with $R_d$
defined in \eqref{eq:Rddefn}), where we expect good graphicality,
is much smaller than the ball over which
$e^{\tau/2}X$ is graphical over the tangent space $T_{\bOh}X$. More
precisely, we have the following under a stronger assumption for
$D_V(M_\tau)$.
\begin{lemma}\label{lem:graph30}
  Suppose that $D_V(M_\tau) = d\geq e^{-\tau/20}$ and $\tau$ is
  sufficiently large. Then on the region $|\mathbf{x}| < R_d$, we can view
  $e^{\tau/2}X$ as the graph of $v$ over $T_{\bOh}X$, where
\begin{equation}\label{eq:Xgraphest}
  |v|, |\nabla v| \leq C e^{-\tau/2} R_d = C e^{-\tau/2} |\ln d| \leq Cd^9.
\end{equation}
\end{lemma}
\begin{proof}
  Note that in a neighbourhood of $\bOh$, $X$ is the graph of a function
  $V$ over $T_{\bOh}X$ with $|V|\leq C|\mathbf{x}|^2$ and $|\nabla V| \leq
  C|\mathbf{x}|$. Rescaling this, and using that $R_d \ll e^{\tau/2}$
  for large $\tau$, we find that on the ball $|\mathbf{x}| < R_d$,
  $e^{\tau/2}X$ is the graph of $v$ over $T_{\bOh}X$ such that $|v|, |\nabla v| \leq C
e^{-\tau/2} R_d^2$. The required estimate follows from this. 
\end{proof}

Using this, the results following Lemma~\ref{lem:IVest} hold in the present
context, with minor adjustments of the proofs,
as long as we always ensure that $D_V(M_\tau) \geq
e^{-\tau/20}$, and that $\tau$ is sufficiently
large. In Condition \eqref{cond:ast}, to make sense of the condition
$\int_\gamma\lambda =0$ for loops $\gamma\in M_\tau\cap B_1$,
we use a Darboux chart for $\omega$ on $B_1\cap e^{\tau/2}X$ to define the
Liouville form $\lambda$. Assuming $\tau$
is large, such a chart exists as in the proof of
Lemma~\ref{lem:nablazw2} above. Note that in a Darboux chart the
holonomy of $\nabla^\xi$ around a loop $\gamma$ is given by
$$ e^{-2\pi i \int_D \omega} = e^{-2\pi i \int_\gamma \lambda}, $$
where $\gamma = \partial D$. It follows that the rationality condition
in Definition~\ref{defn:rationalL} coincides with the rationality
condition in Definition~\ref{dfn:exact} restricted to loops contained
in the chart.

Proposition~\ref{prop:d30} takes the following form.
\begin{prop}\label{prop:dcompact}
  Let $\delta_2 > 0$. There is an $\epsilon_2 > 0$, depending on
  $\delta_2$ such that if the flow
  $M_\tau$ satisfies Condition \eqref{cond:ddag} for $\tau\in [\tau_0-1, \tau_0
  + T + 10]$ with $T > 0$, $M_{\tau_0}$ satisfies Condition \eqref{cond:ast},
  $\tau_0$  is sufficiently large (depending
  on $\delta_2$), and
  \begin{enumerate}
    \item $\mathcal{A}_\alpha(M_{\tau_0}) -
      \mathcal{A}_\alpha(M_{\tau_0+T}) < \epsilon_2$,
    \item $D_{V_0}(M_{\tau_0}) < \epsilon_2$,
    \end{enumerate}
    then $D_{V_0}(M_{\tau_0+T}) < \delta_2$. 
\end{prop}

Using this, the proofs of 
Theorems~\ref{thm:intro1} and \ref{thm:intro2} from the Introduction
follow the same arguments as the corresponding results in
Section~\ref{sec:neckpinch}. Let us finally consider
Theorem~\ref{thm:intro3}.

\begin{proof}[Proof of Theorem~\ref{thm:intro3}]
  Using Theorem~\ref{thm:Lawlor}, which also holds in the current
  setting, we know that we have a sequence $(\mathbf{x}_k, t_k) \to
  (\bOh,0)$ and $r_k\to 0$ such that the rescalings $\tilde{L}_k = r_k^{-1}(L_{t_k} -
  \mathbf{x}_k)$ converge, smoothly on compact sets,  to a Lawlor neck
  $\tilde{L}_\infty$ asymptotic to  
  $V_0$ at infinity. Here $V_0$ is the (unique) tangent flow at the
  singularity $(\bOh,0)$. Note that by Condition \eqref{cond:ast}, this Lawlor
  neck is exact, and so up to scale there are only two
  possibilities. Changing the scales $r_k$ if necessary,
  we can assume that in terms of suitable coordinates $z,w$
  for which $V_0 = \{zw=0\}$, we have $\tilde{L}_\infty = \{zw =
 \pm1% e^{(\psi \pm \pi/2)i} 
  \}$. %, for $\psi$ depending on the choice of $z,w$.
  For sufficiently large $k$, we can remove a large ball $B_R$ from
  $\tilde{L}_k$, and replace it with a small Lagrangian perturbation of the two
  planes given by $V_0$.    Hence $\tilde{L}_k$ is an immersed Lagrangian where at the self-intersection point at the origin two sheets of the Lagrangian are intersecting transversely.
%  Since topologically $\tilde{L}_k$ is a  sphere, the result is a union $\tilde{M}_{1,k} \cup \tilde{M}_{2,k}$ of two  Lagrangian spheres intersecting transversely at the origin. 
  This is
  exactly the reverse of the graded connected sum construction in
  \cite[Section 3.1]{TY02} (potentially taking the Lagrangians to be two sheets of the same Lagrangian in the connected sum). Hence $\tilde{L}_k$ is a graded self-connected sum and so is $L_{t_k}$ as desired.
  
Suppose from now on that $\tilde{L}_k$ is not connected. Then we can write  $\tilde{L}_k = \tilde{M}_{1,k}\#
  \tilde{M}_{2,k}$ as a graded connected sum. The choice of which component is $\tilde{M}_{1,k}$, respectively 
  $\tilde{M}_{2,k}$, depends on which of the two possible Lawlor necks
  $\tilde{L}_\infty$ is. Note that the Lagrangian angles of $\tilde{M}_{1,k}, \tilde{M}_{2,k}$
  approach the constant $\theta_{V_0}$ on the ball $B_R$ as we let $R,
  k\to\infty$ in the construction.

  The upshot of this discussion is that after scaling back to the
  original flow, we can write $L_{t_k} = M_{1,k} \# M_{2,k}$ as claimed.  It now remains for us to show that \eqref{eq:stab1} holds and, if $L$ is almost calibrated, that \eqref{eq:stab2} holds.   Unless
  the initial Lagrangian $L$ is special Lagrangian, in which case no
  singularity would form, we will have:
  \begin{equation}\label{eq:Mik1}
    \begin{aligned} \inf_L \theta &< \inf_{M_{1,k}} \theta,
      \inf_{M_{2,k}}\theta ;\\
      \sup_L \theta &> \sup_{M_{1,k}} \theta,
      \sup_{M_{2,k}}\theta; \\
      \mathrm{vol}(L) &> \mathrm{vol}(M_{1,k}) + \mathrm{vol}(M_{2,k}),
    \end{aligned} \end{equation}
  as long as $k$ is sufficiently large. It follows from the last
  inequality that
   $$ \mathrm{vol}(L) > \left|\int_{M_{1,k}} \Omega\right| + \left|
    \int_{M_{2,k}} \Omega \right|, $$
  for sufficiently large $k$.  We deduce that \eqref{eq:stab1} holds.

  Recall that in the almost
  calibrated case $\phi(M_{i,k})$ is uniquely defined by
  \begin{equation} \label{eq:phiMdefn}
    \phi(M_{i,k}) = \mathrm{arg}\, \int_{M_{i,k}} \Omega =
    \mathrm{arg}\, \int_{M_{i,k}} e^{i\theta}\, d\mathcal{H}^2.
  \end{equation}
 Using \eqref{eq:Mik1}, this implies that
  $$ \phi(M_{i,k}) \in (\inf_L \theta, \sup_L \theta), $$
  for sufficiently large $k$.   We thus conclude that \eqref{eq:stab2} holds in the almost calibrated case.
\end{proof}

\bibliography{ancient}

\providecommand{\bysame}{\leavevmode\hbox to3em{\hrulefill}\thinspace}
\providecommand{\MR}{\relax\ifhmode\unskip\space\fi MR }
\providecommand{\MRhref}[2]{%
  \href{http://www.ams.org/mathscinet-getitem?mr=#1}{#2}
}
\providecommand{\href}[2]{#2}
\begin{thebibliography}{10}

\bibitem{BegleyMoore}
T.~Begley and K.~Moore, \emph{On short time existence of {L}agrangian mean
  curvature flow}, Math. Ann. \textbf{367} (2017), no.~3-4, 1473--1515.
  \MR{3623231}

\bibitem{Bogachev}
V.~I. Bogachev, \emph{Gaussian measures}, Mathematical Surveys and Monographs,
  vol.~62, American Mathematical Society, Providence, RI, 1998. \MR{1642391}

\bibitem{CS21}
O.~Chodosh and F.~Schulze, \emph{Uniqueness of asymptotically conical tangent
  flows}, Duke Math. J. \textbf{170} (2021), no.~16, 3601--3657. \MR{4332673}

\bibitem{ColdingMinicozzi:generic}
T.~H. Colding and W.~P. Minicozzi, \emph{Generic mean curvature flow {I}:
  generic singularities}, Ann. of Math. (2) \textbf{175} (2012), no.~2,
  755--833. \MR{2993752}

\bibitem{CM:uniqueness}
\bysame, \emph{{Uniqueness of blowups and \L{}ojasiewicz inequalities}}, Ann.
  of Math. (2) \textbf{182} (2015), no.~1, 221--285. \MR{3374960}

\bibitem{ColdingMinicozzi:frequency}
\bysame, \emph{Parabolic frequency on manifolds}, Int. Math. Res. Not. (2021),
  to appear.

\bibitem{Ecker.logSobolev}
K.~Ecker, \emph{Logarithmic {S}obolev inequalities on submanifolds of
  {E}uclidean space}, J. Reine Angew. Math. \textbf{522} (2000), 105--118.
  \MR{1758578}

\bibitem{Edelen21}
N.~Edelen, \emph{Degeneration of 7-dimensional minimal hypersurfaces which are
  stable or have bounded index}, arXiv:2103.13563.

\bibitem{Fukaya03}
K.~Fukaya, \emph{Galois symmetry on {F}loer cohomology}, Turkish J. Math.
  \textbf{27} (2003), no.~1, 11--32. \MR{1975330}

\bibitem{Huisken}
G.~Huisken, \emph{Asymptotic behavior for singularities of the mean curvature
  flow}, J. Differential Geom. \textbf{31} (1990), no.~1, 285--299.
  \MR{1030675}

\bibitem{IlmanenNevesSchulze}
T.~Ilmanen, A.~Neves, and F.~Schulze, \emph{On short time existence for the
  planar network flow}, J. Differential Geom. \textbf{111} (2019), no.~1,
  39--89. \MR{3909904}

\bibitem{Joyce15}
D.~Joyce, \emph{Conjectures on {B}ridgeland stability for {F}ukaya categories
  of {C}alabi-{Y}au manifolds, special {L}agrangians, and {L}agrangian mean
  curvature flow}, EMS Surv. Math. Sci. \textbf{2} (2015), no.~1, 1--62.
  \MR{3354954}

\bibitem{LambertLotaySchulze}
B.~Lambert, J.~D. Lotay, and F.~Schulze, \emph{Ancient solutions in
  {L}agrangian mean curvature flow}, Ann. Sc. Norm. Super. Pisa Cl. Sci.
  \textbf{XXII} (2021), 1169--1205. \MR{4334316}

\bibitem{LSS22}
J.~D. Lotay, F.~Schulze, and G.~Sz\'ekelyhidi, \emph{Ancient solutions and
  translators of {L}agrangian mean curvature flow}, arXiv:2204.13836.

\bibitem{Neves:zero-maslov}
A.~Neves, \emph{Singularities of {L}agrangian mean curvature flow:
  zero-{M}aslov class case}, Invent. Math. \textbf{168} (2007), no.~3,
  449--484. \MR{2299559}

\bibitem{Neves:survey}
\bysame, \emph{Recent progress on singularities of {L}agrangian mean curvature
  flow}, Surveys in geometric analysis and relativity, Adv. Lect. Math. (ALM),
  vol.~20, Int. Press, Somerville, MA, 2011, pp.~413--438. \MR{2906935}

\bibitem{Neves:singularities}
\bysame, \emph{Finite time singularities for {L}agrangian mean curvature flow},
  Ann. of Math. (2) \textbf{177} (2013), no.~3, 1029--1076. \MR{3034293}

\bibitem{NevesTian:translating}
A.~Neves and G.~Tian, \emph{Translating solutions to {L}agrangian mean
  curvature flow}, Trans. Amer. Math. Soc. \textbf{365} (2013), no.~11,
  5655--5680. \MR{3091260}

\bibitem{SW01}
R.~Schoen and J.~Wolfson, \emph{Minimizing area among {L}agrangian surfaces:
  the mapping problem}, J. Differential Geom. \textbf{58} (2001), no.~1, 1--86.
  \MR{1895348}

\bibitem{Schulze14}
F.~Schulze, \emph{Uniqueness of compact tangent flows in mean curvature flow},
  J. Reine Angew. Math. \textbf{690} (2014), 163--172. \MR{3200339}

\bibitem{Seidel99}
P.~Seidel, \emph{Lagrangian two-spheres can be symplectically knotted}, J.
  Differential Geom. \textbf{52} (1999), no.~1, 145--171. \MR{1743463}

\bibitem{Simon.asympt}
L.~Simon, \emph{Asymptotics for a class of nonlinear evolution equations, with
  applications to geometric problems}, Ann. of Math. (2) \textbf{118} (1983),
  no.~3, 525--571. \MR{727703}

\bibitem{Simon.lectures}
\bysame, \emph{Isolated singularities for extrema of geometric variational
  problems}, Miniconference on nonlinear analysis ({C}anberra, 1984), Proc.
  Centre Math. Anal. Austral. Nat. Univ., vol.~8, Austral. Nat. Univ.,
  Canberra, 1984, pp.~46--50. \MR{799211}

\bibitem{Smo96}
K.~Smoczyk, \emph{A canonical way to deform a lagrangian submanifold},
  arXiv:dg-ga/9605005.

\bibitem{SYZ96}
A.~Strominger, S.-T. Yau, and E.~Zaslow, \emph{Mirror symmetry is
  {$T$}-duality}, Nuclear Phys. B \textbf{479} (1996), no.~1-2, 243--259.
  \MR{1429831}

\bibitem{Sz20}
G.~Sz\'ekelyhidi, \emph{Uniqueness of certain cylindrical tangent cones},
  arXiv:2012.02065.

\bibitem{Thomas01}
R.~P. Thomas, \emph{Moment maps, monodromy and mirror manifolds}, Symplectic
  geometry and mirror symmetry ({S}eoul, 2000), World Sci. Publ., River Edge,
  NJ, 2001, pp.~467--498. \MR{1882337}

\bibitem{TY02}
R.~P. Thomas and S.-T. Yau, \emph{Special {L}agrangians, stable bundles and
  mean curvature flow}, Comm. Anal. Geom. \textbf{10} (2002), no.~5,
  1075--1113. \MR{1957663}

\bibitem{Wang04}
M.-T. Wang, \emph{The mean curvature flow smoothes {L}ipschitz submanifolds},
  Comm. Anal. Geom. \textbf{12} (2004), no.~3, 581--599. \MR{2128604}

\bibitem{White.regularity}
B.~White, \emph{A local regularity theorem for mean curvature flow}, Ann. of
  Math. (2) \textbf{161} (2005), no.~3, 1487--1519. \MR{2180405}

\end{thebibliography}
\bibliographystyle{amsplain}
\end{document}